\RequirePackage{fix-cm}
\documentclass[10pt,3p, times]{article} %T%T%T% \documentclass[smallextended]{svjour3}       % onecolumn (second format)
\usepackage{amsmath,amsxtra,amssymb,latexsym,amscd,amsthm,amsfonts,amstext,makeidx,upgreek,dsfont,xcolor}
\usepackage[margin=0.7in]{geometry}

\usepackage{mathtools}
\usepackage{cancel}
\usepackage{xcolor}
\usepackage{color}
\usepackage{url}
\usepackage{bigints}
\usepackage{mathrsfs}  
\usepackage{fouridx}
\usepackage{bbm}
\usepackage[colorlinks,backref]{hyperref}

\newcommand{\ip}[2]{\fourIdx{}{0}{}{\!x}{\mathcal{ X}}}

\newcommand\dela[1]{}
\usepackage{todonotes}

\newtheorem{theorem}{Theorem}[section]

\newtheorem{proposition}[theorem]{Proposition}

\newtheorem{corollary}[theorem]{Corollary}
\newtheorem{definition}[theorem]{Definition}

\newtheorem{remark}[theorem]{Remark}

\newcommand{\diver}{\mathrm{div}}
\newcommand{\Tr}{\mathrm{Tr}}

\newcommand{\R}{\mathbb{R}}

\newcommand{\Tb}{\mathbb{T}}
\newcommand{\Lb}{\mathbb{L}}
\newcommand{\Arm}{\mathrm{A}}
\newcommand{\Irm}{\mathrm{I}}
\newcommand{\Prm}{\mathrm{P}}
\newcommand{\Drm}{\mathrm{D}}
\newcommand{\Wrm}{\mathrm{W}}
\newcommand{\drm}{\mathrm{d}}
\newcommand{\Brm}{\mathrm{B}}
\newcommand{\Lrm}{\mathrm{L}}
\newcommand{\Frm}{\mathrm{F}}

\newcommand{\Jcal}{\mathcal{J}}
\newcommand{\Kcal}{\mathcal{K}}
\newcommand{\Pcal}{\mathcal{P}}
\newcommand{\A}{\mathcal{A}}
\newcommand{\Hb}{\mathbb{H}}
\newcommand{\Xb}{\mathbb{X}}
\newcommand{\Vb}{\mathbb{V}}
\newcommand{\Wb}{\mathbb{W}}
\newcommand{\N}{\mathbb{N}}
\newcommand{\ub}{\boldsymbol{u}}
\newcommand{\vb}{\boldsymbol{v}}
\newcommand{\wb}{\boldsymbol{w}}
\newcommand{\zb}{\boldsymbol{z}}
\newcommand{\kb}{\boldsymbol{k}}
\newcommand{\xb}{\boldsymbol{x}}
\newcommand{\z}{\boldsymbol{z}}
\newcommand{\yb}{\boldsymbol{y}}
\newcommand{\fb}{\boldsymbol{f}}
\newcommand{\gb}{\boldsymbol{g}}
\newcommand{\pb}{\mathfrak{p}}
\newcommand{\qb}{\mathfrak{q}}
\newcommand{\eb}{\boldsymbol{e}}
\newcommand{\mfrk}{\mathfrak{m}}

\newcommand{\vfrk}{\mathfrak{v}}
\newcommand{\wi}{\widetilde}
\newcommand{\Eb}{\mathbb{E}}
\newcommand{\Pb}{\mathbb{P}}

\allowdisplaybreaks

\begin{document}

\numberwithin{equation}{section}

\title{\bf Global-in-time optimal control of stochastic third-grade fluids with additive noise}

\author{Kush Kinra$^{1}$$^{*}$ and Fernanda Cipriano$^{2}$
\\ \\
\small $^1$ Center for Mathematics and Applications (NOVA Math), 
\\ \small NOVA School of Science and Technology (NOVA FCT),	Portugal.
\\
\small $^2$ Center for Mathematics and Applications (NOVA Math) and Department of Mathematics,
\\ \small NOVA School of Science and Technology (NOVA FCT),	Portugal.
 \\ \small Email: kushkinra@gmail.com, k.kinra@fct.unl.pt, cipriano@fct.unl.pt
\\ \small
\noindent \textsuperscript{*}Corresponding author:
Kush Kinra, Email: kushkinra@gmail.com, k.kinra@fct.unl.pt
  }
  \date{}
\maketitle

\begin{abstract}
In this article, we address the velocity tracking control problem for a class of stochastic non-Newtonian fluids. More precisely, we consider the stochastic third-grade fluid equation perturbed by infinite-dimensional additive white noise and defined on the two-dimensional torus $\mathbb{T}^2$. The control acts as a distributed random external force. Taking an \emph{infinite-dimensional Ornstein-Uhlenbeck process},  the stochastic system is converted into an equivalent pathwise deterministic one, which allows to show the well-posedness of the original stochastic system globally in time. The state being a stochastic process with sample paths in $\mathrm{L}^\infty(0,T;\mathbb{H}^3(\mathbb{T}^2))$ and finite moments can be controlled in an optimal way. Namely, we establish the existence and uniqueness of solutions to the corresponding linearized state and adjoint equations. Furthermore, we derive an appropriate stability result for the state equation and verify that the G\^ateaux derivative of the control-to-state mapping coincides with the solution of the linearized state equation. Finally, we establish the first-order optimality conditions and prove the existence of an optimal solution.
\end{abstract}

\noindent {\bf Keywords:} Stochastic third-grade fluids; optimal control; necessary optimality condition; infinite-dimensional Wiener process.

\noindent \textbf{Mathematics Subject Classification:} 35R60, 49K20, 76A05, 76D55, 60H15.

\section{Introduction}
The optimal control of the evolution of fluid flows is a significant mathematical challenge with great  implications in real-world applications. This is particularly true for non-Newtonian fluids of differential type, which exhibit complex rheological behaviors that depend on the rate of deformation.  Unlike Newtonian fluids, which satisfy the linear Newtonian viscosity law, these fluids display nonlinear stress-strain relationships, leading to unique phenomena such as shear-thinning, shear-thickening, and viscoelastic effects  (see, for example, \cite{FR80} and references therein). These properties are critical in various industrial and technological processes, including the manufacturing of polymers, the processing of food products,  biomedical applications such as blood flow modeling and many others. The ability to control and optimize these flows is essential for improving efficiency, ensuring process stability, and enhancing product quality.  However, due to the strong nonlinearity and complexity of their governing equations, the mathematical treatment of their optimal control is not an easy issue, requiring involving analysis and sophisticated mathematical tools.
Non-Newtonian fluids of differential type constitute a hierarchy  of complex fluids, the rheological relation of fluids of grade  $n$, $n>1$, is  nonlinear, and its constitutive law reads
\begin{align*}
	\mathbb{S}=-p\Irm+\Frm(\Arm_1(\vb),\cdots,\Arm_n(\vb)), \quad 
\end{align*}
where $\mathbb{S}$ is the Cauchy stress
tensor and  $\Frm$ is an isotropic polynomial function of degree $n$ subject to the usual requirement of material frame indifference,  $\vb$ is  fluid's velocity field and  $\Arm_n, n\geq 1$, are the Rivlin-Ericksen kinematic tensors (see \cite{RE55}), which are defined by 
\begin{align*}
	\left\{\begin{array}{ll}
		\Arm_1(\vb)&=\nabla \vb+\nabla \vb^T,%%%deformation tensor
		\\ \Arm_n(\vb)&=\dfrac{\drm}{\drm t}\Arm_{n-1}(\vb)+\Arm_{n-1}(\vb)(\nabla \vb)+(\nabla \vb)^T\Arm_{n-1}(\vb), \quad n=2,3,\cdots,
	\end{array}
	\right.
\end{align*}
where $\dfrac{\drm}{\drm t}=\dfrac{\partial}{\partial t}+(\vb\cdot \nabla)$ which is also known as material derivative. 

  This work is devoted to third-grade fluids, which are characterized by the following constitutive law 	
\begin{align*}
	\mathbb{S}=-p\Irm+\nu \Arm_1(\vb)+\alpha_1\Arm_2(\vb)+\alpha_2\Arm_1^2(\vb)+\beta_1 \Arm_3(\vb)+\beta_2(\Arm_1(\vb)\Arm_2(\vb)+\Arm_2(\vb)\Arm_1(\vb))+\beta_3 \Tr(\Arm_1^2(\vb))\Arm_1(\vb),
\end{align*}
where 	$\nu$ is the viscosity, and $\{\alpha_i\}_{i=1,2}$ and $\{\beta_i\}_{i=1,2,3}$ are material moduli.  The momentum equations are given by
$$\dfrac{\mathrm{d}\vb}{\mathrm{d}t}=\dfrac{\partial\vb}{\partial t}+(\vb\cdot \nabla) \vb = \diver(\mathbb{S}).$$ 
The  compatibility  with thermodynamic laws imposes the following restriction on the parameters (see \cite{FR80})
\begin{equation}\label{third-grade-paremeters}
	\nu \geq 0, \quad \alpha_1\geq 0, \quad |\alpha_1+\alpha_2 |\leq \sqrt{24\nu\beta}, \quad \beta_1=\beta_2=0, \, \beta_3=\beta \geq 0.
\end{equation}	
Let us notice that    $\beta=0$ corresponds to  the second grade fluids, and $\alpha_1=\alpha_2=0$	and $\beta$=0 gives the 	Navier-Stokes equations. To better understand and describe the properties of various nanofluids, numerous simulation studies have been carried out using third-grade fluid models (see \cite{HUAal20,PP19,RHK18} and references therein).  Nanofluids are engineered colloidal suspensions of nanoparticles-typically made of metals, oxides, carbides, or carbon nanotubes-dispersed in a base fluid such as water, ethylene glycol, or oil.  These nanofluids exhibit higher thermal conductivity than the base fluid and hold significant promise for a wide range of technological applications, including heat transfer systems, microelectronics, fuel cells, pharmaceutical processes, hybrid-powered engines, engine cooling, and vehicle thermal management. 
Therefore, investigating the kinematic properties of third-grade fluids holds significant importance in various applications.
 The governing equations for third-grade fluids are expressed as:
\begin{equation}\label{third-grade-fluids-equations}
	\left\{\begin{aligned}
		\partial_t(\vb-\alpha_1\Delta \vb)-\nu \Delta \vb+(\vb\cdot \nabla)(\vb-\alpha_1\Delta \vb) &+\displaystyle\sum_{j=1}^2 [\vb-\alpha_1\Delta \vb]^j\nabla \vb^j-(\alpha_1+\alpha_2)\text{div}([\Arm(\vb)]^2)\\ -\beta \text{div}[\Tr(\Arm(\vb)[\Arm(\vb)]^T)\Arm(\vb)]
		& = -\nabla p + \fb,  \\[0.2cm]	\text{div}\;\vb& =0,	  
		\end{aligned}\right.
\end{equation}
where $\vb(x,t):\mathbb{T}^2\times [0,\infty) \to\R^2$, $p(x,t):\mathbb{T}^2\times [0,\infty) \to\R$ and $\fb(x,t):\mathbb{T}^2\times [0,\infty) \to\R^2$ represent velocity field, pressure field and external forcing, respectively, at time $t$ and position $x$,  $[\vb-\alpha_1\Delta \vb]^j$ and $\vb^j$ denotes the $j^{th}$ component of these vectors, $\Arm(\vb) := \Arm_1(\vb),$ $\nu>0$ denotes the viscosity of the  fluid  and  $\alpha_1>0,$ $\alpha_2\in\R$, $\beta>0$ are  material moduli verifying \eqref{third-grade-paremeters}.  In this article, we are interested in the stochastic version of the system \eqref{third-grade-fluids-equations} which is given by
\begin{equation}\label{stochastic-third-grade-fluids-equations}
	\left\{\begin{aligned}
		\drm(\vb-\alpha_1\Delta \vb)+\bigg\{-\nu \Delta \vb+(\vb\cdot \nabla)(\vb-\alpha_1\Delta \vb) &+\displaystyle\sum_{j=1}^2 [\vb-\alpha_1\Delta \vb]^j\nabla \vb^j-(\alpha_1+\alpha_2)\text{div}([\Arm(\vb)]^2)\\ -\beta \text{div}[\Tr(\Arm(\vb)[\Arm(\vb)]^T)\Arm(\vb)]\bigg\}\drm t
		& =\left( -\nabla p + \fb \right)\drm t+ \drm\Wrm, && \hspace{-10mm} \text{in }  \mathbb{T}^2 \times (0,\infty),  \\
        \text{div}\;\vb& =0,&& \hspace{-10mm} \text{in }  \mathbb{T}^2 \times [0,\infty),	  \\
      \vb(x,0)&=\vb_0(x) \quad && \hspace{-10mm} \text{in } \mathbb{T}^2,
		\end{aligned}\right.
\end{equation}
where $\Wrm$ is a $GG^{\ast}$-Wiener process defined in some given probability space $(\Omega,\mathscr{F},\mathbb{P})$,  $G$ being a Hilbert-Schmidt  operator from $\mathbb{U}$ to $\Hb$. Here $\mathbb{U}$ denotes some Hilbert space
and $\Hb$ is  the space of functions
in 
 $\Lb^2(\Tb^2)$  which are mean-zero and divergence-free. We recall that $\Wrm$ lives in a bigger space $\widetilde{\mathbb{U}}$ such 
 that the inclusion $i:\mathbb{U}\to \widetilde{\mathbb{U}}$ is an 
 Hilbert-Schmidt operator. For $\mathrm{L}>0$, we identify the  two-dimensional torus $\mathbb{T}^2=\left(\frac{\R}{\mathrm{L}\mathbb{Z}}\right)^2$, with the rectangle $[0,\mathrm{L}]^2$
  with periodic boundary conditions (see \cite[Chapter 3]{Grafakos_2014}), namely $\vb(\cdot,\cdot)$, $p(\cdot,\cdot)$ and $\fb(\cdot,\cdot)$ satisfy the following periodic conditions:
\begin{align}\label{2}
	\vb(x+\mathrm{L}e_{i},\cdot) = \vb(x,\cdot), \ {p}(x+\mathrm{L}e_{i},\cdot) = {p}(x,\cdot) \ \text{and} \ \fb(x+\mathrm{L}e_{i},\cdot) = \fb(x,\cdot), \quad x\in\R^{2},
	\quad i\in\{1,2\},
\end{align}
where $\{e_{1},e_{2}\}$ is the canonical basis of $\R^{2}.$

 The main objective of this article is to investigate an optimal  control problem associated with the stochastic state equation \eqref{stochastic-third-grade-fluids-equations} subject to the periodic boundary conditions \eqref{2}. The resolution of the control problem involves three key steps:
\begin{itemize}
\item [i)] Global-in-time solvability of system \eqref{stochastic-third-grade-fluids-equations} with an appropriate regularity to ensure the well-posedness of linearized state equation.
    \item [ii)] Analyzing the G\^ateaux derivative of the control-to-state mapping, which is characterized by the unique solution of the linearized state equation and addressing the well-posedness of the adjoint equation.
 \item [iii)] Deriving the connection between the solution of the linearized equation and the adjoint state solution, which enables the formulation of the necessary first-order optimality condition in terms of the adjoint state.
\end{itemize}

 Without exhaustiveness, the third-grade fluid equations with Dirichlet boundary condition were studied  in \cite{Amrouche+Cioranescu_1997,Sequeira+Videman_1995}. Later on in \cite{Busuioc+Iftimie_2004,Busuioc+Iftimie_2007}, the authors proved  the existence of a global solution with $\Hb^2$ initial data and Navier-slip boundary condition in 2D as well as 3D, and demonstrated that uniqueness holds in 2D.  The authors of \cite{Cipriano+Didier+Guerra_2021} extended  the later deterministic findings to 2D stochastic models.  In \cite{Tahraoui+Cipriano_2024}, for stochastic third-grade fluids equations perturbed by multiplicative white noise, the authors proved that  the local (in time) solution exists and is unique. Such solution corresponds to an adapted stochastic process 
with sample paths in  $\Hb^3$,  defined up to a specific positive stopping time.  We also mention some works discussing optimal control problem for stochastic Navier-Stokes equations, see \cite{Breckner_1999_thesis,Chemetov+Cipriano_2025,Cutland+Grzesiak_2007,DaPrato+Debussche_2000,Menaldi+Sritharan_2003,Sritharan_2000} etc. Compared to the Navier–Stokes equations, the primary challenge in third-grade fluid equations arises from the presence of third-order derivatives in the nonlinear term, which significantly complicates the analysis. An optimal control problem for stochastic second-grade fluids has been addressed in \cite{Chemetov+Cipriano_2017}.
In the framework of multiplicative noise, it is worth mentioning that the derivation of first-order optimality conditions for the Navier–Stokes equations, as well as for second-grade fluid equations, is possible only under suitable restrictions on the model parameters (see \cite{Breckner_1999_thesis,Chemetov+Cipriano_2017}).
The optimal control of a generic stochastic fluid flow  is not an easy issue.

This work aims to solve the optimal control problem for stochastic third-grade fluid equations \eqref{stochastic-third-grade-fluids-equations} 
 with additive noise, without additional  restrictions on the physical parameters, and over a prescribed time interval. 
In order to solve the stochastic optimal control problem for the random vector field $\vb$
governed  by  the stochastic  third-grade fluid equations \eqref{stochastic-third-grade-fluids-equations}, we first prove  that  $\vb$ satisfies the following key exponential integrability condition
    \begin{align}\label{vb-regularity}
   \Eb\bigg[\exp\bigg\{\hat{c}\int_0^T\|\vb(s)\|^{2}_{\dot{\Hb}^3} \drm s \bigg\}\bigg] <+\infty, \;\; \text{ for } \hat{c}>0.
\end{align}
To handle the additive white noise, we decompose the stochastic system \eqref{stochastic-third-grade-fluids-equations} into two parts: one  part is a linear stochastic differential equation involving the  white noise, while the other one  corresponds to  a random partial differential equation. Under suitable assumptions on the white noise $\Wrm$, we establish the aforementioned regularity of the solution to stochastic system \eqref{stochastic-third-grade-fluids-equations}, which plays a key role in proving the well-posedness of the linearized state equation and demonstrating that the G\^ateaux derivative of the cost functional coincides with the solution of the linearized state equation. 

To the best of our knowledge, this is the first work, which addresses the global-in-time optimal control problem for stochastic  third-grade fluid models. Recently, the authors in \cite{Tahraoui+Cipriano_2025_ESAIM} considered the  system \eqref{third-grade-fluids-equations} perturbed by a multiplicative noise and addressed the corresponding optimal control problem. However, due to the lack of global-in-time  $\Hb^3$-regularity of the sample paths, the  authors in \cite{Tahraoui+Cipriano_2025_ESAIM} was able to control the state just locally-in-time.   Namely,  the cost functional was defined locally by taking an  appropriate stopping time, and the stopping time plays  a major role at each step of the analysis. In contrast to \cite{Tahraoui+Cipriano_2025_ESAIM}, the present work  tackles the  global-in-time optimal control problem associated with system \eqref{stochastic-third-grade-fluids-equations}. We also mention the work \cite{Tahraoui+Cipriano_2023} where authors considered optimal control problem for deterministic third-grade fluid equations \eqref{third-grade-fluids-equations} supplemented by Navier-slip boundary condition.

The current paper's plan is as follows. We begin Section \ref{section-2} by introducing the appropriate functional spaces. Then, we define linear and nonlinear operators. In addition, we present a linear stochastic equation along with its solvability result, which enables us to transform the stochastic equation \eqref{stochastic-third-grade-fluids-equations} into an equivalent pathwise deterministic system (see \eqref{random-third-grade-fluids-equations} below).  We conclude the section by presenting the main result of the article.  In Section \ref{section-3}, we prove the existence and uniqueness of global-in-time solution of system \eqref{stochastic-third-grade-fluids-equations} with appropriate regularity which is required to investigate our main result. Section \ref{section-4} is devoted to the well-posedness of linearized state equation. In Section \ref{Sec-Gateaux-differentiability}, we address the G\^ateaux differentiability of the control-to-state mapping. Section \ref{Sec-adjoint} discusses the well-posedness of the adjoint system corresponding to the linearized state equation. The existence of a solution to the control problem is shown in the final section by establishing a duality property between the solutions of the linearized equation and the adjoint equation. Ultimately, by utilizing the duality relation, we further show that the solution to the control problem satisfies the first-order optimality conditions.

\section{Notations and auxiliary results}\label{section-2}

The goal of this section is to provide necessary notations, function spaces, auxiliary results and the main goal of this article.

\subsection{Function spaces}\label{N-FS}
Let us consider $m\in\mathbb{N}$ and  $p>0$. The usual norms for scalar (resp. vector) valued functions in the classical Lebesgue and Sobolev spaces
	${\mathrm{L}}^p(\mathbb{T}^2)$ (resp. 
	${\mathbb{L}}^p(\mathbb{T}^2)$), ${\mathrm{W}}^{m,p}(\mathbb{T}^2)$ (resp. ${\mathbb{W}}^{m,p}(\mathbb{T}^2)$) and ${\mathrm{H}}^{m}(\mathbb{T}^2)$ (resp. ${\mathbb{H}}^{m}(\mathbb{T}^2)$) will be denoted by $\|\cdot \|_p$, $\|\cdot\|_{{\mathrm{W}}^{m,p}}$ (resp. $\|\cdot\|_{{\mathbb{W}}^{m,p}}$) and 
	$\|\cdot\|_{{\mathrm{H}}^{m}}$ (resp. $\|\cdot\|_{{\mathbb{H}}^{m}}$), respectively. 
	We set $\|\cdot\|_{{\mathrm{W}}^{0,p}}=
	{\mathrm{L}}^p(\mathbb{T}^2)$ (resp. $\|\cdot\|_{{\mathbb{W}}^{0,p}}={\mathbb{L}}^p(\mathbb{T}^2)$).

    Let  $\dot{\mathrm{C}}^{\infty}(\mathbb{T}^2;\R^2)$ denote the space of all infinitely differentiable  functions defined on $\mathbb{T}^2$ with values in $\R^2$ such that $\int_{\mathbb{T}^2}\vb(x)\drm x=\mathbf{0}$. We notice that the functions in $\dot{\mathrm{C}}^{\infty}(\mathbb{T}^2;\R^2)$ satisfy the periodic boundary condition $\vb(x+\mathrm{L}e_i)=\vb(x)$. The Sobolev space  $\dot{\Hb}^s(\mathbb{T}^2)$, $s\in\mathbb{N}_0=:\mathbb{N}\cup\{0\}$, is defined as the completion of $\dot{\mathrm{C}}^{\infty}(\mathbb{T}^2;\R^2)$  with respect to the following Sobolev norm  
	$$\|\vb\|_{ {\Hb}^s}:=\left(\sum_{0\leq|\alpha|\leq s}\|\Drm^{\alpha}\vb\|_{\mathbb{L}^2(\mathbb{T}^2)}^2\right)^{1/2}.$$  According to Proposition 5.39, \cite{JCR}, we have 
	$$\dot{\Hb}^s (\mathbb{T}^2) = \left \{\vb:\vb=\sum_{\kb\in\mathbb{Z}^2}\vb_{\kb} e^{2\pi i \kb\cdot x / \mathrm{L}},\vb_0=\mathbf{0}, \ \bar{\vb}_{\kb}=\vb_{-\kb} , \ \|\vb\|^2_{\dot{\Hb}^s} := \sum_{\kb \in\mathbb{Z}^2}|\kb|^{2s}|\vb_{\kb}|^2<\infty\right\},$$
    and from Proposition 5.38, \cite{JCR}, we infer that 
	$\|\cdot\|_{ \dot{\Hb}^s}$ defines a norm on the space 
	$\|\cdot\|_{ \dot{\Hb}^s}$, which is equivalent to 
	the  induced standard Sobolev norm  $\|\cdot\|_{ {\Hb}^s}$
	on  $\dot{\Hb}^s(\mathbb{T}^2)$.  The zero mean condition provides the well-known \emph{Poincar\'{e} inequality}, 
	\begin{align}\label{poin}
		\lambda_1\int_{\mathbb{T}^2}|\vb(x)|^2\drm x \leq \int_{\mathbb{T}^2}|\nabla\vb(x)|^2\drm x, \text{ for any } \vb\in\dot{{\Hb}}^1(\mathbb{T}^2),
	\end{align} where $\lambda_1 = \frac{4\pi^2}{\mathrm{L}^2}$ (Lemma 5.40, \cite{JCR}). Let us define
	$\mathcal{V}:=\{\vb\in\dot{\mathrm{C}}^{\infty}(\mathbb{T}^2;\R^2):\nabla\cdot\vb=0\}.$ 
    We denote by $\Hb$ the closure of $\mathcal{V}$ in the Lebesgue spaces $\dot{\Lb}^2(\mathbb{T}^2)$; and by  $\Vb$  the closure of $\mathcal{V}$ in the Sobolev space $\dot{{\Hb}}^1(\mathbb{T}^2)$.

	The spaces $\Hb$ and $\Vb$ are Hilbert spaces with inner products defined by
	\begin{align}\label{ipH}
		(\ub,\vb)&:=\int_{\mathbb{T}^2}\ub(x)\cdot \vb(x)\drm x=\sum_{i=1}^2\int_{\mathbb{T}^2}\ub^i(x)\vb^i(x) \; \drm x, \; \text{ for all } \ub,\vb \in \Hb, \\
		(\ub,\vb)_{\Vb} &: = (\nabla \ub,\nabla \vb), \; \text{ for all } \ub,\vb \in \Vb. \label{ipV}
	\end{align} 
	We denote the corresponding norms by  $\|\cdot \|_{2}$ and $\|\cdot\|_{\Vb}$, respectively. 
	\bigskip	

	Let	us	introduce	the		Banach	space	$(\Xb:= {\mathbb{W}}^{1,4}(\mathbb{T}^2)\cap \Vb ,\Vert	\cdot\Vert_{\Xb})$	with	$\Vert	\cdot\Vert_\Xb := \Vert	\cdot\Vert_{{\mathbb{W}}^{1,4}}+\|\cdot\|_{\Vb}.$ We	recall	that	${\mathbb{W}}^{1,4}(\mathbb{T}^2)$	endowed	with the usual norm  $\Vert	\cdot\Vert_{ {\mathbb{W}}^{1,4}}$
	$$\Vert	\wb \Vert_{ {\mathbb{W}}^{1,4}}^4=\int_{\mathbb{T}^2}	\vert	\wb(x)\vert^4 \drm x + \int_{\mathbb{T}^2} \vert	\nabla	\wb(x)\vert^4 \drm x$$	
    is a Banach	space. We represent by  $\langle \cdot,\cdot\rangle $ the duality relation between the spaces $\Vb$  and its dual $\Vb'$ as well as $\Xb$  and its dual $\Xb'$.

	Next, let us introduce  the scalar product between two matrices $A:B=\Tr(AB^T)$ and denote $\vert A\vert^2:=A:A.$
	The divergence of a  matrix $A\in \mathcal{M}_{2\times 2}(E)$ is given by 
	$\big\{\text{div}(A)_i\big\}_{i=1}^{2}=\bigg\{\displaystyle\sum_{j=1}^2\partial_ja_{ij}\bigg\}_{i=1}^{2}. $ We recall that
	\begin{align*}
		(A,B)=\int_{\mathbb{T}^2} A(x) : B(x)\; \drm x ; \quad  \text{ for all } A,B \in \mathcal{M}_{2\times 2}(\Lb^2(\mathbb{T}^2)).
	\end{align*}

Throughout the article,  we denote by $C$   generic constant, which may vary from line to line. 

\subsection{Linear and nonlinear operators}\label{subsec-operator}
Let $\mathcal{P}: \dot{\Lb}^2(\mathbb{T}^2) \to \Hb$ be the Helmholtz-Hodge (or Leray) orthogonal projection  (cf.  \cite{JBPCK}). We define the Stokes operator
	\begin{equation*}
		\A\ub:=-\mathcal{P}\Delta\ub,\;\ub\in\Drm(\A),
	\end{equation*}
	where $\Drm(\A) = \big\{\ub \in\dot{\Hb}^{2}(\mathbb{T}^2):\nabla\cdot\ub=0\big\}$. It should be noted that $\mathcal{P}$ and $\Delta$ commutes in a torus, see \cite[Lemma 2.9]{RRS}.  
Therefore, in the sequel, for $\ub\in\Drm(\A)$, we will write 
\begin{align*}
    \Upsilon(\ub):= \ub+\alpha_1\A\ub = \ub-\alpha_1\Delta\ub. 
\end{align*}

    Given $\ub \in \Drm(\A)$ 
with Fourier expansion $\ub = \sum\limits_{\kb\in\mathbb{Z}^2} e^{2\pi i \kb\cdot x /  \mathrm{L}}\ub_{\kb},$ one obtains 
\begin{align*}
	-\Delta\ub=\frac{4\pi^2}{\mathrm{L}^2}\sum_{\kb\in\mathbb{Z}^2} e^{2\pi i \kb\cdot x /  \mathrm{L}}|\kb|^2\ub_{\kb}.
\end{align*}
Since $\A^{-1}$ is a compact self-adjoint operator in $\Hb$, there exists an  orthonormal basis in $\Hb$ of  eigenfunctions   $\{\wb_i\}_{i\in\mathbb{N}}\subset\dot{\mathrm{C}}^{\infty}(\mathbb{T}^2;\R^2)$ 
of the Stokes operator;
namely we have  $\A \wb_i=\lambda_i \wb_i$, for $i=1,2,\ldots,$ where  
the eigenvalues of $\A$ verify $0<\lambda_1\leq \lambda_2\leq \ldots\to\infty$. Note that $\lambda_1=\frac{4\pi^2}{\mathrm{L}^2}$ is the smallest eigenvalue of $\A$ appearing in the Poincar\'e inequality \eqref{poin}. 

The following relation holds
\begin{equation}\label{ipHV}
\lambda_i (\wb_i,\vb)=(\A \wb_i, \vb)=(-\Delta \wb_i, \vb)=
(\nabla \wb_i, \nabla \vb)=( \wb_i, \vb)_\Vb, \quad\text{	for all }  \vb\in\Vb.
\end{equation}
Hence $\{\wb_i\}_{i\in\mathbb{N}}$ is an orthogonal basis in $\Vb.$

We recall that the operators $\A^{\lambda}$, $\lambda\in\mathbb{R}$, 
are well defined and 
$$\Drm(\A^{s/2})=\big\{\ub\in \dot{\Hb}^{s}(\mathbb{T}^2):\nabla\cdot\ub=0\big\}.$$
In addition it can be shown that there  exists a positive constant $C$ such that
$\|\A^{s/2}\ub\|_{2}=C\|\ub\|_{\dot{\Hb}^{s}},$ for all $\ub\in\Drm(\A^{s/2})$, $s\geq 0$ (see \cite[Chapter 6]{JCR}). Note that the operator $\A$ is a non-negative self-adjoint operator in $\Hb$ with a compact resolvent and
\begin{align}\label{2.7a}
	|( \A\ub,\vb)| \leq C\|\ub\|_{\Vb}
	\|\vb\|_{\Vb},\ \textrm{ for all }\ \ub\in \Drm(\A),
	\; \vb \in \Vb, \text{ then }\ \|\A\ub\|_{\Vb'}\leq \|\ub\|_{\Vb}.
\end{align}
Therefore, there exists a unique extension of $\A$ denoted by the same symbol  verifying
\begin{align*}
	\A:{\Vb}\to {\Vb'}, \qquad \langle \A\ub,\vb\rangle:= (\nabla \ub,\nabla \vb), \qquad \forall \ub, \; \vb \in \Vb.
\end{align*}

Next, we define the {trilinear form} $b(\cdot,\cdot,\cdot):\Vb\times\Vb\times\Vb\to\R$ by $$b(\ub,\vb,\wb)=\int_{\mathbb{T}^2}(\ub(x)\cdot\nabla)\vb(x)\cdot\wb(x)\drm x=\sum_{i,j=1}^2\int_{\mathbb{T}^2}\ub^i(x)\frac{\partial \vb^j(x)}{\partial x_i}\wb^j(x)\drm x.$$ If $\ub, \vb$ are such that the linear map $b(\ub, \vb, \cdot) $ is continuous on $\Vb$, the corresponding element of $\Vb^{\prime}$ is denoted by $\Brm(\ub, \vb)$. We represent   $\Brm(\vb) = \Brm(\vb, \vb)=\mathcal{P}(\vb\cdot\nabla)\vb$. Using an integration by parts, it is immediate that 
\begin{equation}\label{b0}
	\left\{
	\begin{aligned}
		b(\ub,\vb,\vb) &= 0,\ \text{ for all }\ \ub,\vb \in\Vb,\\
		b(\ub,\vb,\wb) &=  -b(\ub,\wb,\vb),\ \text{ for all }\ \ub,\vb,\wb\in \Vb.
	\end{aligned}
	\right.
\end{equation} 
For  $\ub,\vb\in\Vb$, we have 
\begin{align}\label{213-B}
	|\langle \Brm(\ub)-\Brm(\vb),\wb\rangle| \leq | b(\ub,\ub-\vb,\wb)| + |b(\ub-\vb,\vb,\wb)|  \leq \big[\|\ub\|_{4} \|\nabla(\ub-\vb)\|_{2} + \|\ub-\vb\|_{4} \|\nabla\vb\|_{2}\big]\|\wb\|_{4}
\end{align}
for all $\ub,\vb,\wb\in\Vb$.  Thus the operator $\Brm(\cdot):\Vb\to\Vb^{\prime}$ is locally Lipschitz. 
In addition, the operator $\Brm$ enjoys the following important orthogonality property (see \cite[Lemma 3.1, p. 404]{temam2012infinite}): 
\begin{align}\label{Au-orthogonal}
   (\Brm(\vb, \vb), \A\vb)= b(\vb,\vb,\A\vb)=0, \; \text{for any } \vb\in\Drm(\A).
\end{align}

Let us now define the operator $\Jcal(\vb) := -\mathcal{P}\diver(\Arm(\vb)\Arm(\vb))$.  Note that for $\vb\in\Xb$, we have
\begin{align*}
	\|\Jcal(\vb)\|_{\Xb^{\prime}} \leq C \|\vb\|^2_{\Xb},
\end{align*}
and hence the map $\Jcal(\cdot):\Xb\to\Xb^{\prime}$. For  $\ub,\vb\in\Xb$, we have  
\begin{align}\label{213-J}
	|\langle \Jcal(\ub)-\Jcal(\vb),\wb\rangle| & = \left| \frac12 \int_{\mathbb{T}^2}[\Arm(\ub -\vb)\Arm(\ub)+\Arm(\vb)\Arm(\ub - \vb)]:\Arm(\wb)\drm x \right| \nonumber
	\\&\leq C\big[\|\Arm(\ub)\|_{4}+\|\Arm(\vb)\|_{4}\big] \|\Arm(\ub - \vb)\|_{4} \|\Arm(\wb)\|_{2},
\end{align}
for all $\ub,\vb,\wb\in\Xb$.  Thus the operator $\Jcal(\cdot):\Xb\to\Xb^{\prime}$ is locally Lipschitz. 
\iffalse In addition, due to divergence-free condition, we obtain $\Tr([\Arm(\vb)]^3)=0$, for any $\vb\in\Xb$. Therefore, we have 
	\begin{align}\label{Juu0}
		\langle\Jcal(\vb),\vb\rangle = \frac12 \int_{\mathbb{T}^2}\Tr([\Arm(\vb(x))]^3)\drm x =0, \; \text{for any } \vb\in\Xb.
	\end{align} 
\fi 

Finally, we define the operator $\Kcal(\vb):= -\mathcal{P}\diver(|\Arm(\vb)|^{2}\Arm(\vb))$. It is immediate that $$\langle\mathcal{K}(\vb),\vb\rangle =\frac12\|\Arm(\vb)\|_{4}^{4}.$$ Note that for $\vb\in\Xb$, we have 
\begin{align*}
	\|{\Kcal}(\vb)\|_{\Xb^{\prime}} \leq C \|\vb\|^3_{\Xb},
\end{align*}
and hence the map $\Kcal(\cdot):\Xb\to\Xb^{\prime}$. For  $\ub,\vb\in\Xb$,  we infer 
\begin{align}\label{213-K}
	|\langle \Kcal(\ub)-\Kcal(\vb),\wb\rangle
		|& = \bigg| \frac12 \int_{\mathbb{T}^2}[\Arm(\ub-\vb):\Arm(\ub)+\Arm(\vb):\Arm(\ub-\vb)]\Arm(\ub):\Arm(\wb)\drm x 
	\nonumber\\ & \quad + \frac12 \int_{\mathbb{T}^2}|\Arm(\vb)|^2\Arm(\ub-\vb):\Arm(\wb)\drm x \bigg| \nonumber
	\\&\leq C\big[\|\Arm(\ub)\|_{4}^2+\|\Arm(\vb)\|_{4}^2\big] \|\Arm(\ub-\vb)\|_{4} \|\Arm(\wb)\|_{4},
\end{align}
for all $\ub,\vb,\wb\in\Xb$.  Thus the operator $\mathcal{K}(\cdot):\Xb\to\Xb^{\prime}$ is locally Lipschitz.

\iffalse
Now, from \cite[Equation (2.13)]{Hamza+Paicu_2007}, we have 
\begin{align}\label{CL3}
	&  \beta\left<	\Kcal(\wb_1) - \Kcal(\wb_2), \wb_1 - \wb_2 \right>
	  = \frac{\beta}{2}\int_{\mathbb{T}^2}( |\Arm(\wb_1)|^2-|\Arm(\wb_2)|^2)^2 + \frac{\beta}{2}\int_{\mathbb{T}^2} |\Arm(\wb_1-\wb_2)|^2( |\Arm(\wb_1)|^2+|\Arm(\wb_2)|^2),
\end{align}
for all $\wb_1,\wb_2\in ........$

Let us complete this subsection by recalling the following continuous Sobolev embeddings from \cite[Theorem B]{Brezis+Mironescu_2019} (see also \cite[Theorem 2.20, p.44]{Aubin_1982}) which will be used frequently in the sequel: 
\begin{enumerate}
	\item We have
	\begin{align}\label{delta-embedding-infty}
		\mathbb{W}^{3,2}(\mathbb{T}^2) \hookrightarrow \mathbb{W}^{1,\infty}(\mathbb{T}^2), \;\;\; \text{	that is}, \;\;\;  \|\ub\|_{\mathbb{W}^{1,\infty}}\leq S \|\ub\|_{\mathbb{W}^{3,2}},
	\end{align}
	for $\ub\in \mathbb{W}^{3,2}(\mathbb{T}^2)$, where $S>0$ is a constant.
	\item For any given $1\leq p <\infty$, we have
	\begin{align}\label{p-embedding}
		\mathbb{W}^{2,2}(\mathbb{T}^2)\hookrightarrow \mathbb{W}^{1,p}(\mathbb{T}^2), \;\;\; \text{	that is}, \;\;\;  \|\ub\|_{\mathbb{W}^{1,p}}\leq S_{p} \|\ub\|_{\mathbb{W}^{2,2}},
	\end{align}
	for $\ub\in \mathbb{W}^{2,2}(\mathbb{T}^2)$, where $S_{p}>0$ is a constant.
\end{enumerate}
\fi

\subsection{Linear stochastic equation}
In order to establish global-in-time solution of system \eqref{stochastic-third-grade-fluids-equations}, we decompose system \eqref{stochastic-third-grade-fluids-equations} in  a linear equation which involves the noise and a random partial differential equation. Namely, we write $\vb=\ub+\zb$ such that $\zb$ solves the following stochastic linear system:
\begin{equation}\label{eqn_z}
	\left\{
	\begin{aligned}
		\drm(\Irm-\alpha_1\Delta)\zb  - \theta \Delta (\Irm-\alpha_1\Delta)\zb \drm t +(\Irm-\alpha_1\Delta)\zb \drm t &=\drm \mathrm{W},\\
		%\mathrm{div}\;\zb&=0,\\	
        \z(0)&=\boldsymbol{0},
	\end{aligned}
	\right.
\end{equation}
 where $\theta>0$ is a given constant, $\Wrm$ is $\R^2$-valued two-sided trace-class $GG^{\ast}$-Wiener process with spatial mean zero and divergence-free, and $\ub$ solves the following non-linear random system:
 \begin{equation}\label{random-third-grade-fluids-equations}
	\left\{\begin{aligned}
		\partial_t(\ub-\alpha_1\Delta\ub)  -\nu \Delta \ub  +((\ub+\zb)\cdot \nabla)(\ub+\zb-\alpha_1\Delta\ub & -\alpha_1\Delta\zb) +\displaystyle\sum_{j=1}^2[\ub+\zb-\alpha_1\Delta\ub-\alpha_1\Delta\zb]^j\nabla [\ub+\zb]^j  
        \\ -(\alpha_1+\alpha_2)\diver[\Arm(\ub+\zb)^2] -\beta \diver[|\Arm(\ub+\zb)|^2\Arm(\ub+\zb)]
&= -\nabla {p} + \fb + \zb +  (\nu - \theta-\alpha_1)\Delta\zb + \theta\alpha_1 \Delta^2\z, \\ & \quad  && \hspace{-30mm} \text{in }  \mathbb{T}^2 \times (0,\infty),  \\
        \text{div}\;\ub& =0,&& \hspace{-30mm} \text{in }  \mathbb{T}^2 \times [0,\infty),	  \\
      \ub(x,0)&=\vb_0(x), \quad && \hspace{-30mm} \text{in } \mathbb{T}^2,
		\end{aligned}\right.
\end{equation}
with periodic boundary conditions. 
\begin{remark}
    Here, $\zb$ is divergence-free due to the assumptions on the noise $\Wrm$, so there is no need to explicitly include the condition $\mathrm{div}\;\zb=0$ in \eqref{eqn_z}. Therefore, we denote by $p$ the pressure term associated with $\ub$.
\end{remark}

In view of the factorization method, one can obtain the regularity of $\zb$ on a given stochastic basis $(\Omega,\mathcal{F},\{\mathcal{F}_{t}\}_{t\in [0,T]},\mathbb{P})$ with $\{\mathcal{F}_{t}\}_{t\in[0,T]}$ being the usual filtration. Particularly, we prove the following result for regularity of $\zb$ using the similar arguments as in \cite[Theorem 5.16]{DaZ}. 

\begin{proposition}\label{thm_z_alpha}
Suppose that $\Tr\left((-\Delta)^{3-2\gamma} GG^{\ast}\right)<\infty$ for $\gamma\in(0,\frac12)$. Then for $p\geq2$
	\begin{align}\label{Z-regularity}
\mathbb{E}\left[\sup_{t\in[0,T]}\|\z(t)\|^{p}_{\dot{\Hb}^{5}}\right] \leq (p-1)^{\frac{p}{2}}L^{p},
	\end{align}
	where $L\geq1$ depends on $\Tr\left((-\Delta)^{3-2\gamma} GG^{\ast}\right)$, $\gamma$ and $\theta$, and is independent of $p$.  In addition, for any given $\hat{c}>0$, we can choose
    $$\theta> \hat{c} + \frac{\| GG^{\ast}\|_{Op}}{\lambda_1},$$ where $\| GG^{\ast}\|_{Op}$ is the operator norm of operator $GG^*$ from $\Hb$ to $\Hb$ and $\lambda_1$ is the smallest eigenvalue of Stokes operator, such that the following exponential estimate hold:
    \begin{align}\label{Z-exponential-regularity}
        \Eb \left[ \exp\left\{2\hat{c} \int_0^t  \|\zb(s)\|^2_{\dot{\Hb}^3} \drm s \right\} \right]  <+\infty,
    \end{align}
    for all $t\in[0,T].$
\end{proposition}

\begin{remark}
    Note that $\dot{\Hb}^5$-regularity of $\zb$ in \eqref{Z-regularity} is required to obtain the $\dot{\Hb}^3$-regularity of $\ub$ (see Subsection \ref{sec-more-regularity} below). 
\end{remark}

	\begin{proof}[Proof of Proposition \ref{thm_z_alpha}]
    The proof is divided into two steps.
    \vskip 2mm
    \noindent
	\textbf{Step I:} \textit{Proof of \eqref{Z-regularity}.}	Let us recall that the unique stationary solution to \eqref{eqn_z} has the explicit form 
		\begin{align*}
			\Upsilon(\z(t)) =\int_{0}^{t}e^{-(t-r)} S(t-r) \drm\Wrm(r),
		\end{align*}
		where $S(t)=e^{\theta\Delta t}$ denotes an analytic semigroup generated by the operator $\theta\Delta$. By \cite[Section 2.1]{Liu+Rockner_Book_2015}, the Wiener process $\Wrm$ can be written as $\Wrm=\sum\limits_{k\in\N}\sqrt{c_{k}}\beta_k\eb_{k}$ for an orthogonal basis $\{\eb_{k}\}_{k\in\N}$ of $\Hb$ consisting of eigenfunctions of $GG^{\ast}$ with corresponding eigenvalues $c_{k}$, and the coefficient satisfy $\sum\limits_{k\in\N}c_{k}<\infty,$ where $\{\beta_k\}_{k\in\N}$ is a sequence of mutually independent standard real-valued Brownian motions. Also, in view of our assumption, we have (see \cite[Appendix C]{DaZ})		
        \begin{align}
            \Tr\left((-\Delta)^{3-2\gamma}GG^{\ast}\right)=\sum_{k=1}^{\infty}c_k\left\|(-\Delta)^{\frac32-\gamma}\eb_{k}\right\|^2_{2}<\infty, \;\; \text{ for } \; \gamma \in \left(0,\frac12\right).
        \end{align}
	 Then it holds for $\gamma\in\left(0,\frac12\right)$ and $t\geq s$,
		\begin{align*}
			&\mathbb{E}\left[\left\|\Upsilon(\z(t))-\Upsilon(\z(s))\right\|^2_{2}\right]
			\nonumber\\ 	& =\mathbb{E}\left[\left\|\int\limits_{s}^{t}e^{-(t-r)}S(t-r)\drm\Wrm(r)\right\|^2_{2}\right] + \mathbb{E}\left[\left\|\int\limits_{0}^{s}\left[e^{-(t-r)}S(t-r)-e^{-(s-r)}S(s-r)\right]\drm\Wrm(r)\right\|^2_{2}\right]
			\nonumber\\ 	& =\sum_{k=1}^{\infty}c_{k}\int\limits_{s}^{t}e^{-2(t-r)}\left\|S(t-r)\eb_{k}\right\|^2_{2}\drm r  + \sum_{k=1}^{\infty}c_{k}\int\limits_{0}^{s}\left\|\left[e^{-(t-r)}S(t-r)-e^{-(s-r)}S(s-r)\right]\eb_{k}\right\|^2_{2}\drm r
			\nonumber\\ 	& \leq \sum_{k=1}^{\infty}c_{k}\int\limits_{s}^{t}e^{-2(t-r)}\left\|S(t-r)\eb_{k}\right\|^2_{2}\drm r  + \sum_{k=1}^{\infty}c_{k}\int\limits_{0}^{s} \left[e^{- (t-r)} - e^{- (s-r)}\right]^2\left\|S(t-r)\eb_{k}\right\|^2_{2}\drm r
			\nonumber\\ & \quad + \sum_{k=1}^{\infty}c_{k}\int\limits_{0}^{s}e^{-2(s-r)}\left\|\left[S(t-r)-S(s-r)\right]\eb_{k}\right\|^2_{2}\drm r
			 \nonumber\\ 	& \leq C\left[  \Tr(GG^{\ast})\left\{ \int\limits_{s}^{t}e^{-2(t-r)}\drm r   +  \int\limits_{0}^{s} \left[e^{-(t-r)}-e^{-(s-r)}\right]^2 \drm r\right\} + \sum_{k=1}^{\infty}c_{k}\int\limits_{0}^{s}e^{-2(s-r)}\left\|\left[S(t-r)-S(s-r)\right]\eb_{k}\right\|^2_{2}\drm r\right]
			\nonumber\\ 	&  \leq C\left[ \Tr(GG^{\ast}) (t-s)^{2\gamma}   + \sum_{k=1}^{\infty}c_{k}\int\limits_{0}^{s}e^{-2(s-r)}\left\|\left(\int_{s-r}^{t-r}(-\theta\Delta) S(\tau) \drm\tau\right) \eb_{k}\right\|^2_{2} \drm r \right]
			\nonumber\\ 	&  \leq C\left[  \Tr(GG^{\ast}) (t-s)^{2\gamma} +  \Tr(GG^{\ast})\int\limits_{0}^{s}e^{-2(s-r)}\left|\int_{s-r}^{t-r}\frac{1}{\tau}\drm\tau\right|^2\drm r \right]
			\nonumber\\ 	&  \leq C \left[  (t-s)^{2\gamma} +  \int\limits_{0}^{s}e^{-2(s-r)}(s-r)^{-2\gamma}\left|\int_{s-r}^{t-r}\tau^{\gamma-1}\drm\tau\right|^2\drm r \right]
			\nonumber\\ 	&  =  C\left[ (t-s)^{2\gamma}   +  \frac{1}{\gamma^2}\int\limits_{0}^{s}e^{-2(s-r)}(s-r)^{-2\gamma}\left[(t-r)^{\gamma}-(s-r)^{\gamma}\right]^2\drm r \right]
			\nonumber\\ 	& \leq C  (t-s)^{2\gamma}  \left[ 1 +  \int\limits_{0}^{s}e^{-2(s-r)}(s-r)^{-2\gamma}\drm r \right]
			\nonumber\\ 	  	& \leq C  (t-s)^{2\gamma}  ,
		\end{align*}
		where the constant $C$ depends only on the semigroup $S(\cdot)$, $\Tr(GG^{\ast})$, $\theta$ and $\gamma$ but it is independent of time. Again, it holds for $\gamma\in(0,\frac12)$, $\xi\in(\gamma,\frac12)$ and $t\geq s$,
		\begin{align*}
		&\mathbb{E}\left[\left\|(-\Delta)^{\frac32}\Upsilon(\zb(t))-(-\Delta)^{\frac32}\Upsilon(\zb(s))\right\|^2_{2}\right]
		\nonumber\\ 	& =\mathbb{E}\left[\left\|\int\limits_{s}^{t}e^{-(t-r)}(-\Delta)^{\frac32} S(t-r)\drm\Wrm(r)\right\|^2_{2}\right] 
	+ \mathbb{E}\left[\left\|\int\limits_{0}^{s}\left[e^{-(t-r)}(-\Delta)^{\frac32} S(t-r)-e^{-(s-r)}(-\Delta)^{\frac32} S(s-r)\right]\drm\Wrm(r)\right\|^2_{2}\right]
		\nonumber\\ & =\sum_{k=1}^{\infty}c_{k}\int\limits_{s}^{t}e^{-2 (t-r)}\left\|(-\Delta)^{\frac32} S(t-r)\eb_{k}\right\|^2_{2}\drm r  
+\sum_{k=1}^{\infty}c_{k}\int\limits_{0}^{s}\left\|\left[e^{- (t-r)}(-\Delta)^{\frac32} S(t-r)-e^{- (s-r)}(-\Delta)^{\frac32} S(s-r)\right]\eb_{k}\right\|^2_{2}\drm r
		\nonumber\\ 	& \leq C  \bigg[ \sum_{k=1}^{\infty}c_{k}\int\limits_{s}^{t}e^{-2 (t-r)}\left\|(-\Delta)^{\frac32-\gamma}(-\Delta)^{\gamma} S(t-r)\eb_{k}\right\|^2_{2}\drm r  
		\nonumber\\ & \quad + \sum_{k=1}^{\infty}c_{k}\int\limits_{0}^{s} e^{-2 (t-r)}\left\|(-\Delta)^{\frac32- \gamma}(-\Delta)^{\gamma}\left[S(t-r)-S(s-r)\right]\eb_{k}\right\|^2_{2}\drm r
		\nonumber\\ & \quad + \sum_{k=1}^{\infty}c_{k}\int\limits_{0}^{s} \left[e^{- (t-r)}-e^{- (s-r)}\right]^2\left\|(-\Delta)^{\frac32-\gamma}(-\Delta)^{\gamma} S(s-r)\eb_{k}\right\|^2_{2}\drm r \bigg]
		\nonumber\\ 	& \leq C \bigg[ \Tr((-\Delta)^{3-2\gamma}GG^{\ast}) \int \limits_{s}^{t} e^{-2 (t-r)} (t-r)^{-2\gamma} \drm r
			+ \Tr((-\Delta)^{3 -2\gamma}GG^{\ast})\int\limits_{0}^{s} \left[e^{- (t-r)}-e^{- (s-r)}\right]^2(s-r)^{-2\gamma} \drm r
		\nonumber\\ & \quad  + \sum_{k=1}^{\infty}c_{k}\int\limits_{0}^{s}e^{-2 (t-r)}\left\|(-\Delta)^{\frac32 - \gamma}(-\Delta)^{\gamma}\left(\int_{s-r}^{t-r}\frac{\drm}{\drm\tau}S(\tau)\drm\tau\right)\eb_{k}\right\|^2_{2}\drm r \bigg]
		\nonumber\\ 	&  \leq  C\bigg[ (t-s)^{2(\xi-\gamma)}   + \sum_{k=1}^{\infty}c_{k}\int\limits_{0}^{s}e^{-2 (t-r)}\left\|(-\Delta)^{\frac32-\gamma}(-\theta\Delta)^{\gamma}\left(\int_{s-r}^{t-r}(\theta\Delta) S(\tau) \drm\tau\right) \eb_{k}\right\|^2_{2}\drm r \bigg]
		\nonumber\\ 	&  \leq C \bigg[   (t-s)^{2(\xi-\gamma)}  +  \int\limits_{0}^{s}e^{-2(t-r)}\left|\int_{s-r}^{t-r}\tau^{-1-\gamma}\drm\tau\right|^2\drm r \bigg]
		\nonumber\\ 	&  \leq C \bigg[ (t-s)^{2(\xi-\gamma)}   +  \int\limits_{0}^{s}e^{-2 (t-r)}(s-r)^{-2\xi}\left|\int_{s-r}^{t-r}\tau^{\xi-1-\gamma}\drm\tau\right|^2\drm r \bigg]
		\nonumber\\ 	&  \leq C \bigg[  (t-s)^{2(\xi-\gamma)}   +  \int\limits_{0}^{s}e^{-2 (t-r)}(s-r)^{-2\xi}\left[(t-r)^{\xi-\gamma}-(s-r)^{\xi-\gamma}\right]^2\drm r \bigg]
		\nonumber\\ 	&  \leq C \bigg[ (t-s)^{2(\xi-\gamma)}   +  (t-s)^{2(\xi-\gamma)}\int\limits_{0}^{s}e^{-2 (t-r)}(s-r)^{-2\xi}\drm r\bigg]
		\nonumber\\ 	&   \leq C   (t-s)^{2(\xi-\gamma)},
	\end{align*}	
		where the constant $C>0$ depends only on the semigroup $S(\cdot)$, $\Tr((-\Delta)^{3 - 2 \gamma}GG^{\ast})$, $\theta$, $\gamma$ and $\xi$ but it is independent of time, and we have also used the properties of semigroup $S(\cdot)$ from \cite[Theorems 5.2 and 6.13]{Pazy_1983}. Using Gaussianity (see for e.g. \cite[Subsection 5.4]{Papoulis_1984}), we have 
		\begin{align*}
			\mathbb{E}\left[\|\z(t)-\z(s)\|^p_{\dot{\Hb}^{5}}\right] \leq (p-1)^{\frac p2}\left(\mathbb{E}\left[\|\z(t)-\z(s)\|^2_{\dot{\Hb}^{5}}\right]\right)^{\frac p2}.
		\end{align*}
		By Kolmogorov's continuity criterion (see \cite[Theorem 3.3]{DaZ}), we obtain \eqref{Z-regularity} as desired.
        \vskip 2mm
        \noindent
        \textbf{Step II:} \textit{Proof of \eqref{Z-exponential-regularity}.}   Applying It\^o formula to the process $\|(\Irm-\alpha_1\Delta)\zb(\cdot)\|_2^2$ and integrating form $0$ to $t$, we obtain 
    \begin{align}
       &  \|(\Irm-\alpha_1\Delta)\zb(t)\|_2^2  + 2\theta \int_0^t  \|(-\Delta)^{\frac12}(\Irm-\alpha_1\Delta)\zb(s)\|_2^2 \drm s + 2 \int_0^t  \|(\Irm-\alpha_1\Delta)\zb(s)\|_2^2 \drm s
       \nonumber\\ & = 2 \int_0^t ((\Irm-\alpha_1\Delta)\zb(s) , \drm \Wrm (s)) +  \Tr (GG^*)t
       \nonumber\\ & = 2 \int_0^t ( (\Irm-\alpha_1\Delta)\zb(s) , \drm \Wrm (s)) \drm s - 2  \int_0^t \|G^{\ast} (\Irm-\alpha_1\Delta)\zb(s)\|^2_2 \drm s +  \Tr (GG^*)t
        + 2 \int_0^t \|G^{\ast}(\Irm-\alpha_1\Delta)\zb(s)\|^2_2 \drm s
       \nonumber\\ & = 2 \int_0^t ( (\Irm-\alpha_1\Delta)\zb(s) , \drm \Wrm (s))  \drm s - 2 \int_0^t \|G^{\ast} (\Irm-\alpha_1\Delta)\zb(s)\|^2_2  \drm s +  \Tr (GG^*)t
         \nonumber\\ & \quad + 2 \int_0^t ( (\Irm-\alpha_1\Delta)\zb(s),  GG^{\ast} (\Irm-\alpha_1\Delta)\zb(s)) \drm s
        \nonumber\\ & \leq 2 \int_0^t ( (\Irm-\alpha_1\Delta)\zb(s)  , \drm \Wrm (s)) \drm s  - 2 \int_0^t \|G^{\ast}(\Irm-\alpha_1\Delta)\zb(s)\|^2_2  \drm s +  \Tr (GG^*)t
         + 2\|GG^{\ast}\|_{Op} \int_0^t \|(\Irm-\alpha_1\Delta)\zb(s)\|_2^2 \drm s
         \nonumber\\ & \leq 2 \int_0^t ( (\Irm-\alpha_1\Delta)\zb(s) , \drm \Wrm (s)) \drm s - 2 \int_0^t \|G^{\ast}(\Irm-\alpha_1\Delta)\zb(s)\|^2_2  \drm s  +  \Tr (GG^*)t
         \nonumber\\ & \quad + \frac{2 \|GG^{\ast}\|_{Op}}{\lambda_1} \int_0^t \|(-\Delta)^{\frac12}(\Irm-\alpha_1\Delta)\zb(s)\|_2^2 \drm s,
    \end{align}
where $\|GG^{\ast}\|_{Op}$ denotes the operator norm of operator $GG*$ from $\Hb$ to $\Hb$.  This gives
    \begin{align}\label{expo-in-delta}
     & \exp\left\{2\left(\theta-\frac{\| GG^{\ast}\|_{Op}}{\lambda_1}\right) \int_0^t  \|(-\Delta)^{\frac12}(\Irm-\alpha_1\Delta)\zb(s)\|_2^2 \drm s \right\}
    \nonumber\\ &  \leq  \exp\left\{2 \int_0^t ( (\Irm-\alpha_1\Delta)\zb(s) , \drm \Wrm (s)) - 2 \int_0^t \|G^{\ast} (\Irm-\alpha_1\Delta)\zb(s)\|^2_2 \right\} e^{\Tr (GG^*)t}.
    \end{align}

Since $\mathcal{E}(t):= \exp\left\{2 \int_0^t ( (\Irm-\alpha_1\Delta)\zb(s) , \drm \Wrm (s)) - 2 \int_0^t \|G^{\ast} (\Irm-\alpha_1\Delta)\zb(s)\|^2_2 \right\}$  is a nonnegative local martingale, it follows by Fatou’s Lemma that $\mathcal{E}$ is a supermartingale (see \cite[pp. 253]{Liu+Rockner_Book_2015}), and since $\mathcal{E}(0)=1$, we have
    \begin{align}\label{eqn-supermartingale}
        \Eb\left[\mathcal{E}(t)\right] \leq \Eb\left[\mathcal{E}(0)\right] =1.
    \end{align}
    Hence by \eqref{expo-in-delta} and \eqref{eqn-supermartingale}, we have 
    \begin{align}
        \Eb \left[ \exp\left\{2\left(\theta-\frac{\| GG^{\ast}\|_{Op}}{\lambda_1}\right) \int_0^t  \|(-\Delta)^{\frac12}(\Irm-\alpha_1\Delta)\zb(s)\|_2^2 \drm s \right\} \right] \leq  e^{\Tr (GG^*)t}.
    \end{align}
    Finally, for any given $\hat{c}>0$, one can choose $\theta$ sufficiently large satisfying $$\hat{c}< \theta-\frac{\| GG^{\ast}\|_{Op}}{\lambda_1},$$ and obtain
    \begin{align}
        \Eb \left[ \exp\left\{2\hat{c} \int_0^t  \|\zb(s)\|^2_{\dot{\Hb}^3} \drm s \right\} \right] \leq  e^{\Tr (GG^*)t}<+\infty.
    \end{align}
     This completes the proof.
	\end{proof}

\subsection{Control problem}
Our main goal is to control the solution of the system \eqref{stochastic-third-grade-fluids-equations}
by a stochastic distributed force $\fb$. The control variables $\fb$ belong to the set $\mathcal{F}_{ad}$ of admissible controls, which
is defined as a nonempty compact convex subset of $\Lrm^p(\Omega;\mathrm{L}^2(0,{T};\dot{\Hb}^{1}(\mathbb{T}^2)))$, $p>4$.   Let $\vb$ be the solution to system \eqref{stochastic-third-grade-fluids-equations} corresponding to control $\fb\in \mathcal{F}_{ad}$. We consider the cost functional $ \mathrm{J}: \mathcal{F}_{ad}  \to \mathbb{R}^+$ given by 
\begin{align}\label{eqn-cost-functional}
    \mathrm{J}(\fb,\vb  )=\frac12 \mathbb{E}\bigg[ \int_0^{T} \|\vb(t)-\vb_d(t)\|^2_2 \drm t\bigg] + \frac{\lambda}{2}\mathbb{E}\bigg[\int_0^{T}\|\fb(t)\|_2^2\drm t\bigg],
\end{align}
where $\vb_d \in \Lrm^2(\Omega;\mathrm{L}^2(0,{T};\dot{\Lb}^2(\mathbb{T}^2)))$ corresponds to a desired target field and any $\lambda> 0$. The control problem reads
\begin{align}\label{eqn-control-problem}
    \min_{\fb\in\mathcal{F}_{ad}}\bigg\{\mathrm{J}(\fb,\vb) \; : \; \vb \text{ is the solution of \eqref{stochastic-third-grade-fluids-equations} with force } \fb \bigg\}.
\end{align}
\begin{remark}
 In this article, the Lagrangian $\mathfrak{L}$ is given by 
    \begin{align*}
        \mathfrak{L}(\cdot,\fb,\vb)=\frac12 \|\vb-\vb_d\|^2_2 + \frac{\lambda}{2}\|\fb\|_2^2.
    \end{align*}
    Therefore, we have
    	\begin{align*}
    		\Eb\left[\int_0^{T} (\nabla_{\vb}\mathfrak{L}(t,\fb(t),\vb(t)), \yb(t))\drm t\right] & = \Eb \left[\int_0^{T} (\yb(t), \vb(t)-\vb_d(t))\drm t\right],\\
    	\Eb	\left[\int_0^{T} (\nabla_{\fb}\mathfrak{L}(t,\fb(t),\vb(t)), \gb(t))\drm t\right] & = \lambda \Eb\left[\int_0^{T} (\gb(t), \fb(t))\drm t\right],
    	\end{align*}
    	for any $\yb \in \Lrm^2(\Omega; \mathrm{L}^{\infty}(0,{T};\Drm(\A^{\frac{3}{2}})))$ and $\gb\in  \mathcal{F}_{ad}$.
\end{remark}
\subsection{Main results} 
Our main result demonstrates the existence of a solution to the control problem and formulates the first-order optimality conditions.
\begin{theorem}\label{main-thm1}
    Assume that $\vb_0\in\Drm(\A^{\frac{3}{2}})$. Then the control problem \eqref{eqn-control-problem} admits, at least, one optimal solution 
    \begin{align}
        (\widetilde{\fb}, \widetilde{\vb})\in \mathcal{F}_{ad}\times  \Lrm^2(\Omega;\mathrm{L}^{\infty}(0,{T}; \Drm(\A^{\frac{3}{2}}))),
    \end{align}
    where $\widetilde{\vb}$ is the unique solution of \eqref{stochastic-third-grade-fluids-equations} with $\fb=\widetilde{\fb}$. Moreover, there exists a unique solution $\pb$ of \eqref{eqn-adjoint-pi} with $\gb=\nabla_{\vb}\mathfrak{L}(\cdot,\widetilde{\fb},\widetilde{\vb}),$ such that if $\widetilde{\mfrk}$ is the solution of \eqref{eqn-linearized-pi} for $\vb=\widetilde{\vb}$ and $\boldsymbol{\psi}=\boldsymbol{\psi}-\widetilde{\fb}$, the following duality property 
    \begin{align}
        \Eb\left[\int_0^{T} (\boldsymbol{\psi}(t) - \widetilde{\fb}(t), \widetilde{\pb}(t))\drm t\right] = \Eb\left[\int_0^{T} (\nabla_{\vb}\mathfrak{L}(t,\widetilde{\fb}(t), \widetilde{\vb}(t)), \widetilde{\mfrk}(t))\drm t\right],
    \end{align}
    and the following optimality condition hold 
    \begin{align}\label{oocc}
        \Eb \left[ \int_0^{T} (\boldsymbol{\psi}(t)-\widetilde{\fb}(t), \widetilde{\pb}(t) + \nabla_{\fb}\mathfrak{L}(t,\widetilde{\fb}(t), \widetilde{\vb}(t)) )\drm t\right]  \geq 0.
    \end{align}
\end{theorem}

\iffalse 
An essential step in the analysis of the control problem involves studying the solutions of the coupled system, which consists of the state equation \eqref{stochastic-third-grade-fluids-equations}, the adjoint equation \eqref{eqn-adjoint-pi}, and the optimality condition \eqref{oocc}. Our next result contributes to this direction by proving the uniqueness of the solutions to the coupled system for \eqref{eqn-control-problem}.

{\color{red}
\begin{theorem}\label{main-thm2}
Let  $\lambda > 2 \Gamma_2\widetilde{\lambda}\left[\Gamma+\frac{4}{\mu} |\alpha| [S_4]^2 \Gamma_1 +\frac{112\beta}{\mu} [S_4]^2 S  \kappa \Gamma_1 \right]$, where  $\Gamma_1, \Gamma_2, \widetilde{\lambda}, \Gamma, \mu, \alpha, \beta, S_4, S^{\varepsilon}$ and  $\kappa$ are given in \eqref{Stability-1}, \eqref{Stability-2}, \eqref{estimate-adjoint}, \eqref{estimate-trilinear}, \eqref{I},  \eqref{I}, \eqref{I}, \eqref{p-embedding} (for $p=4$), \eqref{delta-embedding-infty} and \eqref{estimate-state}, respectively. Then, the optimal control problem \eqref{eqn-control-problem} has a unique  solution.
\end{theorem}}
\fi

\section{Existence and uniqueness of global solutions of state equation \texorpdfstring{\eqref{stochastic-third-grade-fluids-equations}}{}}\label{section-3}
In this section, we establish the existence of unique global solution to system \eqref{stochastic-third-grade-fluids-equations}. Taking projection $\mathcal{P}$ to system \eqref{stochastic-third-grade-fluids-equations}, we obtain
\begin{equation}\label{STGF-Projected}
    \left\{
    \begin{aligned}
         \drm\Upsilon(\vb)  +
         \biggl\{\nu \A \vb + \Pcal[(\vb\cdot\nabla)\Upsilon(\vb)]  + \mathcal{P}\biggl[\displaystyle\sum_{j=1}^2 [\Upsilon(\vb)]^j \nabla \vb^j\biggr] -(\alpha_1+\alpha_2)\Jcal(\vb) 
  -\beta \Kcal(\vb)\biggr\}\drm t & = \mathcal{P} \fb \drm t + \drm \Wrm, 
  \\ & \;\; \text{ for a.e. } t\in(0,T)  \\
\vb(0)& = \vb_0,
    \end{aligned}
    \right.
\end{equation}
in $\Hb$.  Let us first provide the definition of a unique global strong solution in the probabilistic sense to the system \eqref{stochastic-third-grade-fluids-equations}.
\begin{definition}\label{def-Solution}
	 A $\Drm(\A)$-valued $\{\mathscr{F}_t\}_{t\geq 0}$-adapted stochastic process $\vb(\cdot)$ is called a \emph{strong solution} to the system \eqref{stochastic-third-grade-fluids-equations} if the following conditions are satisfied: 
	\begin{enumerate}
		\item [(i)] the process $\{ \vb(t)\}_{t\geq 0}$ has trajectories in 
        %$\vb\in\mathrm{L}^2(\Omega;\mathrm{L}^{\infty}(0,T;\Drm(\A)))$ and $\vb(\cdot)$ has a $\Drm(\A)$-valued  modification, which is progressively measurable with continuous paths in $\Vb$ and 
        $\mathrm{C}([0,T];\Drm(\A))\cap\mathrm{L}^{\infty}(0,T;\Drm(\A^{\frac32}))$, $\mathbb{P}$-a.s.,
		\item [(ii)] the following equality holds for every $t\in [0, T ]$ and for any  $\psi\in\Hb$:
\begin{align}\label{4.5}
(\Upsilon(\vb(t)),\psi)&=(\Upsilon(\vb_0),\psi)-\int_0^t\langle\nu \A\vb(s) - (\alpha_1+\alpha_2) \Jcal(\vb(s)) + \beta\Kcal(\vb(s)),\psi\rangle\drm s +\int_0^t b(\vb(s),\psi(s),\Upsilon(\vb(s))) \drm s \nonumber \\ &\quad - \int_0^t b(\psi(s), \vb(s),\Upsilon(\vb(s)))  \drm s + \int_0^t(\fb(s) ,\psi)\drm s + (\Wrm(t),\psi),\ \mathbb{P}\text{-a.s.} 
\end{align}	
	\end{enumerate}
\end{definition}

\begin{theorem}\label{Wellposedness-state}
    Suppose that $\Tr\left((-\Delta)^{3-2\gamma} GG^{\ast}\right)<\infty$ for $\gamma\in(0,\frac12)$,  $\vb_0\in \Drm(\A^{\frac32})$ and $\fb\in\Lrm^2(\Omega;\Lrm^2(0,T;\dot{\Hb}^1(\Tb^2)))$. Then, there exists a unique solution $\vb$ to system \eqref{stochastic-third-grade-fluids-equations} in the sense of Definition \ref{def-Solution}. In addition, the solution $\vb\in\mathrm{L}^{p}(\Omega;\mathrm{L}^{\infty}(0,T;\Drm(\A^{\frac32})))$, $p\geq2$, and for any given $\hat{c}>0$
    \begin{align}\label{eqn-vb-exponential-moments}
   \Eb\bigg[\exp\bigg\{\hat{c}\int_0^t\|\vb(s)\|^{2}_{\dot{\Hb}^3} \drm s \bigg\}\bigg] <+\infty,
\end{align}
for all $t\in[0,T]$.
\end{theorem}
\begin{remark}
    Note that we are considering deterministic initial data, that is, $\vb_0\in \Drm(\A^{\frac32})$, but one can also take 
$\vb_0\in \Lrm^2(\Omega; \Drm(\A^{\frac32}))$ under an appropriate condition, which could help in obtaining \eqref{eqn-vb-exponential-moments}. 
\end{remark}

\subsection{Proof of Theorem \ref{Wellposedness-state}}

As we have already discussed that we prove the existence and uniqueness of strong solution $\vb$ to system \eqref{STGF-Projected} by decomposition method, that is, $\vb=\ub+\zb$ where $\ub$ and $\zb$ are the unique solutions of systems \eqref{random-third-grade-fluids-equations} and \eqref{eqn_z}, respectively. We already know the existence of a unique regular solution to system \eqref{eqn_z}, see Proposition \ref{thm_z_alpha}. For $\vb_0\in \Drm(\A^{\frac{3}{2}})$, $\fb\in\Lrm^2(\Omega;\Lrm^2(0,T;\dot{\Hb}^{1}(\Tb^2)))$ and $\Tr\left((-\Delta)^{3-2\gamma} GG^{\ast}\right)<\infty$, $\gamma\in(0,\frac12)$ (see Proposition \ref{thm_z_alpha}), we have for $p\geq 2$
\begin{align}\label{pathwise-regularity}
       \|\A^{\frac{3}{2}}\vb_0\|^2_{2} +   \int_0^T\|\fb(t)\|_{\dot{\Hb}^{1}}^2 \drm t + \sup_{t\in[0,T]}\|\zb(t)\|_{\dot{\Hb}^{5}}^p < +\infty,   \;\;\Pb\text{-a.s.},
\end{align}
that is, there exists a set $\widetilde{\Omega}\subset \Omega$ with $\Pb(\widetilde{\Omega})=1$ such that \eqref{pathwise-regularity} holds for each $\omega\in\widetilde{\Omega}$.

Let us choose and fix $T > 0$ and $\omega\in\widetilde{\Omega}$, and we show the existence of a unique solution to random system \eqref{random-third-grade-fluids-equations} for each $\omega\in\widetilde{\Omega}$ by treating system \eqref{random-third-grade-fluids-equations} as deterministic with given data $\vb_0$, $\fb$ and $\zb$ satisfying \eqref{pathwise-regularity}. Taking projection $\mathcal{P}$ to system \eqref{random-third-grade-fluids-equations}, we obtain
\begin{equation}\label{CTGF-Projected}
    \left\{
    \begin{aligned}
         \partial_t\Upsilon(\ub)  +\nu \A \ub + \Pcal[((\ub+\zb)\cdot\nabla)\Upsilon(\ub+\zb)] & + \mathcal{P}\left[\displaystyle\sum_{j=1}^2 (\Upsilon(\ub+\zb))^j \nabla (\ub+\zb)^j\right] 
  -(\alpha_1+\alpha_2)\Jcal(\ub+\zb) 
 \\  -\beta \Kcal(\ub+\zb) & = \mathcal{P} \fb  +  \zb +  (\theta+\alpha_1-\nu)\A\zb + \theta\alpha_1 \A^2\z, \;\; \text{ for a.e. } t\in(0,T)  \\
\ub(0)& = \vb_0,
    \end{aligned}
    \right.
\end{equation}
in $\Hb$.  We will show the existence of unique solution to system \eqref{CTGF-Projected} in the following sense:
\begin{definition}\label{def-ranndom-TGF}
   A function  $\ub \in  \mathrm{C}([0,T]; \Drm(\A))\cap \mathrm{L}^{\infty}(0,T;\Drm(\A^{\frac32}))$  with $\partial_t\ub \in\mathrm{L}^{2}(0,T;\Drm(\A)),$ 
    is called a \emph{strong solution} to system \eqref{CTGF-Projected}, if for $\vb_0$, $\fb$ and $\zb$ satisfying \eqref{pathwise-regularity}, it satisfies 
	\begin{itemize}
		\item [(i)] for any $\psi\in \Hb,$ 
		\begin{align}\label{W-randomTGF}
\left\langle\partial_t\Upsilon(\ub(t)),\psi\right\rangle &=-\int_0^t\langle\nu \A\ub(s) - (\alpha_1+\alpha_2) \Jcal(\ub(s)+\zb(s)) - \beta\Kcal(\ub(s)+\zb(s)),\psi\rangle\drm s
            \nonumber \\ &\quad  +\int_0^t b(\ub(s)+\zb(s),\psi(s),\Upsilon(\ub(s)+\zb(s))) \drm s  - \int_0^t b(\psi(s), \ub(s)+\zb(s),\Upsilon(\ub(s)+\zb(s)))  \drm s 
            \nonumber \\ &\quad  + \int_0^t(\fb(s)  +  \zb(s) +  (\theta+\alpha_1-\nu)\A\zb(s) + \theta\alpha_1 \A^2\z(s) ,\psi)\drm s,
		\end{align}
		for a.e. $t\in[0,T];$
		\item [(ii)] the initial data:
		$$\ub(0)=\vb_0 \ \text{ in }\ \Drm(\A^{\frac32}).$$
	\end{itemize}
\end{definition}

\begin{theorem}\label{solution}
	There exists a unique  solution $\ub$ to system \eqref{CTGF-Projected} in the sense of Definition \ref{def-ranndom-TGF}. 
    %In addition, for $\vb_0$, $\fb$ and $\zb$ satisfying \eqref{pathwise-regularity} with $s=2$, the unique solution $\ub \in  \mathrm{C}([0,T]; \Drm(\A))\cap \mathrm{L}^{\infty}(0,T;\Drm(\A^{\frac32}))$.
\end{theorem}

The proof of Theorem \ref{solution} is divided into the several steps which are given as separate subsubsections in the sequel.

\subsubsection{Finite-dimensional approximation} Here, we consider the basis $\{\wb_i\}_{i\in\mathbb{N}}$ of eigenfunctions of the Stokes operator $\A$, which is orthonormal in $\Hb$ and orthogonal in $\Vb$. We define $\Hb_n:=\mathrm{span}\{\boldsymbol{w}_j:j=1,\ldots, n\}$ for $n\in\N$. Let $\Prm_n$ denote the orthogonal projection from $\Vb'$ onto $\Hb_n$. For every $\xb\in\Vb'$, the projection is given by $\Prm_n\xb=\sum\limits_{j=1}^n \langle\xb,\boldsymbol{w}_j\rangle \boldsymbol{w}_j$. Since each element $\xb\in\Hb$ induces a functional $\xb^*\in\Vb'$ defined by $\langle\xb^*,\yb\rangle=(\xb,\yb)$, $\yb\in\Vb$, the restriction $\Prm_n|_\Hb$ corresponds to the orthogonal projection from $\Hb$ onto $\Hb_n$.  Therefore, for all $\xb\in\Hb$,
$\mathrm{P}_n\xb=\sum\limits_{i=1}^n\left(\xb,\boldsymbol{w}_i\right)\boldsymbol{w}_i$ and  $\|\mathrm{P}_n\xb-\xb\|_{2}\to 0$ as $n\to\infty$.
By the relation \eqref{ipHV}, for all $\xb\in\Vb$, $\Prm_n\xb$ also represents the orthogonal projection with respect to the inner product $(\cdot,\cdot)_{\Vb}$ in $\Vb$, implying that $\|\mathrm{P}_n\xb-\xb\|_{\Vb}\to 0$ as $n\to\infty$. \iffalse We define the following projected operators: 
\begin{align*}
    \A_n\cdot:=\Prm_n\A\cdot, \Brm_n(\cdot):=\Prm_n\Brm(\cdot),  \mathcal{J}_n(\cdot):=\Prm_n\mathcal{J}(\cdot), \mathcal{K}_n(\cdot):=\Prm_n\mathcal{K}(\cdot), \fb_n:=\Prm_n[\mathbb{P}\fb] \text{ and } \zb_n:=\Prm_n[\zb].
\end{align*}\fi
For each $n\in\N$, we search for a approximate solution of the form 
\begin{align*}
	\ub_n(x,t):=\sum\limits_{k=1}^n g_k^n(t)\boldsymbol{w}_k(x),
\end{align*}
where $g_1^n(t),\ldots,g_k^n(t)$ are unknown scalar functions of $t$ such that it satisfies the following finite-dimensional system of ordinary differential equations in $\Hb_n$:
\begin{equation}
\label{Finite-dimensional-approx}
    \left\{
    \begin{aligned}
         \partial_t\Upsilon(\ub_n)  +\nu \A \ub_n + \Prm_n\Pcal[((\ub_n+\zb)\cdot\nabla)\Upsilon(\ub_n+\zb)] & + \Prm_n\mathcal{P}\left[\displaystyle\sum_{j=1}^2 (\Upsilon(\ub_n+\zb))^j \nabla (\ub_n+\zb)^j\right] 
 \\  -(\alpha_1+\alpha_2)\Prm_n[\Jcal(\ub_n+\zb) ]
   -\beta \Prm_n [\Kcal(\ub_n+\zb)] & = \mathcal{P} \fb_n  +  \zb_n +  (\theta+\alpha_1-\nu)\A\zb_n + \theta\alpha_1 \A^2\z_n,  \text{ for a.e. } t\in(0,T)  \\
\ub_n(0)& =  \Prm_n[\vb_0]=: \ub_{0,n},
    \end{aligned}
    \right.
\end{equation}
where $\fb_n:=\Prm_n[\mathcal{P}\fb]$ and $\zb_n:=\Prm_n[\zb]$.
So we have to solve a system of $n$ ordinary differential equations (ODE). Classical results on ODE ensures the existence of $g_k^n$ in an interval $[0, T_n]$, for $0<T_n\leq T$. We obtain the time $T_n=T$ by establishing the uniform energy estimates of the solutions satisfied by the system \eqref{Finite-dimensional-approx}.

\subsubsection{A priori estimates}
\textbf{\underline{Step I}:} Taking the inner product in $\eqref{Finite-dimensional-approx}_1$ with $\ub_n$, we get
\begin{align}
   &  \frac12 \frac{\drm}{\drm t}\bigg[\|\ub_n(t)\|_2^2 + \alpha_1 \|\nabla\ub_n(t)\|_2^2\bigg] + \nu \|\nabla\ub_n(t)\|_2^2 + \frac{\beta}{2}\|\Arm(\ub_n(t)+\zb(t))\|^4_4
   %\nonumber\\ & = - b(\ub_n(t)+\zb(t), \ub_n(t)+\alpha_1\A\ub_n(t),\ub_n(t)) - b(\ub_n(t)+\zb(t), \zb(t)+\alpha_1\A\zb(t),\ub_n(t))
   %\nonumber\\ & \quad - b(\ub_n(t), \ub_n(t)+\zb(t),\ub_n(t)+\alpha_1\A\ub_n(t)) -  b(\ub_n(t), \ub_n(t)+\zb(t),\zb(t)+\alpha_1\A\zb(t))
   %\nonumber\\ & \quad 
   %- \frac{(\alpha_1+\alpha_2)}{2}\int_{\Tb^2}[\Arm(\ub_n(x,t)+\zb(x,t))]^2:\Arm(\ub_n(x,t)) \drm x 
   %\nonumber\\ & \quad + \frac{\beta}{2}\int_{\Tb^2}|\Arm(\ub_n(x,t)+\zb(x,t))|^2[\Arm(\ub_n(x,t)+\zb(x,t)):\Arm(\zb(x,t))] \drm x
   %\nonumber\\ & \quad + ( \fb(t), \ub_n(t))  +  (\zb(t),\ub_n(t)) + (\nu-\alpha_1)(\A\zb(t),\ub_n(t))
   \nonumber\\ & = \underbrace{\alpha_1 b(\zb(t), \ub_n(t), \A\ub_n(t))}_{I_1} + \underbrace{ b(\zb(t), \ub_n(t), \Upsilon(\zb(t)))}_{I_2} 
    - \underbrace{b(\ub_n(t), \zb(t),\ub_n(t))}_{I_3} 
     -  \underbrace{\alpha_1b(\ub_n(t), \zb(t),\A(\ub_n(t) +\zb(t)))}_{I_4}
   \nonumber\\ & \quad - \underbrace{\frac{(\alpha_1+\alpha_2)}{2}\int_{\Tb^2}[\Arm(\ub_n(x,t)+\zb(x,t))]^2:\Arm(\ub_n(x,t)) \drm x }_{I_5}
   \nonumber\\ & \quad + \underbrace{\frac{\beta}{2}\int_{\Tb^2}|\Arm(\ub_n(x,t)+\zb(x,t))|^2[\Arm(\ub_n(x,t)+\zb(x,t)):\Arm(\zb(x,t))] \drm x}_{I_6}
   \nonumber\\ & \quad + \underbrace{[( \fb(t), \ub_n(t))  +  (\zb(t) +  (\theta+\alpha_1-\nu)\A\zb(t) + \theta\alpha_1 \A^2\z(t),\ub_n(t))]}_{I_7},\label{EE-ub1}
\end{align}
for a.e. $t\in[0,T]$. Let us estimate each term on the right hand side of \eqref{EE-ub1} using integration by parts, and H\"older's, Sobolev and Young's inequalities as follows:
\begin{align}
    |I_1| & = |\alpha_1 b(\zb, \ub_n, \A\ub_n)|
     =\left|\alpha_1\int_{\Tb^2}\sum_{i,j,k}^2
    \partial_k \zb^i
     \partial_i \ub_n^j
    \partial_k \ub_n^j \drm x\right| 
       \leq C \|\zb\|_{\Wb^{1,\infty}}\|\nabla\ub_n\|_{2}^2  \leq C \|\zb\|_{\dot{\Hb}^{3}}\|\nabla\ub_n\|_{2}^2,\label{EE-ub2}\\
 %   \nonumber\\ 
    |I_2| & \leq \|\zb\|_{\infty}\|\nabla\ub_n\|_2\|\z+\alpha_1\A\zb\|_2 \leq C \|\zb\|_{\dot{\Hb}^2}^2\|\nabla\ub_n\|_2 \leq C \|\zb\|_{\dot{\Hb}^3}^4 + C\|\nabla\ub_n\|_2^2 ,\label{EE-ub3}\\
  %  \nonumber\\
    |I_3| & \leq \|\nabla\zb\|_2\|\ub_n\|^2_{4} \leq C %\|\zb\|_{\dot{\Hb}^{1}}\|\nabla\ub\|_{2}^2\leq C 
    \|\zb\|_{\dot{\Hb}^{3}}\|\nabla\ub_n\|_{2}^2,\label{EE-ub4}\\
  %  \nonumber\\ 
    |I_4| & = \alpha_1\left|\int_{\Tb^2}\sum\limits_{i,j,k=1}^2 \partial_k \ub_n^i
   \partial_i \zb^j
   \partial_k (\ub_n^j+\zb^j) \drm x + \int_{\Tb^2}\sum\limits_{i,j,k=1}^2 \ub_n^i
  \partial_i\partial_k \zb^j\partial_k (\ub_n^j+\zb^j) \drm x\right| 
    \nonumber\\ & \leq C \|\nabla\ub_n\|_2\|\nabla\zb\|_4\|\Arm(\ub_n+\zb)\|_4 + C  \|\ub_n\|_4\|\zb\|_{\dot{\Hb}^2}\|\Arm(\ub_n+\zb)\|_4
    \nonumber\\ & \leq C \|\nabla\ub_n\|_2\|\zb\|_{\dot{\Hb}^2}\|\Arm(\ub_n+\zb)\|_4
      \leq \frac{\beta}{12} \|\Arm(\ub_n+\zb)\|^4_4 + C \|\zb\|^4_{\dot{\Hb}^3} + C\|\nabla\ub_n\|^2_2,\label{EE-ub5}\\
%    \nonumber\\ 
    |I_5| & \leq \frac{\beta}{12} \|\Arm(\ub_n+\zb)\|^4_4 + C\|\nabla\ub_n\|^2_2,\label{EE-ub6}\\
%    \nonumber\\ 
    |I_6| & \leq \frac{\beta}{12} \|\Arm(\ub_n+\zb)\|^4_4 + C\|\Arm(\zb)\|^4_4 \leq \frac{\beta}{12} \|\Arm(\ub_n+\zb)\|^4_4 + C \|\zb\|^4_{\dot{\Hb}^3}, \label{EE-ub7}\\
%    \nonumber\\ 
    |I_7| & \leq C [\|\fb\|_2+\|\zb\|_{\dot{\Hb}^4}]\|\ub_n\|_2 \leq C [\|\fb\|_2^2+\|\zb\|^4_{\dot{\Hb}^4}+1]+ C \|\ub_n\|^2_2.\label{EE-ub8}
\end{align}
Combining \eqref{EE-ub1}-\eqref{EE-ub8}, we obtain
\begin{align}
    &   \frac{\drm}{\drm t}\bigg[\|\ub_n(t)\|_2^2 + \alpha_1 \|\nabla\ub_n(t)\|_2^2\bigg] + 2\nu \|\nabla\ub_n(t)\|_2^2 + \frac{\beta}{2}\|\Arm(\ub_n(t)+\zb(t))\|^4_4
    \nonumber\\ & \leq C [\|\fb(t)\|_2^2+\|\zb(t)\|^4_{\dot{\Hb}^4}+1] + C[1+\|\zb(t)\|_{\dot{\Hb}^3}]\bigg[\|\ub_n(t)\|_2^2 + \alpha_1 \|\nabla\ub_n(t)\|_2^2\bigg],
\end{align}
for a.e. $t\in[0,T]$, which, by an application of the Gronwall Lemma, gives
\begin{align}
    &   \|\ub_n(t)\|_2^2 + \alpha_1 \|\nabla\ub_n(t)\|_2^2 + 2\nu \int_{0}^{t} \|\nabla\ub_n(s)\|_2^2 \drm s + \frac{\beta}{2}\int_{0}^{t}\|\Arm(\ub_n(s)+\zb(s))\|^4_4\drm s
    \nonumber\\ & \leq \bigg[\|\ub_n(0)\|_2^2 + \alpha_1 \|\nabla\ub_n(0)\|_2^2 +  C \int_{0}^{t}\{\|\fb(s)\|_2^2+\|\zb(s)\|^4_{\dot{\Hb}^4}+1\}\drm s\bigg] e^{C\int_{0}^{t}\{1+\|\zb(s)\|_{\dot{\Hb}^3}\}\drm s}
    \nonumber\\ & \leq C\bigg[\|\vb_0\|_{\Vb}^2 +   \int_{0}^{t}\{\|\fb(s)\|_2^2+\|\zb(s)\|^4_{\dot{\Hb}^4}+1\}\drm s\bigg] e^{C\int_{0}^{t}\{1+\|\zb(s)\|_{\dot{\Hb}^3}\}\drm s}:= K_1(t),\label{EE-ub9}
\end{align}
for all $t\in[0,T]$. In view of \eqref{pathwise-regularity},
%the fact that $\fb\in\mathrm{L}^{2}(0,T;\dot{\Lb}^{2}(\mathbb{T}^2))$ and $\zb\in\mathrm{L}^{\infty}(0,T;\dot{\Hb}^{3}(\mathbb{T}^2))$ $\Pb$-a.s., 
we have from \eqref{EE-ub9} that
		\begin{align}\label{EE-ub10}
			\{\ub^n\}_{n\in\N} \text{ is a bounded sequence in }\mathrm{L}^{\infty}(0,T;\Vb)\cap\mathrm{L}^{4}(0,T; {\mathbb{W}}^{1,4}(\mathbb{T}^2)).
		\end{align}
%\vskip 2mm
%\noindent
%\textbf{Step III.} \textit{A priori estimates when $\ub_0\in\Drm(\A)$.} 
\vskip 2mm
\noindent 
\textbf{\underline{Step II}:} Since $\{\wb_i\}_{i\in\mathbb{N}}$ is a sequence of eigenfunctions of the Stokes operator $\A$, it gives $\A\ub_n\in \Hb_n.$ Taking the inner product in $\eqref{Finite-dimensional-approx}_1$ with $\A\ub_n$, we get
\begin{align}
   &  \frac12 \frac{\drm}{\drm t}\bigg[\|\nabla\ub_n(t)\|_2^2 + \alpha_1 \|\A\ub_n(t)\|_2^2\bigg] + \nu \|\A\ub_n(t)\|_2^2 
   \nonumber\\ & + \underbrace{\frac{\beta}{2}\int_{\Tb^2}|\Arm(\ub_n(x,t)+\zb(x,t))|^2[\Arm(\ub_n(x,t)+\zb(x,t)):\Arm(\A\ub_n(x,t)+\A\zb(x,t))] \drm x}_{\widetilde{I}_1}
   %\nonumber\\ & = - b(\ub_n(t)+\zb(t), \ub_n(t)+\alpha_1\A\ub_n(t),\A\ub_n(t)) - b(\ub_n(t)+\zb(t), \zb(t)+\alpha_1\A\zb(t),\A\ub_n(t))
   %\nonumber\\ & \quad - b(\A\ub_n(t), \ub_n(t)+\zb(t),\ub_n(t)+\alpha_1\A\ub_n(t)) -  b(\A\ub_n(t), \ub_n(t)+\zb(t),\zb(t)+\alpha_1\A\zb(t))
   %\nonumber\\ & \quad - \frac{(\alpha_1+\alpha_2)}{2}\int_{\Tb^2}[\Arm(\ub_n(x,t)+\zb(x,t))]^2:\Arm(\A\ub_n(x,t)) \drm x 
   %\nonumber\\ & \quad + \frac{\beta}{2}\int_{\Tb^2}|\Arm(\ub_n(x,t)+\zb(x,t))|^2[\Arm(\ub_n(x,t)+\zb(x,t)):\Arm(\A\zb(x,t))] \drm x
   %\nonumber\\ & \quad + ( \fb(t), \A\ub_n(t))  +  (\zb(t),\A\ub_n(t)) + (\nu-\alpha_1)(\A\zb(t),\A\ub_n(t))
   \nonumber\\ & = - \underbrace{\alpha_1 b(\A(\ub_n(t)+\zb(t)),\ub_n(t)+\zb(t),\A(\ub_n(t)+\zb(t)))}_{\widetilde{I}_2} - \underbrace{b(\zb(t)-\alpha_1 \A \zb(t),\ub_n(t),\A\ub_n(t))}_{\widetilde{I}_3}
   \nonumber\\ & \quad + \underbrace{\alpha_1 b(\A\zb(t),\ub_n(t)+\zb(t),\A\zb(t))}_{\widetilde{I}_4} +\underbrace{\alpha_1 b(\A\zb(t),\zb(t),\A\ub_n(t))}_{\widetilde{I}_5} - \underbrace{b(\ub_n(t)+\zb(t), \Upsilon(\zb(t)),\A\ub_n(t))}_{\widetilde{I}_6}
   \nonumber\\ & \quad - \underbrace{\frac{(\alpha_1+\alpha_2)}{2}\int_{\Tb^2}[\Arm(\ub_n(x,t)+\zb(x,t))]^2:\Arm(\A\ub_n(x,t)) \drm x}_{\widetilde{I}_7} 
   \nonumber\\ & \quad + \underbrace{\frac{\beta}{2}\int_{\Tb^2}|\Arm(\ub_n(x,t)+\zb(x,t))|^2[\Arm(\ub_n(x,t)+\zb(x,t)):\Arm(\A\zb(x,t))] \drm x}_{\widetilde{I}_8}
   \nonumber\\ & \quad + \underbrace{[( \fb(t), \A\ub_n(t))  +  (\zb(t) +   (\theta+\alpha_1-\nu)\A\zb(t) + \theta\alpha_1 \A^2\zb(t) ,\A\ub_n(t))]}_{\widetilde{I}_9},\label{EE-Aub1}
\end{align}
for a.e. $t\in[0,T]$. Let us estimate each term on the right hand side of \eqref{EE-Aub1} using integration by parts, and H\"older's, Sobolev and Young's inequalities as follows:
\begin{align}
    % & \frac{\beta}{2}\int_{\Tb^2}|\Arm(\ub_n(x)+\zb(x))|^2[\Arm(\ub_n(x)+\zb(x)):\Arm(\A\ub_n(x)+\A\zb(x))] \drm x
   % \nonumber\\ & = \frac{\beta}{2}\int_{\Tb^2}|\Arm(\ub_n(x)+\zb(x))|^2|\nabla\Arm(\ub_n(x)+\zb(x))|^2\drm x + \beta\sum\limits_{m=1}^2\int_{\Tb^2}\left[|\Arm(\ub_n(x)+\zb(x))\cdot\partial_m\Arm(\ub_n(x)+\zb(x))|^2\right]\drm x
    % \nonumber\\ 
    \widetilde{I}_1 & = \frac{\beta}{2}\int_{\Tb^2}|\Arm(\ub_n(x)+\zb(x))|^2|\nabla\Arm(\ub_n(x)+\zb(x))|^2\drm x + \frac{\beta}{4} \int_{\mathbb{T}^2} |\nabla\left(|\Arm(\ub_n(x)+\zb(x))|^2\right)|^2  \drm x, \label{EE-Aub2}\\
  %  \nonumber\\
   % & |\alpha_1 b(\A(\ub_n+\zb),\ub_n+\zb,\A(\ub_n+\zb))|
    %\nonumber\\ 
   |\widetilde{I}_2| & = \left|\frac{\alpha_1}{2}\int_{\Tb^2} [\A(\ub_n(x)+\zb(x))\cdot \Arm(\ub_n(x)+\zb(x))]\cdot \A (\ub_n(x)+\zb(x))  \right|
    \nonumber\\ & = \left|\frac{\alpha_1}{2}\int_{\Tb^2} [\diver\Arm(\ub_n(x)+\zb(x))\cdot \Arm(\ub_n(x)+\zb(x))]\cdot \A (\ub_n(x)+\zb(x))  \right|
    \nonumber\\ & \leq \frac{\beta}{8}\int_{\Tb^2}|\Arm(\ub_n(x)+\zb(x))|^2|\nabla\Arm(\ub_n(x)+\zb(x))|^2\drm x + C \|\A(\ub_n+\zb)\|_2^2
    \nonumber\\ & \leq \frac{\beta}{8}\int_{\Tb^2}|\Arm(\ub_n(x)+\zb(x))|^2|\nabla\Arm(\ub_n(x)+\zb(x))|^2\drm x + C \|\A\ub_n\|_2^2 + C\|\zb\|^2_{\dot{\Hb}^3}
    \nonumber\\ & \leq \frac{\beta}{8}\int_{\Tb^2}|\Arm(\ub_n(x)+\zb(x))|^2|\nabla\Arm(\ub_n(x)+\zb(x))|^2\drm x + C \|\A\ub_n\|_2^2 + C\|\zb\|^4_{\dot{\Hb}^3} + C,\label{EE-Aub3}\\
 %   \nonumber\\
% &   |b(\zb-\alpha_1\A\zb,\ub,\A\ub)|
%  \nonumber\\
|\widetilde{I}_3| & \leq  \|\zb-\alpha_1\A\zb\|_{4}\|\nabla\ub_n\|_{4}\|\A\ub_n\|_2 \leq  C \|\zb\|_{\dot{\Hb}^3}\|\A\ub_n\|_2^2, \label{EE-Aub4}\\
%   \nonumber\\ 
% &  |\alpha_1 b(\A\zb,\ub_n+\zb,\A\zb)|
% \nonumber\\
|\widetilde{I}_4| & \leq C \|\nabla(\ub_n+\zb)\|_2 \|\A\zb\|_4^2 \leq C \|\nabla\ub_n\|^2_2 + C \|\zb\|^4_{\dot{\Hb}^3} + C, %\|\zb\|^3_{\dot{\Hb}^3}
 \label{EE-Aub5}\\
% \nonumber\\
% &  |\alpha_1 b(\A\zb,\zb,\A\ub_n)|
% \nonumber\\ 
|\widetilde{I}_5| & \leq C \|\A\zb\|_2\|\nabla\zb\|_{\infty}\|\A\ub_n\|_2 \leq C \|\zb\|^2_{\dot{\Hb}^3} \|\A\ub_n\|_2 \leq C \|\A\ub_n\|^2_2 +  C \|\zb\|^4_{\dot{\Hb}^3},\label{EE-Aub6}\\
% \nonumber\\ 
% &|b(\ub_n+\zb, \zb+\alpha_1\A\zb,\A\ub_n)|
% \nonumber\\
|\widetilde{I}_6| & \leq C \|\ub_n+\zb\|_{\infty} \|\zb\|_{\dot{\Hb}^3}\|\A\ub_n\|_2 \leq C  \|\zb\|_{\dot{\Hb}^3}\|\A\ub_n\|^2_2 + C  \|\zb\|^2_{\dot{\Hb}^3}\|\A\ub_n\|_2 \leq C [1+ \|\zb\|_{\dot{\Hb}^3}]\|\A\ub_n\|^2_2 + C  \|\zb\|^4_{\dot{\Hb}^3}, \label{EE-Aub7}\\
% \nonumber\\ 
 %& \left|\frac{(\alpha_1+\alpha_2)}{2}\int_{\Tb^2}[\Arm(\ub_n(x)+\zb(x))]^2:\Arm(\A\ub_n(x)) \drm x \right|
% \nonumber\\ 
|\widetilde{I}_7| & \leq C \int_{\Tb^2}|\Arm(\ub_n(x)+\zb(x))||\nabla\Arm(\ub_n(x)+\zb(x))||\A\ub_n(x)| \drm x 
 \nonumber\\ & \leq \frac{\beta}{4}\int_{\Tb^2}|\Arm(\ub_n(x)+\zb(x))|^2|\nabla\Arm(\ub_n(x)+\zb(x))|^2\drm x + C\|\A\ub_n\|^2_2, \label{EE-Aub8}\\
% \nonumber\\ 
% & \left| \frac{\beta}{2}\int_{\Tb^2}|\Arm(\ub_n(x)+\zb(x))|^2[\Arm(\ub_n(x)+\zb(x)):\Arm(\A\zb(x))] \drm x \right|
% \nonumber\\ 
 |\widetilde{I}_8| & \leq C \|\Arm(\ub_n+\zb)\|_{12} \|\Arm(\ub_n+\zb)\|_{3} \|\zb\|_{\dot{\Hb}^3}
  \nonumber\\ & \leq C \||\Arm(\ub_n+\zb)|^2\|^{\frac12}_{6} \|\A(\ub_n+\zb)\|_{2} \|\zb\|_{\dot{\Hb}^3}
   \nonumber\\ & \leq C \||\Arm(\ub_n+\zb)|^2\|^{\frac12}_{\Hb^1} \|\A(\ub_n+\zb)\|_{2} \|\zb\|_{\dot{\Hb}^3}
   \nonumber\\ & \leq C \|\Arm(\ub_n)\|_{4} \|\A(\ub_n+\zb)\|_{2} \|\zb\|_{\dot{\Hb}^3} + C \left(\int_{\mathbb{T}^2} |\nabla\left(|\Arm(\ub_n(x))|^2\right)|^2   \drm x\right)^{\frac14} \|\A(\ub_n+\zb)\|_{2} \|\zb\|_{\dot{\Hb}^3}
   \nonumber\\ & \leq \frac{\beta}{8} \int_{\mathbb{T}^2} |\nabla\left(|\Arm(\ub_n(x)+\zb(x))|^2\right)|^2  \drm x + C \|\Arm(\ub_n)\|_{4}^4 + C \|\A\ub_n\|_2^2 + C \|\A\zb\|_2^2 +  C \|\zb\|_{\dot{\Hb}^3}^4
   \nonumber\\ & \leq \frac{\beta}{8} \int_{\mathbb{T}^2} |\nabla\left(|\Arm(\ub_n(x)+\zb(x))|^2\right)|^2  \drm x + C \|\Arm(\ub_n+\zb)\|_{4}^4 + C \|\A\ub_n\|_2^2  +  C \|\zb\|_{\dot{\Hb}^3}^4 +C,\label{EE-Aub9}\\
%   \nonumber\\
  %&|( \fb, \A\ub_n)  +  (\zb,\A\ub_n) + (\nu-\alpha_1)(\A\zb,\A\ub_n)|
 % \nonumber\\ 
 |\widetilde{I}_9| & \leq C [\|\fb\|_2+\|\zb\|_{\dot{\Hb}^4}]\|\A\ub_n\|_2 \leq C [\|\fb\|_2^2+\|\zb\|^4_{\dot{\Hb}^4}+1]+ C \|\A\ub_n\|^2_2.\label{EE-Aub10}
   \end{align}
Combining \eqref{EE-Aub1}-\eqref{EE-Aub10}, we obtain
\begin{align}
   &    \frac{\drm}{\drm t}\bigg[\|\nabla\ub_n(t)\|_2^2 + \alpha_1 \|\A\ub_n(t)\|_2^2\bigg] + 2 \nu \|\A\ub_n(t)\|_2^2 
   \nonumber\\ & + \frac{\beta}{2}\int_{\Tb^2}|\Arm(\ub_n(x,t)+\zb(x,t))|^2|\nabla\Arm(\ub_n(x,t)+\zb(x,t))|^2\drm x + \frac{\beta}{4} \int_{\mathbb{T}^2} |\nabla\left(|\Arm(\ub_n(x,t)+\zb(x,t))|^2\right)|^2  \drm x
    \nonumber\\ & \leq C [ \|\Arm(\ub_n(t) +\zb(t))\|_{4}^4 + \|\fb(t)\|_2^2+\|\zb(t)\|^4_{\dot{\Hb}^4}+1] + C[1+\|\zb(t)\|_{\dot{\Hb}^3}]\bigg[\|\nabla\ub_n(t)\|_2^2 + \alpha_1 \|\A\ub_n(t)\|_2^2\bigg],
\end{align}
for a.e. $t\in[0,T]$, which, by an application of the Gronwall Lemma, gives
\begin{align}
    &   \|\nabla\ub_n(t)\|_2^2 + \alpha_1 \|\A\ub_n(t)\|_2^2 + 2\nu \int_{0}^{t} \|\A\ub_n(s)\|_2^2 \drm s + \frac{\beta}{4}\int_{0}^{t}\int_{\Tb^2}|\nabla\left(|\Arm(\ub_n(x,s)+\zb(x,s))|^2\right)|^2  \drm x\drm s
    \nonumber\\ & \leq \bigg[\|\nabla\ub_n(0)\|_2^2 + \alpha_1 \|\A\ub_n(0)\|_2^2 +  C \int_{0}^{t}\{ \|\Arm(\ub_n(s) +\zb(s))\|_{4}^4 + \|\fb(s)\|_2^2+\|\zb(s)\|^4_{\dot{\Hb}^4}+1\}\drm s\bigg] e^{C\int_{0}^{t}\{1+\|\zb(s)\|_{\dot{\Hb}^3}\}\drm s}
    \nonumber\\ & \leq C\bigg[\|\A\vb_0\|_{2}^2 +   \int_{0}^{t}\{ \|\Arm(\ub_n(s) +\zb(s))\|_{4}^4 + \|\fb(s)\|_2^2+\|\zb(s)\|^4_{\dot{\Hb}^4}+1\}\drm s\bigg] e^{C\int_{0}^{t}\{1+\|\zb(s)\|_{\dot{\Hb}^3}\}\drm s}
    \nonumber\\ & \leq C\bigg[\|\A\vb_0\|_{2}^2 + K_1(t)  + \int_{0}^{t}\{  \|\fb(s)\|_2^2+\|\zb(s)\|^4_{\dot{\Hb}^4}+1\}\drm s\bigg] e^{C\int_{0}^{t}\{1+\|\zb(s)\|_{\dot{\Hb}^3}\}\drm s}:= K_2(t),\label{EE-Aub11}
\end{align}
for all $t\in[0,T]$. 
\iffalse In addition, using the continuous embedding ${\Hb}^1(\mathbb{T}^2)\hookrightarrow \mathbb{L}^6(\mathbb{T}^2)$, we get
\begin{align}\label{EE-Aub12}
  \|\Arm(\ub_n+\zb)\|^4_{12}   &= \||\Arm(\ub_n+\zb)|^2\|^2_6   \leq C	 \||\Arm(\ub_n+\zb)|^2\|^2_{{\Hb}^1}  
= C  \|\Arm(\ub_n+\zb)\|^4_{4}  + C  \int_{\mathbb{T}^2} |\nabla\left(|\Arm(\ub_n(x)+\zb(x))|^2\right)|^2   \drm x.
\end{align}\fi 
%Using the fact that $\fb\in\mathrm{L}^{2}(0,T;\dot{\Lb}^{2}(\mathbb{T}^2))$, $\zb\in\mathrm{L}^{\infty}(0,T;\dot{\Hb}^{3}(\mathbb{T}^2))$ and $\ub_n+\zb\in \mathrm{L}^{4}(0,T; {\mathbb{W}}^{1,4}(\mathbb{T}^2))$ (see \eqref{EE-ub9}), 
In view of \eqref{pathwise-regularity} and \eqref{EE-ub9}, we have from \eqref{EE-Aub11} that
		\begin{align}\label{EE-Aub13}
			\{\ub^n\}_{n\in\N} \text{ is a bounded sequence in }\mathrm{L}^{\infty}(0,T;\Drm(\A)).%,\;\;  \Pb\text{-a.s.}%\cap\mathrm{L}^{4}(0,T; {\mathbb{W}}^{1,12}(\mathbb{T}^2)).
		\end{align}
%\vskip 2mm \noindent \textbf{Step IV.} \textit{
%\subsubsection{A priori estimates for \texorpdfstring{$\partial_t\ub_n$}{}.} 
\vskip 2mm
\noindent 
\textbf{\underline{Step III}:} Taking the inner product with $\partial_t\ub_n$ in $\eqref{Finite-dimensional-approx}_1$, we obtain
\begin{align}\label{EE-ubt1}
    & \|\partial_t\ub_n(t)\|_2^2 +\alpha_1\|\partial_t\nabla\ub_n(t)\|_2^2  +\frac{\nu}{2}\frac{\drm}{\drm t}  \|\nabla\ub_n(t)\|_2^2  
  %  \nonumber\\ &  = - b(\ub_n(t)+\zb(t), \ub_n(t)  +\alpha_1\A \ub_n(t), \partial_t\ub_n(t)) - b(\ub_n(t)+\zb(t), \zb(t) +\alpha_1\A \zb(t), \partial_t\ub_n(t)) 
%\nonumber \\ & \quad - b(\partial_t\ub_n(t),\ub_n(t)+\zb(t), \ub_n(t)+\alpha_1\A\ub_n(t)) - b(\partial_t\ub_n(t), \ub_n(t)+\zb(t), \zb(t)+\alpha_1\A\zb(t) )
%\nonumber\\ & 
%\quad +(\alpha_1+\alpha_2)(\diver[\Arm(\ub_n(t)+\zb(t))^2], \partial_t\ub_n(t)) 
% + \beta (\diver[|\Arm(\ub_n(t)+\zb(t))|^2\Arm(\ub_n(t)+\zb(t))],\partial_t\ub_n(t)) 
%\nonumber \\& \quad + ( \fb(t)  +  \zb(t) + (\nu-\alpha_1)\A\zb(t), \partial_t\ub_n(t))
\nonumber\\ &  =  \underbrace{b(\ub_n(t)+\zb(t), \partial_t\ub_n(t), \Upsilon(\ub_n(t)+\zb_n(t)))}_{\widetilde{I}_1}
  - \underbrace{ b(\partial_t\ub_n(t),\ub_n(t)+\zb(t), \Upsilon(\ub_n(t)+\zb_n(t))) }_{\widetilde{I}_2}
\nonumber\\ & 
\quad -\underbrace{(\alpha_1+\alpha_2)(\Arm(\ub_n(t)+\zb(t))^2, \partial_t\nabla\ub_n(t)) }_{\widetilde{I}_3} 
 - \underbrace{ \beta (|\Arm(\ub_n(t)+\zb(t))|^2\Arm(\ub_n(t)+\zb(t)),\partial_t\nabla\ub_n(t)) }_{\widetilde{I}_4}
\nonumber \\& \quad + \underbrace{( \fb(t)  +  \zb(t) +   (\theta+\alpha_1-\nu)\A\zb(t) + \theta\alpha_1 \A^2\zb(t), \partial_t\ub_n(t))}_{\widetilde{I}_5},
\end{align}
for a.e. $t\in[0,T]$. Let us estimate each term on the right hand side of \eqref{EE-ubt1} using integration by parts, and H\"older's, Sobolev and Young's inequalities as follows:
\begin{align}
    |\widetilde{I}_1| & \leq  \|\ub_n + \zb\|_{\infty} \|\partial_t\nabla\ub_n\|_2 \|\ub_n  + \alpha_1\A \ub_n + \zb + \alpha_1\A \zb\|_2
    \nonumber\\ & \leq C \|\partial_t\nabla\ub_n\|_2 \|\A(\ub_n + \zb)\|^2_2 \leq \frac{\alpha_1}{8} \|\partial_t\nabla\ub_n\|_2^2 + C  \|\A \ub_n \|^6_2 + C \|\zb\|^6_{\dot{\Hb}^2}+C,\label{EE-ubt2}\\
|\widetilde{I}_2| & \leq   \|\partial_t\ub_n\|_4 \|\nabla(\ub_n + \zb)\|_4 \|\ub_n + \alpha_1\A\ub_n + \zb +\alpha_1\A\zb \|_2 \nonumber\\
  & \leq C \|\partial_t\nabla\ub_n\|_2  \|\A(\ub_n + \zb)\|^2_2 \leq \frac{\alpha_1}{8} \|\partial_t\nabla\ub_n\|_2^2 + C  \|\A \ub_n \|^6_2 + C \|\zb\|^6_{\dot{\Hb}^2}+C,\label{EE-ubt3}\\
  |\widetilde{I}_3| & \leq \|\partial_t\nabla\ub_n\|_2  \|\Arm(\ub_n + \zb)\|^2_4 \leq \frac12 \|\partial_t\nabla\ub_n\|_2^2  + C \|\Arm(\ub_n + \zb)\|^4_4
  \nonumber\\ & \leq \frac{\alpha_1}{8} \|\partial_t\nabla\ub_n\|_2^2  + C \|\A(\ub_n + \zb)\|^4_2 \leq \frac{1}{2} \|\partial_t\nabla\ub_n\|_2^2 + C  \|\A \ub_n \|^6_2 + C \|\zb\|^6_{\dot{\Hb}^2}+C , 
  \label{EE-ubt4}\\
  |\widetilde{I}_4| & \leq \|\partial_t\nabla\ub_n\|_2  \|\Arm(\ub_n + \zb)\|^3_6 \leq \frac{\alpha_1}{8} \|\partial_t\nabla\ub_n\|_2^2  + C \|\Arm(\ub_n + \zb)\|^6_6
  \nonumber\\ & \leq \frac12 \|\partial_t\nabla\ub_n\|_2^2  + C \|\A(\ub_n + \zb)\|^6_2 \leq \frac{1}{2} \|\partial_t\nabla\ub_n\|_2^2 + C  \|\A \ub_n \|^6_2 + C \|\zb\|^6_{\dot{\Hb}^2}+C , 
  \label{EE-ubt5}\\
  |\widetilde{I}_5| &  \leq C [\|\fb\|_2^2+\|\zb\|^6_{\dot{\Hb}^4}+1]+ \frac12
  \|\partial_t\ub_n\|_2^2.\label{EE-ubt6}
\end{align}
Combining \eqref{EE-ubt1}-\eqref{EE-ubt6}, we obtain
\begin{align}
   &  \|\partial_t\ub_n(t)\|_2^2 +\alpha_1\|\partial_t\nabla\ub_n(t)\|_2^2+   \nu \frac{\drm}{\drm t} \|\nabla\ub_n(t)\|_2^2  
     \leq C [ \|\A\ub_n(t)\|_{2}^6 + \|\fb(t)\|_2^2+\|\zb(t)\|^6_{\dot{\Hb}^4}+1],
\end{align}
for a.e. $t\in[0,T]$, which, using \eqref{EE-Aub11}, gives 
\begin{align}
    &  \nu \|\nabla\ub_n(t)\|_2^2  +  \int_{0}^{t} \bigg[\|\partial_t\ub_n(s)\|_2^2 +\alpha_1\|\partial_t\nabla\ub_n(s)\|_2^2\bigg] \drm s 
    \nonumber\\ & \leq \nu \|\nabla\ub_n(0)\|_2^2  +  C \int_{0}^{t}\{ \|\A\ub_n(s) \|_{2}^6 + \|\fb(s)\|_2^2+\|\zb(s)\|^6_{\dot{\Hb}^4}+1\}\drm s
    \nonumber\\ & \leq C\bigg[\|\ub_0\|_{\Vb}^2 + t[K_2(t)]^3 +  \int_{0}^{t}\{ \|\fb(s)\|_2^2+\|\zb(s)\|^6_{\dot{\Hb}^4}+1\}\drm s\bigg] := K_3(t),\label{EE-ubt7}
\end{align}
for all $t\in[0,T].$  
%Using the fact that $\fb\in\mathrm{L}^{2}(0,T;\dot{\Lb}^{2}(\mathbb{T}^2))$ and $\zb\in\mathrm{L}^{\infty}(0,T;\dot{\Hb}^{3}(\mathbb{T}^2))$,
In view of \eqref{pathwise-regularity} and \eqref{EE-Aub11}, we have from \eqref{EE-ubt7} that
		\begin{align}\label{EE-ubt8}
			\{\partial_t \ub^n\}_{n\in\N} \text{ is a bounded sequence in }\mathrm{L}^{2}(0,T;\Vb).%, \;\; \Pb\text{-a.s.}
		\end{align}

Let us define 
\begin{align}\label{eqn-mathcal-G}
    \mathcal{G} (\ub, \zb) & :=\nu \A \ub  +\mathcal{P}[((\ub+\zb)\cdot \nabla)\Upsilon(\ub+\zb)]  + \mathcal{P}\left[\displaystyle\sum_{j=1}^2[\Upsilon(\ub+\zb)]^j\nabla (\ub^j+\zb^j)\right] 
    \nonumber \\ & \quad 
-(\alpha_1+\alpha_2) \mathcal{P}\bigg[\diver[\Arm(\ub+\zb)^2] \bigg] 
 -\beta  \mathcal{P}\bigg[\diver[|\Arm(\ub+\zb)|^2\Arm(\ub+\zb)] \bigg].
\end{align}
Note that, from \eqref{Finite-dimensional-approx}, we have 
\begin{align}
&\int_0^T\|\mathrm{P}_n\mathcal{G}(\ub_n(s), \zb(s))\|^2_{\Vb^{\prime}}\drm s 
\nonumber\\   & \leq \int_0^T\|\partial_t\ub_n(s)+\alpha_1\partial_t\A\ub_n(s)-\Prm_n\mathcal{P} \fb(s)  - \Prm_n\zb(s) -   (\theta+\alpha_1-\nu)\A\Prm_n\zb(s) - \theta\alpha_1 \Prm_n\A^2\zb(s)\|^2_{\Vb^{\prime}}\drm s
 %\nonumber\\ & \leq C\int_0^T\|\partial_t\ub_n(s)\|^2_{\Vb^{\prime}}\drm s + C\int_0^T\|\partial_t\A\ub_n(s)\|^2_{\Vb^{\prime}}\drm s + C\int_0^T\|\fb(s)\|^2_{\Vb^{\prime}}\drm s + C\int_0^T\|\zb(s)\|^2_{\Vb^{\prime}}\drm s +  C\int_0^T\|\A\zb(s)\|^2_{\Vb^{\prime}}\drm s
   \nonumber\\ & \leq C\int_0^T\|\partial_t\nabla\ub_n(s)\|^2_{2}\drm s + C\int_0^T\|\fb(s)\|^2_{2}\drm s + C\int_0^T\|\zb(s)\|^2_{\dot{\Hb}^3}\drm s
  \nonumber\\ &   \leq CK_3(T) + C\int_0^T\|\fb(s)\|^2_{2}\drm s + C\int_0^T\|\zb(s)\|^2_{\dot{\Hb}^3}\drm s,
\end{align}
 and \eqref{EE-ubt7} provides
\begin{align}\label{eqn-I+A_i_u}
 \int_0^T\|\partial_t\Upsilon(\ub_n(s))\|^2_{\Vb^{\prime}}\drm s 
  & \leq 2\int_0^T\|\partial_t\ub_n(s)\|^2_{\Vb^{\prime}}\drm s + 2\alpha_1^2\int_0^T\|\partial_t\A\ub_n(s)\|^2_{\Vb^{\prime}}\drm s 
  \nonumber\\ & \leq C\int_0^T\|\partial_t\nabla\ub_n(s)\|^2_{2}\drm s 
   \leq CK_3(T),
\end{align}
 which implies 
\begin{align}\label{EE-Gu}
    \{\Prm_n\mathcal{G}(\ub_n, \zb)\}_{n\in\N} \text{ and } \{\partial_t\Upsilon(\ub_n)\}_{n\in\N} \text{ is a bounded sequence in }\mathrm{L}^{2}(0,T;\Vb^{\prime}), \;\; \Pb\text{-a.s.}
\end{align}
 %\vskip 2mm \noindent \textbf{Step V.} \textit{Weak$^{\ast}$, weak and strong limits.}  
 \subsubsection{\texorpdfstring{Weak$^{\ast}$}{}, weak and strong limits.}
 Using \eqref{EE-ub10}, \eqref{EE-Aub13}, \eqref{EE-ubt8}, \eqref{EE-Gu} and the \textit{Banach-Alaoglu theorem}, we infer the existence of an element $\ub\in\mathrm{L}^{\infty}(0,T;\Drm(\A))$ with $\partial_t (\Irm+\alpha_1\A) \ub \in \mathrm{L}^{2}(0,T;\Vb^{\prime})$ and $\mathcal{G}_0\in \mathrm{L}^{2}(0,T;\Vb^{\prime})$ such that
		\begin{equation}
		\left.    \begin{aligned}
			\ub_n\xrightharpoonup{w^*}&\ \ub &&\text{ in }\ \ \ \ \	\mathrm{L}^{\infty}(0,T;\Drm(\A)),\label{S7}\\
			\ub_n\xrightharpoonup{w}&\ \ub   && \text{ in } \ \ \ \ \ \mathrm{L}^{2}(0,T;\Drm(\A))\\
            %\cap\mathrm{L}^{4}(0,T; {\mathbb{W}}^{1,12}(\mathbb{T}^2)),\\ %\label{S8}\\
			\partial_t \Upsilon( \ub_n) \xrightharpoonup{w}&\;  \partial_t \Upsilon(\ub)   && \text{ in }  \ \ \ \ \ \mathrm{L}^{2}(0,T;\Vb^{\prime}),%\label{S8d}
            \\
            \Prm_n\mathcal{G}(\ub_n, \zb)\xrightharpoonup{w}& \; \mathcal{G}_0   && \text{ in }  \ \ \ \ \ \mathrm{L}^{2}(0,T;\Vb^{\prime}),
		\end{aligned}\right\}
		\end{equation}
	along a subsequence (still denoted by the same symbol). In addition, since $\ub_n\in \mathrm{L}^{2}(0,T;\Drm(\A))$, $\partial_t\ub_n\in \mathrm{L}^{2}(0,T;\Hb)$, in view of continuous embeddings $\Drm(\A)\subset \Vb\subset \Hb$ with compact embedding $\Drm(\A)\subset \Vb$ and Aubin-Lions compactness lemma, we also have (along a subsequence, still denoted by the same symbol)
\begin{align}\label{Strong-Convergence-v}
		\ub_n \to & \ \ub \ \ \  \text{ in }\ \ \ \ \	\mathrm{L}^{2}(0,T;\Vb).%,\;\; \mbox{$\Pb$-a.s.}
\end{align}

%\vskip 2mm \noindent \textbf{Step VI.} \textit{Passing $n\to \infty$ in \eqref{Finite-dimensional-approx}.} 
\subsubsection{Passing \texorpdfstring{$n\to \infty$}{} in \texorpdfstring{\eqref{Finite-dimensional-approx}}{}.}
Let us start by  showing that $\mathcal{G}(\ub, \zb)$ is an element in $\mathrm{L}^{2}(0,T;\Vb^{\prime})$ for $\ub\in \mathrm{L}^{\infty}(0,T;\Drm(\A))$, $\fb\in\mathrm{L}^{2}(0,T;\dot{\Lb}^{2}(\mathbb{T}^2))$ and $\zb\in\mathrm{L}^{\infty}(0,T;\dot{\Hb}^{3}(\mathbb{T}^2))$. Since, for $\boldsymbol{\phi}\in\mathrm{L}^{2}(0,T;\Vb)$, we have
\begin{align}\label{Bound-G-ub}
    & \int_0^T\left<\mathcal{G}(\ub(t), \zb(t)),\boldsymbol{\phi}(t)\right> \drm t
%    \nonumber\\ & = \int_0^T\int_{\Tb^2} \bigg\{\nu \A \ub(t)  +((\ub(t)+\zb(t))\cdot \nabla)(\ub(t)  +\alpha_1\A \ub(t)) + ((\ub(t)+\zb(t))\cdot \nabla)( \zb(t) +\alpha_1\A \zb(t)) 
%\nonumber \\ & \quad  + \displaystyle\sum_{j=1}^2[\ub(t)+\alpha_1\A\ub(t)]^j\nabla (\ub^j(t)+\zb^j(t))
%  + \displaystyle\sum_{j=1}^2[\zb(t)+\alpha_1\A\zb(t)]^j\nabla (\ub^j(t)+\zb^j(t))
%-(\alpha_1+\alpha_2) \diver[\Arm(\ub(t)+\zb(t))^2] 
% \nonumber \\ & \quad -\beta  \diver[|\Arm(\ub(t)+\zb(t))|^2\Arm(\ub(t)+\zb(t))] \bigg\}\cdot \boldsymbol{\phi} \drm x \drm t
 \nonumber\\ & = \int_0^T\int_{\Tb^2} \bigg\{\nu \nabla \ub(t)  +((\ub(t)+\zb(t))\otimes(\Upsilon(\ub (t)+\zb(t)))) 
 -(\alpha_1+\alpha_2) \Arm(\ub(t)+\zb(t))^2 
  \nonumber \\ & \quad -\beta  |\Arm(\ub(t)+\zb(t))|^2\Arm(\ub(t)+\zb(t)) \bigg\}: \nabla\boldsymbol{\phi}(t) \drm x \drm t  
     +  
 \int_0^T b(\boldsymbol{\phi}(t),\ub(t)+\zb(t), \Upsilon(\ub (t)+\zb(t))) \drm t
 \nonumber\\ & \leq C\int_0^T\bigg[\|\ub(t)\|_{\Vb}+\|\ub(t)+\zb(t)\|_{\infty} \|\Upsilon(\ub (t)+\zb(t))\|_{2} 
 + \|\Arm(\ub(t)+\zb(t))\|^2_{4} 
   +  \|\Arm(\ub(t)+\zb(t))\|^3_6\bigg]
   \nonumber \\ & \quad \times \|\boldsymbol{\phi}(t)\|_{\Vb}  \drm t  
  +  
 \int_0^T \|\boldsymbol{\phi}(t)\|_{4} \|\nabla(\ub(t)+\zb(t))\|_{4} \|\Upsilon(\ub (t)+\zb(t))\|_{2} \drm t
  \nonumber\\ & \leq C\int_0^T\bigg[\|\ub(t)\|_{\Vb}+\|\A(\ub(t)+\zb(t))\|_{2}^2   
 +  \|\A(\ub(t)+\zb(t))\|^3_2\bigg] \|\boldsymbol{\phi}(t)\|_{\Vb}  \drm t  
 \nonumber\\ & \leq C\left(\int_0^T\bigg[\|\ub(t)\|^2_{\Vb}+\|\A(\ub(t)+\zb(t))\|_{2}^4   
 +  \|\A(\ub(t)+\zb(t))\|^6_2\bigg]\drm t \right)^{\frac12} \|\boldsymbol{\phi}\|_{\mathrm{L}^{2}(0,T;\Vb)},
\end{align}
it implies that $\mathcal{G}(\ub, \zb)\in\mathrm{L}^{2}(0,T;\Vb^{\prime}).$

Let us now show that $\mathcal{G}_0=\mathcal{G}(\ub, \zb)$. For $\boldsymbol{\phi}\in\mathrm{L}^{2}(0,T;\Vb)$, we consider
\begin{align}\label{G_n-weak-convergence}
	& \int_0^{T}\left\langle \mathcal{G}_n(\ub_n(t), \zb(t))- \mathcal{G}(\ub(t), \zb(t)) , \boldsymbol{\phi}(t) \right\rangle \drm t
	\nonumber\\ &  = \underbrace{\int_0^{T}\left\langle \mathcal{G}(\ub_n(t), \zb(t)) ,\Prm_n \boldsymbol{\phi}(t) -\boldsymbol{\phi}(t) \right\rangle \drm t }_{:=G_1(n)} + \underbrace{\int_0^{T}\left\langle \mathcal{G}(\ub_n(t), \zb(t))-  \mathcal{G}(\ub(t), \zb(t)) , \boldsymbol{\phi}(t) \right\rangle\drm t}_{:=G_2(n)}.
\end{align}

Since $\{\ub_n\}_{n\in\N}$ is a bounded sequence in $\mathrm{L}^{\infty}(0,T;\Drm(\A))$ (see \eqref{EE-Aub13}), a calculation similar to \eqref{Bound-G-ub} gives  that the sequence
\begin{center}
    $\left\{\int_0^T\|\mathcal{G}(\ub_n(t), \zb(t))\|^2_{\Vb^{\prime}}\drm t\right\}_{n\in\N} $ is a bounded sequence.
\end{center}
Also, we have 
 $\Prm_n\boldsymbol{\phi}\to \boldsymbol{\phi}$ in $\mathrm{L}^2(0,T;\Vb)$ therefore $\lim\limits_{n\to\infty}G_1(n)=0.$

Let us now show that $\lim\limits_{n\to\infty}G_2(n)=0.$  Define $\vb^n:=\ub_n-\ub$ and consider for $\boldsymbol{\phi}\in\mathrm{L}^{2}(0,T;\Vb)$
\begin{align}\label{eqn-Gu1-Gu2-1}
	&\left\langle \mathcal{G}(\ub_n, \zb)-  \mathcal{G}(\ub, \zb) , \boldsymbol{\phi} \right\rangle 
	\nonumber\\ & = \nu \underbrace{ (\nabla\vb^n,\nabla\boldsymbol{\phi} )}_{G_{21}}  - \underbrace{b(\vb^n ,\boldsymbol{\phi} , \Upsilon(\ub_n+\zb)  )}_{G_{22}} - \underbrace{b(\ub+\zb  ,\boldsymbol{\phi} , \Upsilon(\vb^n)  )}_{G_{23}} 
    %+ \underbrace{b(\boldsymbol{\phi} , \vb^n , \vb^n +\alpha_1\A\vb^n )}_{G_{24}}  
    + \underbrace{b(\boldsymbol{\phi} , \ub+\zb , \Upsilon(\vb^n) )}_{G_{24}} + \underbrace{b(\boldsymbol{\phi} , \vb^n , \Upsilon(\ub_n+\zb) ) }_{G_{25}} 
    \nonumber\\ & \quad + (\alpha_1+\alpha_2) \underbrace{\big[\left(\Arm(\vb^n)\Arm(\ub_n+\zb), \nabla\boldsymbol{\phi}\right)+ \left(\Arm(\ub+\zb)\Arm(\vb^n), \nabla\boldsymbol{\phi}\right)\big]}_{G_{26}}
	\nonumber\\ & \quad  + \beta\underbrace{\bigg[ \big( [\Arm(\vb^n ):\Arm(\ub_n+\zb)]\Arm(\ub_n+\zb )+   [\Arm(\ub+\zb ):\Arm(\vb^n )]\Arm(\ub_n+\zb )  + 
	 |\Arm(\ub+\zb )|^2\Arm(\vb^n ),\nabla\boldsymbol{\phi}\big) \bigg]}_{G_{27}}.
\end{align}
It is immediate from the weak convergence $\eqref{S7}_2$ that 
\begin{align}\label{eqn-Gu1-Gu2-2}
    \int_0^T [\nu G_{21}(t) - G_{23}(t) + G_{24}(t)]\drm t \to 0 \; \text{ as } n\to \infty.
\end{align}
Next using H\"older's, Agmon's and Sobolev's inequalities, and the following Gagliardo-Nirenberg inequality for $1\leq p <\infty$
\begin{align}\label{Gagliardo-Nirenberg-Lp}
\|\nabla\yb(t)\|_p  \leq C	\|\A(\yb(t))\|^{1-\frac1p}_2 \|\yb(t)\|_2^{\frac1p}, \text{	for all }\yb\in\Drm(\A),
\end{align} 
we have
\begin{align}\label{eqn-Gu1-Gu2-3}
    & \left|\int^T_0[-G_{22}(t)+G_{25}(t)]\drm t\right|
    \nonumber\\ & \leq \int^T_0 [\|\vb^n(t)\|_{\infty}\|\boldsymbol{\phi}(t)\|_{\Vb} + \|\boldsymbol{\phi}(t)\|_{4}  \|\nabla\vb^n(t)\|_4 ]\|\Upsilon(\ub_n(t)+\zb(t))\|_2 \drm t
    \nonumber\\ & \leq C \int^T_0 [\|\A\vb^n(t)\|^{\frac12}_2 \|\vb^n(t)\|^{\frac12}_{2} +  \|\A\vb^n(t)\|^{\frac34}_2 \|\vb^n(t)\|^{\frac14}_{2} ][\|\A\ub_n(t)\|_2 + \|\A\zb(t)\|_2]  \|\boldsymbol{\phi}(t)\|_{\Vb} \drm t
    \nonumber\\ & \leq C \int^T_0  \big[\|\A\ub_n(t)\|_2 + \|\A\zb(t)\|_2\big] \|\A\vb^n(t)\|^{\frac34}_2 \|\vb^n(t)\|^{\frac14}_{2}   \|\boldsymbol{\phi}(t)\|_{\Vb} \drm t
    \nonumber\\ & \leq C  \sup_{t\in[0,T]} \big[\|\A\ub_n(t)\|_2 + \|\A\zb(t)\|_2\big] \|\A\vb^n\|^{\frac34}_{\mathrm{L}^{2}(0,T;\Hb)} \|\vb^n\|^{\frac14}_{\mathrm{L}^{2}(0,T;\Hb)}   \|\boldsymbol{\phi}\|_{\mathrm{L}^{2}(0,T;\Vb)} 
    \nonumber\\ & \leq C \sup_{t\in[0,T]} \big[  [K_2(t)]^{\frac12} + \|\A\zb(t)\|_2\big] [K_2(T)]^{\frac38} \|\vb^n\|^{\frac14}_{\mathrm{L}^{2}(0,T;\Hb)}   \|\boldsymbol{\phi}\|_{\mathrm{L}^{2}(0,T;\Vb)} 
    \nonumber\\ & \to 0  \text{ as } n\to \infty,
\end{align}
where we have used the strong convergence obtained in \eqref{Strong-Convergence-v}. Finally, using H\"older's inequality and \eqref{Gagliardo-Nirenberg-Lp}, we obtain
\begin{align}\label{eqn-Gu1-Gu2-4}
     & \left|\int^T_0[(\alpha_1+\alpha_2)G_{26}(t)+\beta G_{27}(t)]\drm t\right|
    \nonumber\\ & \leq C \int^T_0\big[ \|\Arm(\ub_n(t)+\zb(t))\|_3+ \|\Arm(\ub(t)+\zb(t))\|_{3} + \|\Arm(\ub_n(t)+\zb(t))\|_6^2+ \|\Arm(\ub(t)+\zb(t))\|_{6}\|\Arm(\ub_n(t)+\zb(t))\|_6
    \nonumber\\ & \qquad + \|\Arm(\ub(t)+\zb(t))\|^2_6\big]
    \|\Arm(\vb^n(t))\|_{6}\|\boldsymbol{\phi}(t)\|_{\Vb}\drm t
    \nonumber\\ & \leq C \int^T_0\big[1+  \|\A\ub_n(t)\|_2^2  +  \|\A\zb(t)\|_{2}^2 \big]
    \|\A\vb^n(t)\|_{2}^{\frac56} \|\vb^n(t)\|_{2}^{\frac16}\|\boldsymbol{\phi}(t)\|_{\Vb}\drm t
    \nonumber\\ & \leq C \sup_{t\in[0,T]} \big[1+  K_2(t)  +  \|\A\zb(t)\|_{2}^2 \big] [K_2(T)]^{\frac{5}{12}} \|\vb^n\|^{\frac16}_{\mathrm{L}^{2}(0,T;\Hb)}   \|\boldsymbol{\phi}\|_{\mathrm{L}^{2}(0,T;\Vb)} 
    \nonumber\\ & \to 0  \text{ as } n\to \infty,
\end{align}
where we have used the strong convergence obtained in \eqref{Strong-Convergence-v}. Combing \eqref{eqn-Gu1-Gu2-1}-\eqref{eqn-Gu1-Gu2-2} and \eqref{eqn-Gu1-Gu2-3}-\eqref{eqn-Gu1-Gu2-4}, we infer that $\lim\limits_{n\to\infty}G_2(n)=0.$ Therefore, from $\eqref{S7}_4$, \eqref{G_n-weak-convergence} and uniqueness of weak limit, we obtain $\mathcal{G}_0=\mathcal{G}(\ub, \zb)$.
Also, we have $\fb_n\to \mathcal{P}\fb $ in $\mathrm{L}^{2}(0,T;\Hb)$. Therefore, on passing to limit as $n\to\infty$ in \eqref{Finite-dimensional-approx}, the limit $\ub(\cdot)$ satisfies:
	\begin{equation}\label{eqn-TGF-limit}
			\partial_t(\Irm+\alpha_1\A)\ub =-\mathcal{G}(\ub, \zb) +\mathcal{P}\fb + \zb +   (\theta+\alpha_1-\nu)\A\zb + \theta\alpha_1 \A^2\zb, \ \ \ \ \text{ in }  \ \ \ \  \mathrm{L}^{2}(0,T;\Vb^{\prime}).
	\end{equation}

Next, we notice that  \cite[Lemma 1.2, p.176]{temam2001navier}, $(\Irm+\alpha_1\A)^{\frac12}\ub\in\mathrm{L}^{2}(0,T;\Vb)$ and $\partial_t(\Irm+\alpha_1\A)^{\frac12}\ub\in \mathrm{L}^{2}(0,T;\Vb')$ imply $(\Irm+\alpha_1\A)^{\frac12}\ub\in \mathrm{C}([0,T];\Hb)$, the real-valued function $t\mapsto\|(\Irm+\alpha_1\A)^{\frac12}\ub(t)\|_{2}^2$ is absolutely continuous and the following equality is satisfied:
\begin{align}
    \frac12\frac{\drm}{\drm t}\|(\Irm+\alpha_1\A)^{\frac12}\ub(t)\|_{2}^2 =  \left<\partial_t (\Irm+\alpha_1\A)^{\frac12}\ub(t),(\Irm+\alpha_1\A)^{\frac12}\ub(t)  \right>=  \left<\partial_t (\Irm+\alpha_1\A)\ub(t),\ub(t)  \right>,
\end{align}
for a.e.  $t\in[0,T]$.  Hence, the initial condition in the Definition \ref{def-ranndom-TGF} also makes sense. 
%This also implies
%\begin{align}\label{EE1}
 %  \frac12 \frac{\drm}{\drm t}\bigg[\|\ub(t)\|_{2}^2+\alpha_1\|\nabla\ub(t)\|_{2}^2\bigg]&  =  \left<\partial_t (\Irm+\alpha_1\A)\ub(t),\ub(t)  \right>
%\nonumber\\ &        =  \left<-\mathcal{G} (\ub(t), \zb(t)) +\mathcal{P}\fb(t)  +  \zb(t) + (\nu-\alpha_1)\A\zb(t),\ub(t)  \right>,
%\end{align}
 %for a.e.  $t\in[0,T]$. Hence, the initial condition in the Definition \ref{def-WS-TGF} makes sense. 

%\vskip 2mm \noindent \textbf{Step VII.} \textit{Uniqueness:} 
\subsubsection{Uniqueness}
In order to obtain the uniqueness, we recall the following identity (form \cite[Appendix]{Busuioc+Iftimie_2004}) for any vector $\wb$ which will be used in the sequel:
\begin{align}\label{equivalance-equation}
    \frac12 \nabla\left[\alpha_1|\nabla\wb|^2-|\wb|^2\right]
      &=(\wb\cdot \nabla)\wb + \diver \big[-\alpha_1\underbrace{\{\wb\cdot \nabla\Arm(\wb)+L(\wb)^{T}\Arm(\wb)+\Arm(\wb)L(\wb)\}}_{=:\mathcal{N}(\wb)}-\alpha_2\Arm(\wb)^2\big]
    \nonumber\\ & \quad - \underbrace{\bigg[(\wb\cdot\nabla)(\wb-\alpha_1\Delta\wb) + \sum_{j=1}^2 (\wb-\alpha_1\Delta\wb)_j\nabla\wb_j - (\alpha_1+\alpha_2)\diver\Arm(\wb)^2\bigg]}_{=:\mathcal{M}(\wb)},
\end{align}
where $L(\wb)$ and $\Arm(\wb)$ are the matrix with entries $[L(\wb)]_{ij}=\partial_{j}\wb_i$ and $[\Arm(\wb)]_{ij}=\partial_{i}\wb_j+ \partial_{j}\wb_i$. We also recall Sobolev embedding (from \cite[p. 723-724]{Chemin+Xu_1997}) which says that there exists a constant $C_{\ast}>0$ (independent of $p$) such that for any $\wb\in\dot{\Hb}^{1}(\Tb^2)$
\begin{align}\label{Best-constant-sobolev-embedding}
    \|\wb\|_{p} \leq C_{\ast}p\|\wb\|_{\dot{\Hb}^1}.
\end{align}

Define $\mathfrak{U}=\ub_1-\ub_2$, where $\ub_1$ and $\ub_2$ are two solutions of system \eqref{random-third-grade-fluids-equations} in the sense of Definition \ref{def-ranndom-TGF}. Then $\mathfrak{U}\in\mathrm{C}([0,T];\Vb)\cap\mathrm{L}^{\infty}(0,T;\Drm(\A))$ and satisfies
\begin{equation}\label{Uni}
	\left\{
	\begin{aligned}
		\partial_t \Upsilon (\mathfrak{U}(t)) &= -\left[\mathcal{G}(\ub_1(t), \zb(t))-\mathcal{G}(\ub_2(t), \zb(t))\right], \\
		\mathfrak{U}(0)&= \textbf{0},
	\end{aligned}
	\right.
\end{equation}
for a.e. $t\in[0,T]$, in $\Vb'$.	Taking the inner product to the equation $\eqref{Uni}_1$ with $\mathfrak{U}$ and using integration by parts, \cite[Equation (27)]{Busuioc+Iftimie_2004}, \eqref{equivalance-equation}, \eqref{b0} and \cite[Equation (30)]{Busuioc+Iftimie_2004}, we obtain 
\begin{align}\label{Uni-1}
    & \frac{1}{2}\frac{\drm}{\drm t} \bigg[\|\mathfrak{U}\|^2_2 + \alpha_1 \|\nabla\mathfrak{U}\|^2_2\bigg] + \nu \|\nabla\mathfrak{U}\|_2^2 + \frac{\beta}{2}\int_{\mathbb{T}^2} |\Arm(\mathfrak{U}(x))|^2( |\Arm(\ub_1(x)+\zb(x))|^2+|\Arm(\ub_2(x)+\zb(x))|^2)\drm x
    \nonumber\\ & + \frac{\beta}{2}\int_{\mathbb{T}^2}( |\Arm(\ub_1(x)+\zb(x))|^2-|\Arm(\ub_2(x)+\zb(x))|^2)^2\drm x 
    \nonumber\\ & = -\left< \mathcal{M}(\ub_1+\zb)-\mathcal{M}(\ub_2+\zb) , \mathfrak{U} \right>
    \nonumber\\ & = -\left< ((\ub_1+\zb)\cdot \nabla)(\ub_1+\zb) - ((\ub_2+\zb)\cdot \nabla)(\ub_2+\zb) , \mathfrak{U} \right> +\alpha_1 \left< \diver [\mathcal{N}(\ub_1+\zb)-\mathcal{N}(\ub_2+\zb)] , \mathfrak{U} \right> 
    \nonumber\\ & \quad + \alpha_2 \left< \diver[\Arm(\ub_1+\zb)^2-\Arm(\ub_2+\zb)^2] , \mathfrak{U} \right>
    \nonumber\\ & = - \underbrace{b (\mathfrak{U},\ub_1+\zb, \mathfrak{U})}_{=:U_1}  - \underbrace{\frac{\alpha_1}{2} \left< \mathcal{N}(\ub_1+\zb)-\mathcal{N}(\ub_2+\zb) , \Arm (\mathfrak{U}) \right> }_{=:U_2}
      - \underbrace{\frac{\alpha_2}{2} \left(\Arm(\mathfrak{U})\Arm(\ub_1+\zb) + \Arm(\ub_2+\zb)\Arm(\mathfrak{U}), \Arm(\mathfrak{U})\right)}_{=:U_3}.
\end{align}

Let us write the term $U_2$ in more simple way by using integration by parts as follows (see  \cite[Equations (31)-(32)]{Busuioc+Iftimie_2004} for detailed calculation): 
\begin{align}\label{Uni-2}
    U_2 & = \frac{\alpha_1}{2} \int_{\Tb^2} [(\ub_1+\zb)\cdot \nabla\Arm(\ub_1+\zb)- (\ub_2+\zb)\cdot \nabla\Arm(\ub_2+\zb)]: \Arm(\mathfrak{U}) 
    \nonumber\\ &  \quad + \frac{\alpha_1}{2} \int_{\Tb^2} [L(\ub_1+\zb)^{T}\Arm(\ub_1+\zb)+\Arm(\ub_1+\zb)L(\ub_1+\zb) - L(\ub_2+\zb)^{T}\Arm(\ub_2+\zb)+\Arm(\ub_2+\zb)L(\ub_2+\zb)]:\Arm(\mathfrak{U})
    \nonumber\\ & = \underbrace{\frac{\alpha_1}{2} \int_{\Tb^2} [\mathfrak{U}\cdot \nabla\Arm(\ub_2+\zb)]: \Arm(\mathfrak{U}) }_{=:U_{21}}
      + \underbrace{\alpha_1 \int_{\Tb^2} [\Arm(\mathfrak{U})^2:L(\ub_1+\zb)+(\Arm(\ub_2+\zb)\Arm(\mathfrak{U})):L(\mathfrak{U})]}_{=:U_{22}}.
\end{align}
Using H\"older's, Sobolev's and Young's inequalities, we find
\begin{align}
|U_1| & \leq \|\nabla(\ub_1+\zb)\|_2 \|\mathfrak{U}\|_4^2 \leq C \|\nabla(\ub_1+\zb)\|_2 \|\nabla\mathfrak{U}\|_2^2,\label{Uni-3}
\\ 
    |U_3 + U_{22}| & \leq C \int_{\Tb^2}|\Arm(\mathfrak{U}(x))|(|\Arm(\ub_1(x)+\zb(x))|+|\Arm(\ub_2(x)+\zb(x))|) |\nabla\mathfrak{U}(x)| \drm x
    \nonumber\\ & \leq C\|\nabla\mathfrak{U}\|_{2}^2 + \frac{\beta}{12} \int_{\mathbb{T}^2} |\Arm(\mathfrak{U}(x))|^2 ( |\Arm(\ub_1(x)+\zb(x))|^2+|\Arm(\ub_2(x)+\zb(x))|^2)\drm x,\label{Uni-4}
\end{align}
where $\gamma>0$ will be chosen later. Finally, we estimate $U_{21}$ which is bit more technical than other terms. Let $\sigma\in(0,1)$ be arbitrary. Then, using H\"older's inequality, \eqref{Best-constant-sobolev-embedding} and Young's inequality, we get
\begin{align}\label{Uni-5}
    |U_{21}| & \leq C \int_{\Tb^2} |\mathfrak{U}(x)| |\nabla\Arm(\ub_2(x)+\zb(x))| |\Arm(\mathfrak{U}(x))| \drm x
    \nonumber \\ & = C \int_{\Tb^2} |\mathfrak{U}(x)| |\Arm(\mathfrak{U}(x))|^{\sigma} |\nabla\Arm(\ub_2(x)+\zb(x))| |\Arm(\mathfrak{U}(x))|^{1-\sigma} \drm x
    \nonumber\\ & \leq C \|\mathfrak{U}\|_{\frac{4}{1-\sigma}}\|\Arm(\mathfrak{U})\|^{\sigma}_{2}\|\nabla\Arm(\ub_2+\zb)\|_{2}\|\Arm(\mathfrak{U})\|^{1-\sigma}_{4}
    \nonumber\\ & \leq \frac{C}{1-\sigma}  \|\nabla\mathfrak{U}\|^{1+\sigma}_{2}\|\A(\ub_2+\zb)\|_{2}\|\Arm(\mathfrak{U})\|^{1-\sigma}_{4}
    \nonumber\\ & \leq \frac{\beta}{12} \|\Arm(\mathfrak{U})\|^4_{4} +  \frac{C}{1-\sigma}  \left(\frac{\beta}{3}\right)^{\frac{\sigma-1}{\sigma+3}} \|\nabla\mathfrak{U}\|^{\frac{4(1+\sigma)}{\sigma+3}}_{2}\|\A(\ub_2+\zb)\|^{\frac{4}{\sigma+3}}_{2}
    \nonumber\\ & \leq \frac{\beta}{6} \int_{\mathbb{T}^2} |\Arm(\mathfrak{U})|^2 ( |\Arm(\ub_1+\zb)|^2+|\Arm(\ub_2+\zb)|^2) +  \frac{C}{1-\sigma}  \left(\frac{\beta}{3}\right)^{\frac{\sigma-1}{\sigma+3}} \|\nabla\mathfrak{U}\|^{\frac{4(1+\sigma)}{\sigma+3}}_{2}\|\A(\ub_2+\zb)\|^{\frac{4}{\sigma+3}}_{2}.
\end{align}
Combining \eqref{Uni-1}-\eqref{Uni-5}, we infer
\begin{align}\label{Uni-6}
     & \frac{\drm}{\drm t} \bigg[\|\mathfrak{U}(t)\|^2_2 + \alpha_1 \|\nabla\mathfrak{U}(t)\|^2_2\bigg] 
    \nonumber\\ & \leq C [1+\|\nabla(\ub_1(t)+\zb(t))\|_2] \|\nabla\mathfrak{U}(t)\|_2^2 +  \frac{C}{1-\sigma}  \left(\frac{\beta}{3}\right)^{\frac{\sigma-1}{\sigma+3}} \|\nabla\mathfrak{U}(t)\|^{\frac{4(1+\sigma)}{\sigma+3}}_{2}\|\A(\ub_2(t)+\zb(t))\|^{\frac{4}{\sigma+3}}_{2}
\end{align}
for a.e. $t\in [0,T]$. In view of \eqref{EE-Aub11}, there exists a constant $\widehat{K}:= \widehat{K}(\alpha_1,\alpha_2,\beta,\ub_1,\ub_2,\zb,T) >0$ independent of $\sigma$ such that
\begin{align}\label{Uni-7}
    \frac{\drm \mathrm{\textbf{U}}(t)}{\drm t}   \leq \frac{\widehat{K}}{1-\sigma}\bigg[\mathrm{\textbf{U}}(t) + \{\mathrm{\textbf{U}}(t)\}^{\frac{2(1+\sigma)}{\sigma+3}}\bigg]
\end{align}
for a.e. $t\in[0,T]$, where $\mathrm{\textbf{U}}(t)=\|\mathfrak{U}(t)\|^2_2 + \alpha_1 \|\nabla\mathfrak{U}(t)\|^2_2$. We also have 
\begin{align}
    \frac{\drm \mathrm{\textbf{U}}(t)}{\drm t}= \frac{\drm}{\drm t} \bigg[\{\mathrm{\textbf{U}}(t)\}^{\frac{1-\sigma}{\sigma+3}}\bigg]^{\frac{\sigma+3}{1-\sigma}}= \frac{\sigma+3}{1-\sigma}\{\mathrm{\textbf{U}}(t)\}^{\frac{2(1+\sigma)}{\sigma+3}}\frac{\drm}{\drm t} \bigg[\{\mathrm{\textbf{U}}(t)\}^{\frac{1-\sigma}{\sigma+3}}\bigg],
\end{align}
for a.e. $t\in[0,T]$, which implies from \eqref{Uni-7} and $\mathrm{\textbf{U}}(t)\geq 0$ that
\begin{align}\label{Uni-8}
    \frac{\drm}{\drm t} \bigg[\{\mathrm{\textbf{U}}(t)\}^{\frac{1-\sigma}{\sigma+3}}\bigg]   \leq \frac{\widehat{K}}{\sigma+3}\bigg[\{\mathrm{\textbf{U}}(t)\}^{\frac{1-\sigma}{\sigma+3}} + 1\bigg],
\end{align}
for a.e. $t\in[0,T]$. Now, an application of Gronwall's inequality to \eqref{Uni-8} gives
\begin{align}
    %& \{\mathrm{\textbf{U}}(t)\}^{\frac{1-\sigma}{\sigma+3}} \leq \frac{\widehat{K}t}{\sigma+3}  e^{\frac{\widehat{K}t}{\sigma+3}} \nonumber \\
  %  \implies & 
    \mathrm{\textbf{U}}(t) \leq \left\{\frac{\widehat{K}t}{\sigma+3}\right\}^{\frac{\sigma+3}{1-\sigma}}e^{\widehat{K}t(1-\sigma)} \leq \left\{\frac{\widehat{K}t}{3}\right\}^{\frac{\sigma+3}{1-\sigma}}e^{\widehat{K}t(1-\sigma)},
\end{align}
for all $t\in[0,T]$. Let us denote $T_0=\frac{3}{2\widehat{K}}$, then we have
\begin{align}
    \mathrm{\textbf{U}}(t) \leq \left\{\frac12\right\}^{\frac{\sigma+3}{1-\sigma}}e^{\frac32}, \;\; \text{ for all } t\in [0,T_0].
\end{align}
If $\mathrm{\textbf{U}}(t)\geq \varepsilon$ for some $\varepsilon>0$. Then, one can always find $\sigma\in (0,1)$ close the $1$ such that 
\begin{align}
    \mathrm{\textbf{U}}(t) \leq \left\{\frac12\right\}^{\frac{\sigma+3}{1-\sigma}}e^{\frac32} \leq \frac{\varepsilon}{2}, \;\; \text{ for all } t\in [0,T_0],
\end{align}
which leads to a contradiction. 
Hence $\mathrm{\textbf{U}}(t)=0$ for all $t\in[0,T_0]$. We repeat the same arguments a finite number of times to
cover the whole interval $[0,T]$, so that the uniqueness is proved.

\subsubsection{More regular a priori estimate}\label{sec-more-regularity}
We will denote by $\mathfrak{D}^k\wb$ a linear combination of derivatives of order not greater than $k$ of the
components of $\wb$. The expression $\mathfrak{D}^k\wb_1\mathfrak{D}^l\wb_2$ denotes a product of terms of the type $\mathfrak{D}^k\wb_1$ and $\mathfrak{D}^l\wb_2$ or a sum of such terms.
\iffalse 
The following result has been taken from \cite{Benyi+Oh+Zhao_2025} which will be used in the sequel. Let us define $$\mathfrak{D}^{s}\equiv \mathfrak{D}^{s}_{\mathbb{T}^2}\equiv (-\Delta)^{\frac{s}{2}}$$ denote the
		 Fourier multiplier operators on $\mathbb{T}^2$ with multipliers $(\frac{2\pi |\boldsymbol{k}|}{L})^s$ and recall the fractional Leibniz rule from \cite{Benyi+Oh+Zhao_2025} (see also \cite{Li_2019}) for the operator $\mathfrak{D}$. For any $s>0$ and $f,g\in \Drm(\mathfrak{D}^s)$, we have 
			\begin{align}\label{Frac-Leib}
			\|\mathfrak{D}^{s}( fg)\|_{2} \leq C_{s}\bigg[\|\mathfrak{D}^s f\|_{2}\|g\|_{\infty}+ \|\mathfrak{D}^s g\|_{2}\|f\|_{\infty} \bigg],
		\end{align}
		where $C_{s}>0$ is a constant depending only on $s$. Also, note that, for $\vb\in\Drm(\A^{s})$ and $s>0$, we have $\mathfrak{D}^{s}\vb=\A^{s}\vb$, consequently
		\begin{align*}
			\|\mathfrak{D}^{s}\vb\|_{2}=\|\A^{s}\vb\|_{2}.
		\end{align*} 
        \fi 
		Let us recall an inequality from \cite[Proposition 1]{Busuioc_2002} which will be used in the sequel: let $\varepsilon_0>0$, then for any $\vb\in \dot{\Hb}^{1+\varepsilon_0}(\mathbb{T}^2)$
	\begin{align}\label{embedding-eps}
		\|\vb\|_{\infty} \leq \frac{C_{\ast}}{\sqrt{\varepsilon_0}}\|\vb\|_{\dot{{\Hb}}^{1+\varepsilon_0}},
	\end{align}
where $C_{\ast}$ is a universal constant independent of $\varepsilon_0$. \iffalse Let us also recall The Laplacian of a vector field $\vb$ is given by:
\begin{align}
 \Delta\vb=\nabla(\nabla\cdot\vb) - \nabla\times(\nabla\times\vb)  
\end{align}
\fi

Since $\{\wb_i\}_{i\in\mathbb{N}}$ is a sequence of eigenfunctions of the Stokes operator $\A$, it gives $\A^2\ub_n\in \Hb_n.$ Taking the inner product with $\A^{2}\ub_n$ in $\eqref{Finite-dimensional-approx}_1$, we obtain
\begin{align}
   &  \frac12 \frac{\drm}{\drm t}\bigg[\|\A\ub_n(t)\|_2^2 + \alpha_1 \|\A^{\frac32}\ub_n(t)\|_2^2\bigg] + \nu \|\A^{\frac32}\ub_n(t)\|_2^2 
   \nonumber\\ & =  \underbrace{-b(\ub_n(t)+\zb(t), \ub_n(t)+\zb(t),\A^2\ub_n(t))}_{:= \widehat{I}_1 } \underbrace{-\alpha_1 b(\ub_n(t)+\zb(t), \A(\ub_n(t)+\zb(t)),\A^2(\ub_n(t)+\zb(t)))}_{:= \widehat{I}_2 }
   \nonumber\\ & \quad  \underbrace{
   +\alpha_1 b(\ub_n(t)+\zb(t), \A(\ub_n(t)+\zb(t)),\A^2\zb(t))}_{:= \widehat{I}_3} - \underbrace{\alpha_1 b(\A^2\ub_n(t), \ub_n(t)+\zb(t), \A\ub_n(t)+\A\zb(t))}_{:= \widehat{I}_4 }
   \nonumber\\ & \quad  \underbrace{-\frac{(\alpha_1+\alpha_2)}{2}\int_{\Tb^2}[\Arm(\ub_n(x,t)+\zb(x,t))]^2:\Arm(\A^2\ub_n(x,t)) \drm x }_{:= \widehat{I}_5 }
   \nonumber\\ & \quad \underbrace{-\frac{\beta}{2}\int_{\Tb^2}|\Arm(\ub_n(x,t)+\zb(x,t))|^2[\Arm(\ub_n(x,t)+\zb(x,t)):\Arm(\A^2(\ub_n(x,t)+\zb(x,t)))] \drm x}_{:= \widehat{I}_6 }
   \nonumber\\ & \quad  \underbrace{{+}\frac{\beta}{2}\int_{\Tb^2}|\Arm(\ub_n(x,t)+\zb(x,t))|^2[\Arm(\ub_n(x,t)+\zb(x,t)):\Arm(\A^2\zb(x,t))] \drm x}_{:= \widehat{I}_7}
   \nonumber\\ & \quad + \underbrace{[( \fb(t), \A^2\ub_n(t))  +  (\zb(t) +   (\theta+\alpha_1-\nu)\A\zb(t) + \theta\alpha_1 \A^2\zb(t),\A^2\ub_n(t))]}_{:= \widehat{I}_8}\label{EE-A2ub1}.
\end{align}
Using integration by parts, let us rewrite each term of right hand side of \eqref{EE-A2ub1} one by one in terms of derivatives $\mathfrak{D}^k$ as follows:
\begin{align}
    \widehat{I}_1 &  = \int_{\Tb^2}\mathfrak{D}^{0}(\ub_n+\zb) \cdot \mathfrak{D}^2(\ub_n+\zb)\cdot \mathfrak{D}^3\ub_n \drm x + \int_{\Tb^2}\mathfrak{D}^{1}(\ub_n+\zb)\cdot\mathfrak{D}^1(\ub_n+\zb)\cdot \mathfrak{D}^3\ub_n \drm x
   % \nonumber\\ & = \int_{\Tb^2}\mathfrak{D}^{1}(\ub_n+\zb) \cdot \mathfrak{D}^2(\ub_n+\zb) \cdot \mathfrak{D}^3\ub_n \drm x
   ,\label{EE-A2ub2}\\
    \widehat{I}_2 &  = \int_{\Tb^2}\mathfrak{D}^{1}(\ub_n+\zb)\cdot \mathfrak{D}^3(\ub_n+\zb)  \cdot \mathfrak{D}^3(\ub_n+\zb) \drm x, \label{EE-A2ub3} \\
    \widehat{I}_3 &  = \int_{\Tb^2}\mathfrak{D}^{0}(\ub_n+\zb) \cdot \mathfrak{D}^3(\ub_n+\zb)\cdot \mathfrak{D}^4\zb \drm x, \label{EE-A2ub4}\\
    \widehat{I}_4 &  = \int_{\Tb^2}\mathfrak{D}^{3}\ub_n \cdot \mathfrak{D}^2(\ub_n+\zb) \cdot \mathfrak{D}^2(\ub_n+\zb) \drm x + \int_{\Tb^2}\mathfrak{D}^{3}\ub_n \cdot \mathfrak{D}^1(\ub_n+\zb) \cdot \mathfrak{D}^3(\ub_n+\zb) \drm x,\label{EE-A2ub5} \\
    \widehat{I}_5 &  = \int_{\Tb^2}\mathfrak{D}^{1}(\ub_n+\zb) \cdot \mathfrak{D}^2(\ub_n+\zb) \cdot \mathfrak{D}^3\ub_n \drm x, \label{EE-A2ub6} \\
    \widehat{I}_7 &  = \underbrace{\int_{\Tb^2} \mathfrak{D}^{1}(\ub_n+\zb) \cdot  \mathfrak{D}^{1}(\ub_n+\zb) \cdot \mathfrak{D}^3(\ub_n+\zb) \cdot \mathfrak{D}^3\zb \drm x}_{\widehat{I}_{71}} + \underbrace{\int_{\Tb^2} \mathfrak{D}^{1}(\ub_n+\zb) \cdot  \mathfrak{D}^{2}(\ub_n+\zb) \cdot \mathfrak{D}^2(\ub_n+\zb) \cdot \mathfrak{D}^3\zb \drm x}_{\widehat{I}_{72}}, \label{EE-A2ub7} \\ 
    \widehat{I}_6 & = - \frac{\beta}{2} \int_{\Tb^2} |\Delta\Arm(\ub_n+\zb)|^2 |\Arm(\ub_n+\zb)|^2 \drm x - 
    \beta \int_{\Tb^2} [\Arm(\ub_n+\zb):\Delta\Arm(\ub_n+\zb)]^2 \drm x 
    \nonumber\\ & \quad + \underbrace{\int_{\Tb^2} \mathfrak{D}^{1}(\ub_n+\zb) \cdot  \mathfrak{D}^{2}(\ub_n+\zb) \cdot \mathfrak{D}^2(\ub_n+\zb) \cdot \mathfrak{D}^3(\ub_n+\zb) \drm x}_{\widehat{I}_{61}}. \label{EE-A2ub8}
\end{align}
Equalities \eqref{EE-A2ub1}-\eqref{EE-A2ub8} provide that
\begin{align}
   &  \frac12 \frac{\drm}{\drm t}\bigg[\|\A\ub_n\|_2^2 + \alpha_1 \|\A^{\frac32}\ub_n\|_2^2\bigg] + \nu \|\A^{\frac32}\ub_n\|_2^2 
   \nonumber\\ & \leq \widehat{I}_8 + \widehat{I}_{71} + \widehat{I}_{72} + \widehat{I}_{61} + \underbrace{\int_{\Tb^2}\mathfrak{D}^{1}(\ub_n+\zb)\cdot\mathfrak{D}^3(\ub_n+\zb)\cdot \mathfrak{D}^3(\ub_n+\zb) \drm x}_{\widehat{I}_9} + \underbrace{\int_{\Tb^2}\mathfrak{D}^{1}(\ub_n+\zb)\cdot\mathfrak{D}^3(\ub_n+\zb)\cdot \mathfrak{D}^4\zb \drm x}_{\widehat{I}_{10}}
   \nonumber\\ & \quad + \underbrace{\int_{\Tb^2}\mathfrak{D}^{3}(\ub_n+\zb) \cdot \mathfrak{D}^2(\ub_n+\zb) \cdot \mathfrak{D}^2(\ub_n+\zb) \drm x}_{\widehat{I}_{11}} + \underbrace{\int_{\Tb^2}\mathfrak{D}^{3}\zb \cdot \mathfrak{D}^2(\ub_n+\zb) \cdot \mathfrak{D}^2(\ub_n+\zb) \drm x}_{\widehat{I}_{12}}.\label{EE-A2ub9}
\end{align}
Let us now estimate each term of right hand side of \eqref{EE-A2ub9} using H\"older's, Sobolev's, \eqref{embedding-eps}, interpolation and Young's inequalities, as follows:
\begin{align}
    |\widehat{I}_8| & \leq C \|\fb\|^2_{\dot{\Hb}^1} + \|\zb\|^2_{\dot{\Hb}^5} + \frac{\nu}{2} \|\A^{\frac32}\ub_n\|_2^2,\label{EE-A2ub10}\\
    |\widehat{I}_{71}| & \leq C \|\mathfrak{D}^1(\ub_n+\zb)\|^2_{\infty}\|\mathfrak{D}^3(\ub_n+\zb)\|_{2}\|\mathfrak{D}^3\zb\|_{2}\nonumber\\
    & \leq \frac{C}{\sqrt{\varepsilon_0}} \|\ub_n+\zb\|_{\dot{\Hb}^{2+\varepsilon_0}} \|\ub_n+\zb\|^2_{\dot{\Hb}^3}\|\zb\|_{\dot{\Hb}^3}
    %\nonumber\\ &
    \leq \frac{C}{\sqrt{\varepsilon_0}} \|\ub_n+\zb\|_{\dot{\Hb}^{2}}^{1-\varepsilon_0} \|\ub_n+\zb\|^{2+\varepsilon_0}_{\dot{\Hb}^3}\|\zb\|_{\dot{\Hb}^3},\label{EE-A2ub11}\\
    |\widehat{I}_{72}| & \leq C \|\mathfrak{D}^1(\ub_n+\zb)\|_{\infty}\|\mathfrak{D}^2(\ub_n+\zb)\|^2_{4}\|\mathfrak{D}^3\zb\|_{2}\nonumber\\
    & \leq \frac{C}{\sqrt{\varepsilon_0}} \|\ub_n+\zb\|_{\dot{\Hb}^{2+\varepsilon_0}} \|\ub_n+\zb\|^2_{\dot{\Hb}^3}\|\zb\|_{\dot{\Hb}^3}
    %\nonumber\\ &
    \leq \frac{C}{\sqrt{\varepsilon_0}} \|\ub_n+\zb\|_{\dot{\Hb}^{2}}^{1-\varepsilon_0} \|\ub_n+\zb\|^{2+\varepsilon_0}_{\dot{\Hb}^3}\|\zb\|_{\dot{\Hb}^3},\label{EE-A2ub12}\\
    |\widehat{I}_{61}| & \leq C \|\mathfrak{D}^1(\ub_n+\zb)\|_{\infty}\|\mathfrak{D}^2(\ub_n+\zb)\|^2_{4}\|\mathfrak{D}^3(\ub_n+\zb)\|_{2}\nonumber\\
    & \leq \frac{C}{\sqrt{\varepsilon_0}} \|\ub_n+\zb\|_{\dot{\Hb}^{2+\varepsilon_0}} \|\ub_n+\zb\|^2_{\dot{\Hb}^{\frac52}}\|\ub_n+\zb\|_{\dot{\Hb}^3} 
    %\nonumber\\ & 
    \leq \frac{C}{\sqrt{\varepsilon_0}} \|\ub_n+\zb\|_{\dot{\Hb}^{2}}^{2-\varepsilon_0} \|\ub_n+\zb\|^{2+\varepsilon_0}_{\dot{\Hb}^3},\label{EE-A2ub13}\\
    |\widehat{I}_{9}| & \leq C \|\mathfrak{D}^1(\ub_n+\zb)\|_{\infty}\|\mathfrak{D}^3(\ub_n+\zb)\|^2_{2}\nonumber\\
    & \leq \frac{C}{\sqrt{\varepsilon_0}} \|\ub_n+\zb\|_{\dot{\Hb}^{2+\varepsilon_0}} \|\ub_n+\zb\|^2_{\dot{\Hb}^3}
    %\nonumber\\ & 
    \leq \frac{C}{\sqrt{\varepsilon_0}} \|\ub_n+\zb\|_{\dot{\Hb}^{2}}^{1-\varepsilon_0} \|\ub_n+\zb\|^{2+\varepsilon_0}_{\dot{\Hb}^3},\label{EE-A2ub14}\\
    %|\widehat{I}_{10}| & \leq C \|\mathfrak{D}^1(\ub_n+\zb)\|_{\infty}\|\mathfrak{D}^3(\ub_n+\zb)\|_{2}\|\mathfrak{D}^4\zb\|_{2}\nonumber\\
    %& \leq \frac{C}{\sqrt{\varepsilon_0}} \|\ub_n+\zb\|_{\dot{\Hb}^{2+\varepsilon_0}} \|\ub_n+\zb\|_{\dot{\Hb}^3}\|\zb\|_{\dot{\Hb}^4}\nonumber\\
    %& \leq \frac{C}{\sqrt{\varepsilon_0}} \|\ub_n+\zb\|_{\dot{\Hb}^{2}}^{1-\varepsilon_0} \|\ub_n+\zb\|^{1+\varepsilon_0}_{\dot{\Hb}^3}\|\zb\|_{\dot{\Hb}^4},\label{}\\
    |\widehat{I}_{10}| & \leq C \|\mathfrak{D}^1(\ub_n+\zb)\|_{\infty}\|\mathfrak{D}^3(\ub_n+\zb)\|_{2}\|\mathfrak{D}^4\zb\|_{2}
    %\nonumber\\ &
    \leq C \|\zb\|_{\dot{\Hb}^4}  \|\ub_n+\zb\|^2_{\dot{\Hb}^3},\label{EE-A2ub15}\\
    |\widehat{I}_{11}| & \leq C  \|\mathfrak{D}^2(\ub_n+\zb)\|^2_{4}\|\mathfrak{D}^3(\ub_n+\zb)\|_{2}\nonumber\\
    & \leq C \|\ub_n+\zb\|^2_{\dot{\Hb}^{\frac52}}\|\ub_n+\zb\|_{\dot{\Hb}^3}
    %\nonumber\\ & 
    \leq C \|\ub_n+\zb\|_{\dot{\Hb}^{2}} \|\ub_n+\zb\|^{2}_{\dot{\Hb}^3},\label{EE-A2ub16}\\
    |\widehat{I}_{12}| & \leq C  \|\mathfrak{D}^2(\ub_n+\zb)\|^2_{4}\|\mathfrak{D}^3\zb\|_{2} \leq C  \|\zb\|_{\dot{\Hb}^3} \|\ub_n+\zb\|^2_{\dot{\Hb}^{3}}.\label{EE-A2ub17}
\end{align}
Making use of \eqref{EE-A2ub10}-\eqref{EE-A2ub17} in \eqref{EE-A2ub9}, we reach at
\begin{align}
   &   \frac{\drm}{\drm t}\bigg[\|\A\ub_n(t)\|_2^2 + \alpha_1 \|\A^{\frac32}\ub_n(t)\|_2^2\bigg]  + 2\nu \|\A^{\frac32}\ub_n(t)\|_2^2
   \nonumber\\ & \leq C \|\fb(t)\|^2_{\dot{\Hb}^1} + \|\zb(t)\|^2_{\dot{\Hb}^5} + C \bigg[\|\ub_n(t)+\zb(t)\|_{\dot{\Hb}^{2}} + \|\zb(t)\|_{\dot{\Hb}^{4}} \bigg]\|\ub_n(t)+\zb(t)\|^{2}_{\dot{\Hb}^3} 
   \nonumber\\ & \quad + \frac{C}{\sqrt{\varepsilon_0}}\bigg[\|\ub_n(t)+\zb(t)\|_{\dot{\Hb}^{2}}^{1-\varepsilon_0} \|\zb(t)\|_{\dot{\Hb}^3} + \|\ub_n(t)+\zb(t)\|_{\dot{\Hb}^{2}}^{2-\varepsilon_0} + \|\ub_n(t)+\zb(t)\|_{\dot{\Hb}^{2}}^{1-\varepsilon_0}\bigg]\|\ub_n(t)+\zb(t)\|^{2+\varepsilon_0}_{\dot{\Hb}^3},
\end{align}
which implies that $\mathbf{Y}(t) = 1 + \|\A \ub_n(t)\|^2_{2}+\alpha_1\|\A^{\frac{3}{2}}\ub_n(t)\|^2_{2} $  satisfies
\begin{align}
    \frac{\drm\mathbf{Y}(t)}{\drm t} & \leq C \bigg[1+  \|\fb(t)\|^2_{\dot{\Hb}^1} + \|\zb(t)\|^6_{\dot{\Hb}^5} \bigg] 
      + \frac{C}{\sqrt{\varepsilon_0}} \bigg[1+ \|\A\ub_n(t)\|^2_2+\|\zb(t)\|^6_{\dot{\Hb}^5} \bigg]\{\mathbf{Y}(t)\}^{1+\frac{\varepsilon_0}{2}}.
\end{align}
Therefore, we obtain
	\begin{align}
		\frac{\drm}{\drm t}\left[\left\{\mathbf{Y}(t)\right\}^{-\frac{\varepsilon_0}2}\right] & = - \frac{\varepsilon_0}{2\mathbf{Y}(t)^{1+\frac{\varepsilon_0}2}}\frac{\drm\mathbf{Y}(t)}{\drm t}
		\nonumber\\ & \geq  - \frac{\varepsilon_0}{2\mathbf{Y}(t)^{1+\frac{\varepsilon_0}2}}\bigg[C \big[1+  \|\fb(t)\|^2_{\dot{\Hb}^1} + \|\zb(t)\|^6_{\dot{\Hb}^5} \big] 
      + \frac{C}{\sqrt{\varepsilon_0}} \big[1+ \|\A\ub_n(t)\|^2_2+\|\zb(t)\|^6_{\dot{\Hb}^5} \big]\{\mathbf{Y}(t)\}^{1+\frac{\varepsilon_0}{2}}\bigg]
	\nonumber\\ & \geq  - {C}{\sqrt{\varepsilon_0}} \left[ 1+ \|\fb(t)\|^2_{\dot{\Hb}^1} +\|\zb(t)\|^6_{\dot{\Hb}^5} + \|\A \ub_n(t)\|^2_2  \right],
	\end{align}
	which implies
	\begin{align}
		\left\{\mathbf{Y}(t)\right\}^{-\frac{\varepsilon_0}2} \geq \left\{\mathbf{Y}(0)\right\}^{-\frac{\varepsilon_0}2} - {C}{\sqrt{\varepsilon_0}} \int_{0}^{t} \left[ 1+ \|\fb(s)\|_{\dot{\Hb}^{1}}^2 +\|\zb(s)\|^6_{\dot{\Hb}^5}  + \|\A \ub_n(s)\|^2_2  \right] \drm s.
	\end{align}
    In view of \eqref{EE-Aub11} and the fact that $\fb\in \mathrm{L}^2(0,T;\dot{\Hb}^{1}(\mathbb{T}^2))$ and $\zb\in \mathrm{C}([0,T];\dot{\Hb}^{5}(\mathbb{T}^2))$, we infer that, for any given time $T$, there exists $0<\varepsilon_*=\varepsilon_*(T,\omega)\leq 1$ such that for any $t\in[0,T]$
	\begin{align}
		   {C} \int_{0}^{t} \left[ 1+ \|\fb(s)\|_{\dot{\Hb}^{1}}^2 + \|\zb(s)\|_{\dot{\Hb}^{5}}^6  + \|\A \ub_n(s)\|^2_2  \right] \drm s \leq \frac{{\varepsilon_*}^{-\frac12}}{2} \left\{\mathbf{Y}(0)\right\}^{-\frac{\varepsilon_*}2}.
	\end{align}
Note that such an $\varepsilon_*$ exists as the limit when $\varepsilon_0\to 0$ of the right-hand side
is $+\infty$. Moreover, one can choose $\varepsilon_*$ small enough which will be suitable for any given $T>0$ and any given $\omega\in\widetilde{\Omega}$, that is, independent of $T$ and $\omega$. Hence, we obtain
	\begin{align}
	\left\{\mathbf{Y}(t)\right\}^{-\frac{\varepsilon_*}2} \geq \frac{\left\{\mathbf{Y}(0)\right\}^{-\frac{\varepsilon_*}2}}{2},
\end{align}
for $t\in[0,T]$, equivalently
\begin{align}\label{H3-regularity-n}
	\|\A \ub_n(t)\|^2_{2}+\alpha_1\|\A^{\frac{3}{2}}\ub_n(t)\|^2_{2} %+ 2\nu \int_0^t\|\A^{\frac32}\ub_n(s)\|_2^2 \drm s 
    & \leq  4^{\frac{1}{\varepsilon_*}} (1+ \|\A \ub_n(0)\|^2_{2}+\alpha_1\|\A^{\frac{3}{2}}\ub_n(0)\|^2_{2})
    \nonumber\\ & \leq  4^{\frac{1}{\varepsilon_*}} (1+ \|\A \ub(0)\|^2_{2}+\alpha_1\|\A^{\frac{3}{2}}\ub(0)\|^2_{2})
    \nonumber\\ & \leq  4^{\frac{1}{\varepsilon_*}} (1+ \|\A \vb_0\|^2_{2}+\alpha_1\|\A^{\frac{3}{2}}\vb_0\|^2_{2}),
\end{align}
for all $t\in[0,T]$. This gives that $\{\ub_n\}_{n\in\N}$ is a bounded sequence in $\mathrm{L}^{\infty}(0,T; \Drm(\A^{\frac{3}{2}}))$, consequently the unique solution $\ub$ to system \eqref{random-third-grade-fluids-equations} belongs to $\mathrm{L}^{\infty}(0,T;\Drm(\A^{\frac{3}{2}}))$ satisfying
\begin{align}\label{H3-regularity}
	\|\A \ub(t)\|^2_{2}+\alpha_1\|\A^{\frac{3}{2}}\ub(t)\|^2_{2}   \leq  4^{\frac{1}{\varepsilon_*}} (1+ \|\A \vb_0\|^2_{2}+\alpha_1\|\A^{\frac{3}{2}}\vb_0\|^2_{2}),
\end{align}
for all $t\in[0,T]$. 
\subsubsection{Regularity of solution of system \texorpdfstring{\eqref{STGF-Projected}}{}.}
Since, $\vb_0\in \Drm(\A^{\frac32})$ is deterministic, we achieve from \eqref{H3-regularity} that $\ub\in \Lrm^{p}(\Omega;\Lrm^{\infty}(0,T;\Drm(\A^{\frac32})))$, for any $p\geq2$. Hence, in view of \eqref{Z-regularity}, we have $\vb\in\mathrm{L}^{p}(\Omega;\mathrm{L}^{\infty}(0,T;\Drm(\A^{\frac32})))$, $p\geq2$. On the other hand, for given $\hat{c}>0$, let us choose $\theta> \hat{c} + \frac{\| GG^{\ast}\|_{Op}}{\lambda_1}$ and consider  
\begin{align}
   \Eb\bigg[\exp\bigg\{\hat{c}\int_0^t\|\vb(s)\|^{2}_{\dot{\Hb}^3} \drm s \bigg\}\bigg] & \leq \Eb\bigg[\exp\bigg\{2\hat{c}\left[\int_0^t\|\ub(s)\|^{2}_{\dot{\Hb}^3}+\int_0^t\|\zb(s)\|^{2}_{\dot{\Hb}^3}\right] \drm s \bigg\}\bigg]
   \nonumber\\ & \leq \exp\bigg\{2\hat{c} \int_0^t\|\ub(s)\|^{2}_{\dot{\Hb}^3} \drm s \bigg\}\Eb\bigg[\exp\bigg\{2\hat{c}\int_0^t\|\zb(s)\|^{2}_{\dot{\Hb}^3} \drm s \bigg\}\bigg] < + \infty,
\end{align}
for all $t\in[0,T]$, where we have used \eqref{H3-regularity} and \eqref{Z-exponential-regularity}. This completes the proof of Theorem \ref{Wellposedness-state}.

\section{Linearized state equation}\label{Sec-Linearized}\setcounter{equation}{0}\label{section-4}
In this section, we establish well-posedness results for the linearized state equation. Let $T$ represent the time horizon for which the unique solution $\vb$ to the state equation \eqref{STGF-Projected}, as defined in Definition \ref{def-ranndom-TGF}, exists and satisfies \eqref{eqn-vb-exponential-moments}. Additionally, recall that $\Arm(\vb)= \nabla \vb + (\nabla \vb)^{T}$,  with its components denoted as $[\Arm(\vb)]_{ij}:=a_{ij}$.	Let $\boldsymbol{\psi}: \Omega \times \mathbb{T}^2\times [0,T]\to \mathbb{R}^2$ be a force such that $\boldsymbol{\psi} \in \Lrm^p(\Omega; \mathrm{L}^2(0,T; \dot{\Lb}^2(\mathbb{T}^2)))$ (for some $p>2$) and  consider the following system:
\begin{equation}
	\left\{
	\begin{aligned}
		\partial_t (\mfrk - \alpha_1 \Delta \mfrk) &=  \boldsymbol{\psi} -\nabla\pi + \nu\Delta \mfrk - (\vb\cdot\nabla)(\mfrk - \alpha_1 \Delta \mfrk) - (\mfrk\cdot \nabla) (\vb - \alpha_1 \Delta \vb) 
         -\displaystyle\sum_{j=1}^2 [\mfrk - \alpha_1 \Delta \mfrk]^j\nabla \vb^j 
         \\
          & \quad - \displaystyle\sum_{j=1}^2 [\vb - \alpha_1 \Delta \vb]^j\nabla \mfrk^j 
         + (\alpha_1+\alpha_2) \mathrm{div}(\Arm(\vb) \Arm(\mfrk)+ \Arm(\mfrk)\Arm(\vb))  \\
		  & \quad +\beta [\mathrm{div}(\vert \Arm(\vb)\vert^{2}\Arm(\mfrk))]+2\beta  {\rm div}[(\Arm(\mfrk):\Arm(\vb))\Arm(\vb)],&& \hspace{-20mm} \text{in }  \mathbb{T}^2 \times (0, \infty),\\
		\text{div}\; \mfrk&=0, \quad && \hspace{-20mm} \text{in } \mathbb{T}^2 \times [0,\infty),\\
		%		\z &= \boldsymbol{0} &&\text{on } \partial\mathbb{T}^2\times [0,\infty),\\
		\mfrk(x,0)&=\boldsymbol{0}, \quad && \hspace{-20mm} \text{in } \mathbb{T}^2,
	\end{aligned}
	\right.\label{eqn-linearized-pi}
\end{equation}
with periodic boundary conditions, where $\mfrk =(\mfrk^{1},\mfrk^2)$ and $\pi$ are unknown vector and scalar fields, respectively. Applying  the projection $\mathcal{P}$ to the system \eqref{eqn-linearized-pi}, we obtain
\begin{equation}
	\left\{
	\begin{aligned}
		\partial_t \Upsilon(\mfrk ) &=  \Pcal\boldsymbol{\psi} - \nu\A \mfrk - \Pcal[(\vb\cdot\nabla)\Upsilon(\mfrk )] - \Pcal[(\mfrk\cdot \nabla) \Upsilon(\vb)] 
         -\Pcal\bigg[\displaystyle\sum_{j=1}^2[\Upsilon(\mfrk )]^j\nabla \vb^j \bigg] 
          - \Pcal\bigg[\displaystyle\sum_{j=1}^2[\Upsilon(\vb)]^j\nabla \mfrk^j\bigg] 
        \\
        \vspace{2mm} & \quad  + (\alpha_1+\alpha_2) \Pcal [\mathrm{div}(\Arm(\vb) \Arm(\mfrk)+ \Arm(\mfrk)\Arm(\vb))] +\beta \Pcal [\mathrm{div}(\vert \Arm(\vb)\vert^{2}\Arm(\mfrk))]+2\beta \Pcal [{\rm div}[(\Arm(\mfrk):\Arm(\vb))\Arm(\vb)]],\\
		\mfrk(0)&=\boldsymbol{0},
	\end{aligned}
	\right.\label{eqn-linearized}
\end{equation}
in $\Vb'$.

\begin{definition}\label{def-LSTGF}
 Assume that $\boldsymbol{\psi} \in \Lrm^p(\Omega; \mathrm{L}^2(0,T; \dot{\Lb}^2(\mathbb{T}^2)))$ (for some $p>2$).
	A stochastic process $\{\mfrk(t)\}_{t\geq0}$ with trajectories in  $\mathrm{C}([0,{T}]; \Vb)\cap \mathrm{L}^{\infty}(0,{T};\Drm(\A))$,  $\partial_t\mfrk \in\mathrm{L}^{2}(0,{T};\Vb)$ and $\mfrk(0)=\boldsymbol{0}$, 
	is called a \emph{  solution} to system \eqref{eqn-linearized} 
    if  for any $\phi\in \Vb$ the following equation holds 
		\begin{align}\label{W-linearized}
		\left\langle\partial_t\Upsilon(\mfrk(t)),\phi\right\rangle  &=-\int_0^t\langle\nu \A\ub(s) - (\alpha_1+\alpha_2) \mathrm{div}(\Arm(\vb) \Arm(\mfrk)+ \Arm(\mfrk)\Arm(\vb)) -\beta \mathrm{div}(\vert \Arm(\vb)\vert^{2}\Arm(\mfrk))
            \nonumber\\& \quad -2\beta {\rm div}((\Arm(\mfrk):\Arm(\vb))\Arm(\vb)) ,\phi\rangle\drm s
             +\int_0^t b(\vb(s),\phi(s),\Upsilon(\mfrk(s))) \drm s +\int_0^t b(\mfrk(s),\phi(s),\Upsilon(\vb(s))) \drm s  
             \nonumber \\ & \quad - \int_0^t b(\phi(s), \vb(s),\Upsilon(\mfrk(s)))  \drm s - \int_0^t b(\phi(s), \mfrk(s),\Upsilon(\vb(s)))  \drm s,
		\end{align}
		for a.e. $t\in[0,T]$ and $\Pb$-a.s.
\end{definition}

\begin{theorem}\label{solution-L}
		Let $T>0$, $\boldsymbol{\psi} \in \Lrm^p(\Omega; \mathrm{L}^2(0,T; \dot{\Lb}^2(\mathbb{T}^2)))$ (for some $p>2$) and $\vb$ be the solution to the state equation \eqref{STGF-Projected} in the sense of Definition \ref{def-Solution} satisfying \eqref{eqn-vb-exponential-moments}. Then, there exists a unique solution $\mfrk\in \Lrm^2(\Omega; \mathrm{L}^{\infty}(0,T; \Vb))$ to system \eqref{eqn-linearized} in the sense of Definition \ref{def-LSTGF}.
	\end{theorem}

\begin{proof}[Proof of Theorem \ref{solution-L}]
%Let us choose and fix $T > 0$ and $0<\varepsilon<1$.  
Due to assumptions, we have $\boldsymbol{\psi} \in \mathrm{L}^{2}(0,T; \dot{\Lb}^2(\mathbb{T}^2))$ and $\vb \in \mathrm{L}^{\infty}(0,T; \dot{\Hb}^3(\mathbb{T}^2))$, $\Pb$-a.s. Therefore, we will consider the system \eqref{eqn-linearized} pathwise and do the analysis accordingly.
After showing the pathwise well-posedness, we will verify that the second-order moments are finite.  The proof is divided into the following five steps.
		\vskip 2mm
		\noindent
		\textbf{Step I.} \textit{Finite-dimensional approximation.} Let us consider the following approximate equation for system \eqref{eqn-linearized} on the finite-dimensional space $\Hb_n$:
    \begin{equation}\label{finite-dimS-linerarized}
			\left\{
			\begin{aligned}
		\partial_t (\Upsilon(\mfrk_n)) &=  \Prm_n\Pcal\boldsymbol{\psi} - \nu\A \mfrk_n - \Prm_n\Pcal[(\vb\cdot\nabla)\Upsilon(\mfrk_n)] - \Pcal[(\mfrk_n\cdot \nabla) \Upsilon(\vb)] 
           -\Prm_n\Pcal\bigg[\displaystyle\sum_{j=1}^2[\Upsilon(\mfrk_n)]^j\nabla \vb^j \bigg] 
        \\
        \vspace{2mm} & \quad - \Prm_n\Pcal\bigg[\displaystyle\sum_{j=1}^2[\Upsilon(\vb)]^j\nabla (\mfrk_n)^j\bigg] 
          + (\alpha_1+\alpha_2)\Prm_n \Pcal [\mathrm{div}(\Arm(\vb) \Arm(\mfrk_n)+ \Arm(\mfrk_n)\Arm(\vb))]\vspace{2mm} \\
		\vspace{2mm} & \quad +\beta \Prm_n\Pcal [[\mathrm{div}(\vert \Arm(\vb)\vert^{2}\Arm(\mfrk_n))]]+2\beta \Prm_n\Pcal [{\rm div}[(\Arm(\mfrk_n):\Arm(\vb))\Arm(\vb)]],\\
				\mfrk_n(0)&=\boldsymbol{0},
	\end{aligned}
			\right.
		\end{equation}
	for a.e. $t\in[0,{T}]$. Note that system \eqref{finite-dimS-linerarized} defines a system of linear ordinary differential equations, which has a unique local solution $\mfrk_n\in\mathrm{C}([0,{T}_n];\Hb_n)$, for some $0<{T}_n\leq  {T}$. The following uniform a priori estimates show that the time ${T}_n={T}$. 
	
	\vskip 2mm
	\noindent
	\textbf{Step II.} \textit{A priori estimates.} 	Taking the inner product with $\mfrk_n$ to $\eqref{finite-dimS-linerarized}_1$, integrating by parts and using \eqref{b0}, we obtain 
\begin{align}\label{lin-ub1}
  & \frac12 \frac{\drm}{\drm t}\bigg[\|\mfrk_n(t)\|_2^2+\alpha_1 \|\nabla\mfrk_n(t)\|_2^2\bigg] + \nu \|\nabla\mfrk_n(t)\|_2^2 + \frac{\beta}{2}\||\Arm(\vb(t))|\Arm(\mfrk_n(t))\|^2_2
    + \beta \int_{\mathbb{T}^2} (\Arm(\mfrk_n(x,t)):\Arm(\vb(x,t)))^2 \drm x 
   %\nonumber\\ &= (\boldsymbol{\psi}(t), \mfrk_n(t)) - b ( \vb(t), \mfrk_n(t)-\alpha_1 \Delta \mfrk_n(t), \mfrk_n (t)) - b ( \mfrk_n(t), \vb(t)-\alpha_1 \Delta \vb(t), \mfrk_n (t)) 
   %\nonumber\\ & \quad - b (\mfrk_n (t), \vb(t), \mfrk_n(t)-\alpha_1 \Delta \mfrk_n(t)) - b ( \mfrk_n (t), \mfrk_n(t), \vb(t)-\alpha_1 \Delta \vb(t)) 
   %\nonumber\\ & \quad - (\alpha_1+\alpha_2) \left\langle \Arm(\vb(t)) \Arm(\mfrk_n(t)) +  \Arm(\mfrk_n(t))\Arm(\vb(t))), \nabla\mfrk_n(t)\right\rangle
   \nonumber\\ &= \underbrace{(\boldsymbol{\psi}(t), \mfrk_n(t))}_{=:M_1} - \underbrace{\alpha_1 b ( \vb(t), \mfrk_n(t), \Delta \mfrk_n(t))}_{=:M_2} - \underbrace{b ( \mfrk_n(t), \vb(t), \mfrk_n (t))}_{=:M_3} + \underbrace{\alpha_1  b (\mfrk_n (t), \vb(t), \Delta \mfrk_n(t))  }_{=:M_4}
   \nonumber\\ & \quad - \underbrace{(\alpha_1+\alpha_2) \left\langle \Arm(\vb(t)) \Arm(\mfrk_n(t)) +  \Arm(\mfrk_n(t))\Arm(\vb(t))), \nabla\mfrk_n(t)\right\rangle}_{=:M_5},
   \end{align}
for a.e. $t\in[0,T]$. Using integration by parts, H\"older's, Sobolev's, Ladyzhenskaya's, Young's inequalities,  we deduce
\begin{align}
    |M_1|& \leq  C \|\boldsymbol{\psi}\|^2_2 + C\|\mfrk_n\|^2_2,\label{lin-ub2}\\
    |M_2|& = \alpha_1 \left|\int_{\Tb^2}\mathfrak{D}^1\vb\cdot\mathfrak{D}^1\mfrk_n\cdot\mathfrak{D}^1\mfrk_n \drm x\right| \leq C\|\nabla\vb\|_{\infty}\|\nabla\mfrk_n\|^2_2 \leq C \|\vb\|_{\dot{\Hb}^3}\|\nabla\mfrk_n\|^2_2,\label{lin-ub3}\\
    |M_3|& \leq C\|\nabla\vb\|_{\infty}\|\mfrk_n\|^2_2 \leq C \|\vb\|_{\dot{\Hb}^3}\|\nabla\mfrk_n\|^2_2,\label{lin-ub4}\\
    |M_4|& = \left|\int_{\Tb^2}\mathfrak{D}^1\mfrk_n\cdot\mathfrak{D}^1\vb\cdot\mathfrak{D}^1\mfrk_n \drm x + \int_{\Tb^2} \mathfrak{D}^0\mfrk_n\cdot\mathfrak{D}^2\vb\cdot\mathfrak{D}^1\mfrk_n \right| 
    \nonumber\\ & \leq C \|\nabla\vb\|_{\infty}\|\nabla\mfrk_n\|^2_2 + C \|\mathfrak{D}^2\vb\|_{4}\|\mfrk_n\|_4 \|\nabla\mfrk_n\|_2
      \leq  C \|\vb\|_{\dot{\Hb}^3} \|\nabla\mfrk_n\|_2^{2},\label{lin-ub5}\\
      |M_5|& \leq \frac{\beta}{4}\||\Arm(\vb)|\Arm(\mfrk_n)\|^2_2 + C \|\nabla\mfrk_n\|^2_2\label{lin-ub6}.
\end{align}
Combining \eqref{lin-ub1}-\eqref{lin-ub6}, we achieve
\begin{align}\label{lin-ub7}
  &  \frac{\drm}{\drm t}\bigg[\|\mfrk_n(t)\|_2^2+\alpha_1 \|\nabla\mfrk_n(t)\|_2^2\bigg] \leq C \|\boldsymbol{\psi}(t)\|^2_2+ C \{1+ \|\vb(t)\|_{\dot{\Hb}^3}\}\bigg[\|\mfrk_n(t)\|_2^2+\alpha_1 \|\nabla\mfrk_n(t)\|_2^2\bigg],
   \end{align}
   for a.e. $t\in[0,T]$, and the Gronwall inequality implies
\begin{align}\label{lin-ub8}
\|\mfrk_n(t)\|_2^2+\alpha_1 \|\nabla\mfrk_n(t)\|_2^2  \leq  C  e^{Ct + C \int_0^t \|\vb(s)\|_{\dot{\Hb}^3}\drm s}  \int_0^t \|\boldsymbol{\psi}(s)\|^2_2 \drm s,
 \end{align}
for all $t\in[0,{T}]$. Using the fact that 
        $\boldsymbol{\psi} \in\mathrm{L}^{2}(0,{T};\dot{\Lb}^2(\mathbb{T}^2))$ and $\vb\in \mathrm{L}^{\infty}(0, {T};\dot{\Hb}^3(\mathbb{T}^2))$, we have from \eqref{lin-ub8} that
		\begin{align}\label{lin-ub9}
			\{\mfrk_n\}_{n\in\N} \text{ is a bounded sequence in }\mathrm{L}^{\infty}(0,{T};\Vb).%\cap\mathrm{L}^{2}(0,T;\V).
		\end{align} 
        
Next, since $\{\wb_i\}_{i\in\mathbb{N}}$ is a sequence of eigenfunctions of the Stokes operator $\A$, it gives $\A\mfrk_n\in \Hb_n$, therefore taking the inner product with $\A\mfrk_n$ to $\eqref{finite-dimS-linerarized}_1$ and using \eqref{b0}, we obtain 
\begin{align}\label{lin-ub10}
  & \frac12 \frac{\drm}{\drm t}\bigg[\|\nabla\mfrk_n(t)\|_2^2+\alpha_1 \|\A\mfrk_n(t)\|_2^2\bigg] + \nu \|\A\mfrk_n(t)\|_2^2 
   \nonumber\\ &=   \underbrace{(\boldsymbol{\psi}(t),\A\mfrk_n(t))}_{\widetilde{M}_1}  - \underbrace{b( \vb(t), \mfrk_n(t) , \A\mfrk_n(t))}_{\widetilde{M}_2} - \underbrace{b(\mfrk_n(t) , \Upsilon(\vb(t)), \A\mfrk_n(t))}_{\widetilde{M}_3} 
    - \underbrace{b(\A\mfrk_n(t), \vb(t), \Upsilon(\mfrk_n(t))  )}_{\widetilde{M}_4} \nonumber\\ & \quad + \underbrace{b(\A\mfrk_n(t), \Upsilon(\vb(t)) , \mfrk_n(t))}_{\widetilde{M}_5}  + \underbrace{(\alpha_1+\alpha_2) \int_{\mathbb{T}^2}\diver [\Arm(\vb(t)) \Arm(\mfrk_n(t))+\Arm(\mfrk_n(t))\Arm(\vb(t))]\cdot \A\mfrk_n(t)\drm x }_{\widetilde{M}_6}
     \nonumber \\
		 & \quad + \underbrace{\beta\bigg[ \int_{\mathbb{T}^2} \diver [\vert \Arm(\vb(t))\vert^{2}\Arm(\mfrk_n(t)))]\cdot \A\mfrk_n(t)\drm x 
           + 2 \int_{\mathbb{T}^2} \diver[(\Arm(\mfrk_n(t)):\Arm(\vb(t)))\Arm(\vb(t))]\cdot \A\mfrk_n(t))\drm x\bigg]}_{\widetilde{M}_7},
   \end{align}
for a.e. $t\in[0,{T}]$. Let us estimate each term of right hand side of \eqref{lin-ub10} as follows:
\begin{align}
    |\widetilde{M}_1|& \leq  C \|\boldsymbol{\psi}\|^2_2 + C\|\A\mfrk_n\|^2_2,\label{lin-ub11}\\
    |\widetilde{M}_2|& = \left| \int_{\Tb^2}\mathfrak{D}^0\vb\cdot\mathfrak{D}^1\mfrk_n\cdot\mathfrak{D}^2\mfrk_n \drm x  \right|  \leq C \|\vb\|_4\|\mathfrak{D}^1\mfrk_n\|_4\|\mathfrak{D}^2\mfrk_n\|_2 \leq C \|\vb\|_{\dot{\Hb}^3}\|\A\mfrk_n\|_2^2,\label{lin-ub12}\\
    |\widetilde{M}_3|& = \left| \int_{\Tb^2}\mathfrak{D}^0\mfrk_n\cdot\mathfrak{D}^3\vb\cdot\mathfrak{D}^2\mfrk_n \drm x  \right|  \leq C \|\mfrk_n\|_{\infty}\|\mathfrak{D}^3\vb\|_2\|\mathfrak{D}^2\mfrk_n\|_2 \leq C \|\vb\|_{\dot{\Hb}^3}\|\A\mfrk_n\|_2^2,\label{lin-ub13}\\
    |\widetilde{M}_4|& = \left| \int_{\Tb^2}\mathfrak{D}^2\mfrk_n\cdot\mathfrak{D}^1\vb\cdot\mathfrak{D}^2\mfrk_n \drm x  \right|  \leq C \|\nabla\vb\|_{\infty}\|\mathfrak{D}^2\mfrk_n\|^2_2 \leq C \|\vb\|_{\dot{\Hb}^3}\|\A\mfrk_n\|_2^2,\label{lin-ub14}\\
    |\widetilde{M}_5|& = \left| \int_{\Tb^2}\mathfrak{D}^2\mfrk_n\cdot\mathfrak{D}^3\vb\cdot\mathfrak{D}^0\mfrk_n \drm x  \right|  \leq C  \|\mathfrak{D}^2\mfrk_n\|_2 
 \|\mathfrak{D}^3\vb\|_2\|\mfrk_n\|_{\infty} \leq C \|\vb\|_{\dot{\Hb}^3}\|\A\mfrk_n\|_2^2,\label{lin-ub15}\\
 |\widetilde{M}_6|& = \left| \int_{\Tb^2}\mathfrak{D}^1\vb\cdot\mathfrak{D}^2\mfrk_n\cdot\mathfrak{D}^2\mfrk_n \drm x + \int_{\Tb^2}\mathfrak{D}^2\vb\cdot\mathfrak{D}^1\mfrk_n\cdot\mathfrak{D}^2\mfrk_n \drm x  \right| 
 \nonumber\\ & \leq C \|\mathfrak{D}^1\vb\|_{\infty} \|\mathfrak{D}^2\mfrk_n\|^2_2  + C  \|\mathfrak{D}^2\vb\|_4 
 \|\mathfrak{D}^1\mfrk_n\|_4\|\mathfrak{D}^2\mfrk_n\|_{2} \leq C \|\vb\|_{\dot{\Hb}^3}\|\A\mfrk_n\|_2^2,\label{lin-ub16}\\
 |\widetilde{M}_7|& = \left| \int_{\Tb^2}\mathfrak{D}^1\vb\cdot\mathfrak{D}^2\vb\cdot\mathfrak{D}^1\mfrk_n\cdot\mathfrak{D}^2\mfrk_n \drm x + \int_{\Tb^2}\mathfrak{D}^1\vb\cdot\mathfrak{D}^1\vb\cdot\mathfrak{D}^2\mfrk_n\cdot\mathfrak{D}^2\mfrk_n \drm x  \right| 
 \nonumber\\ & \leq C \|\mathfrak{D}^1\vb\|_{\infty}\|\mathfrak{D}^2\vb\|_{4} \|\mathfrak{D}^1\mfrk_n\|_4\|\mathfrak{D}^2\mfrk_n\|_2  + C  \|\mathfrak{D}^1\vb\|^2_{\infty} 
 \|\mathfrak{D}^2\mfrk_n\|^2_{2} \leq C \|\vb\|_{\dot{\Hb}^3}^2\|\A\mfrk_n\|_2^2.\label{lin-ub17}
\end{align}
Combining \eqref{lin-ub11}-\eqref{lin-ub17}, we achieve
\begin{align}\label{lin-ub18}
  &  \frac{\drm}{\drm t}\bigg[\|\nabla\mfrk_n(t)\|_2^2+\alpha_1 \|\A\mfrk_n(t)\|_2^2\bigg] \leq C \|\boldsymbol{\psi}(t)\|^2_2+ C \{1+ \|\vb(t)\|^2_{\dot{\Hb}^3}\}\bigg[\|\nabla\mfrk_n(t)\|_2^2+\alpha_1 \|\A\mfrk_n(t)\|_2^2\bigg],
   \end{align}
   for a.e. $t\in[0,T]$, and the Gronwall inequality implies
\begin{align}\label{lin-ub19}
\|\nabla\mfrk_n(t)\|_2^2+\alpha_1 \|\A\mfrk_n(t)\|_2^2  \leq  C  e^{Ct + C \int_0^t \|\vb(s)\|^2_{\dot{\Hb}^3}\drm s}  \int_0^t \|\boldsymbol{\psi}(s)\|^2_2 \drm s,
 \end{align}
for all $t\in[0,{T}]$. Using the fact that 
        $\boldsymbol{\psi} \in\mathrm{L}^{2}(0,{T};\dot{\Lb}^2(\mathbb{T}^2))$ and $\vb\in \mathrm{L}^{\infty}(0, {T};\dot{\Hb}^3(\mathbb{T}^2))$, we have from \eqref{lin-ub19} that
		\begin{align}\label{lin-ub20}
			\{\mfrk_n\}_{n\in\N} \text{ is a bounded sequence in }\mathrm{L}^{\infty}(0,{T};\Drm(\A)).%\cap\mathrm{L}^{2}(0,T;\V).
		\end{align} 
        
Taking the inner product with $\partial_t\mfrk_n$ to $\eqref{finite-dimS-linerarized}_1$ and using integration by parts, we obtain 
\begin{align}\label{lin-ub21}
  & \|\partial_t\mfrk_n(t)\|_2^2 + \alpha_1 \|\partial_t\nabla\mfrk_n(t)\|^2_2 +  \frac{\nu}{2} \frac{\drm}{\drm t} \|\nabla\mfrk_n(t)\|_2^2 
   \nonumber\\ &=   \underbrace{(\boldsymbol{\psi}(t),\partial_t\mfrk_n(t))}_{\widehat{M}_1} + \underbrace{b(\vb(t),  \partial_t\mfrk_n(t), \Upsilon(\mfrk_n(t)))}_{\widehat{M}_2} - \underbrace{b(\mfrk_n(t), \Upsilon(\vb(t)), \partial_t\mfrk_n(t))}_{\widehat{M}_3} - \underbrace{b(\partial_t\mfrk_n(t), \vb(t), \Upsilon(\mfrk_n(t)))}_{\widehat{M}_4}
   \nonumber\\ & \quad 
   + \underbrace{b(\partial_t\mfrk_n(t), \Upsilon(\vb(t)) , \mfrk_n(t)) }_{\widehat{M}_4}
    - \underbrace{(\alpha_1+\alpha_2) \big(\Arm(\vb(t)) \Arm(\mfrk_n(t))+ \Arm(\mfrk_n(t))\Arm(\vb(t)), \partial_t\nabla\mfrk_n(t)\big)}_{\widehat{M}_5} 
    \nonumber\\ & \quad  - \underbrace{\beta \big( \vert \Arm(\vb(t))\vert^{2}\Arm(\mfrk_n(t)), \partial_t\nabla\mfrk_n(t)\big) }_{\widehat{M}_6}
         - \underbrace{2\beta   \big( (\Arm(\mfrk_n(t)):\Arm(\vb(t)))(\Arm(\vb(t)), \partial_t\nabla\mfrk_n(t))\big)}_{\widehat{M}_7} ,
   \end{align}
   for a.e. $t\in[0,{T}]$. Let us estimate the terms on the right hand side of \eqref{lin-ub21} using H\"older's, Sobolev's and Young's inequalities as follows:
\begin{align}
    |\widehat{M}_1|& \leq  C \|\boldsymbol{\psi}\|^2_2 + \frac{1}{6}\|\partial_t \mfrk\|^2_2,\label{lin-ub22}\\
    |\widehat{M}_2|& \leq C \|\vb\|_{\infty}\|\partial_t \nabla\mfrk_n\|_2\|\mathfrak{D}^2\mfrk_n\|_2 \leq C \|\vb\|_{\dot{\Hb}^3}\|\partial_t \nabla\mfrk_n\|_2\|\A\mfrk_n\|_2 
    %\nonumber\\ & 
    \leq  C \|\vb\|_{\dot{\Hb}^3}^2\|\A\mfrk_n\|_2^2 + \frac{1}{6}\|\partial_t\nabla \mfrk_n\|^2_2,\label{lin-ub23}\\
    |\widehat{M}_3|& \leq C \|\mathfrak{D}^3\vb\|_{2}\|\partial_t \mfrk_n\|_2\|\mfrk_n\|_{\infty} \leq C \|\vb\|_{\dot{\Hb}^3}\|\partial_t \mfrk_n\|_2\|\A\mfrk_n\|_2 
    %\nonumber\\ & 
    \leq  C \|\vb\|_{\dot{\Hb}^3}^2\|\A\mfrk_n\|_2^2 + \frac{1}{6}\|\partial_t \mfrk_n\|^2_2,\label{lin-ub24}\\
    |\widehat{M}_4|& \leq C \|\nabla\vb\|_{\infty}\|\partial_t \mfrk_n\|_2\|\mathfrak{D}^2\mfrk_n\|_{2} \leq C \|\vb\|_{\dot{\Hb}^3}\|\partial_t \mfrk_n\|_2\|\A\mfrk_n\|_2 
    %\nonumber\\ & 
    \leq  C \|\vb\|_{\dot{\Hb}^3}^2\|\A\mfrk_n\|_2^2 + \frac{1}{6}\|\partial_t \mfrk_n\|^2_2,\label{lin-ub25}\\
    |\widehat{M}_5|& \leq C \|\nabla\vb\|_{\infty}\|\partial_t\nabla \mfrk_n\|_2\|\nabla\mfrk_n\|_{2} \leq C \|\vb\|_{\dot{\Hb}^3}\|\partial_t\nabla \mfrk_n\|_2\|\A\mfrk_n\|_2 
    %\nonumber\\ & 
    \leq  C \|\vb\|_{\dot{\Hb}^3}^2\|\A\mfrk_n\|_2^2 + \frac{1}{6}\|\partial_t\nabla \mfrk_n\|^2_2,\label{lin-ub26}\\
    |\widehat{M}_6+\widehat{M}_7|& \leq C \|\nabla\vb\|^2_{\infty}\|\partial_t\nabla \mfrk_n\|_2\|\nabla\mfrk_n\|_{2} \leq C \|\vb\|^2_{\dot{\Hb}^3}\|\partial_t\nabla \mfrk_n\|_2\|\A\mfrk_n\|_2 
    %\nonumber\\ & 
    \leq  C \|\vb\|_{\dot{\Hb}^3}^4\|\A\mfrk_n\|_2^2 + \frac{1}{6}\|\partial_t\nabla \mfrk_n\|^2_2.\label{lin-ub27}
\end{align}
Combining \eqref{lin-ub21}-\eqref{lin-ub27}, we have 
\begin{align}\label{lin-ub28}
    & \|\partial_t\mfrk_n(t)\|_2^2 + \alpha_1 \|\partial_t\nabla\mfrk_n(t)\|^2_2 +  \nu \frac{\drm}{\drm t} \|\nabla\mfrk_n(t)\|_2^2 
        \leq  C\|\boldsymbol{\psi}(t)\|^2_2 + C [1 + \|\vb(t)\|_{\dot{\Hb}^3}^4]\|\A\mfrk_n(t)\|_2^2,
\end{align}
for a.e. $t\in[0,{T}]$, which due to $\mfrk_n(0)=\boldsymbol{0}$ implies 
\begin{align}\label{lin-ub29}
  \int_0^t     \bigg[ \|\partial_t\mfrk_n(s)\|_2^2 + \alpha_1 \|\partial_t\nabla\mfrk_n(s)\|^2_2 \bigg] \drm s \leq    C \int_0^t \|\boldsymbol{\psi}(s)\|^2_2 + [1 + \|\vb(s)\|_{\dot{\Hb}^3}^4]\|\A\mfrk_n(s)\|_2^2 \drm s,
 \end{align}
for all $t\in[0,{T}]$. Using the fact that 
        $\boldsymbol{\psi} \in\mathrm{L}^{2}(0,{T};\dot{\Lb}^2(\mathbb{T}^2))$, $\vb\in \mathrm{L}^{\infty}(0,{T};\dot{\Hb}^3(\Tb^2))$ and $\{\mfrk_n\}_{n\in\N}$ is a bounded sequence in $\mathrm{L}^{\infty}(0,T;\Drm(\A))$ (see \eqref{lin-ub20}), we have from \eqref{lin-ub29} that
		\begin{align}\label{lin-ub30}
			\{\partial_t\mfrk_n\}_{n\in\N} \text{ is a bounded sequence in } \mathrm{L}^{2}(0,{T};\Vb).
		\end{align}
	Note that, similar to \eqref{eqn-I+A_i_u}, in view of \eqref{lin-ub30}, we get
	\begin{align}\label{lin-ub31}
		 \{\partial_t\Upsilon(\mfrk_n)\}_{n\in\N} \text{ is a bounded sequence in }\mathrm{L}^{2}(0,{T};\Vb^{\prime}).
	\end{align}

\vskip 2mm
	\noindent
	\textbf{Step III.} \textit{Weak and strong limits, and passing $n\to \infty$ in system \eqref{finite-dimS-linerarized}.}  Using \eqref{lin-ub20}, \eqref{lin-ub31}, the \textit{Banach-Alaoglu theorem} and \textit{Aubin-Lions compactness lemma} (similar to \eqref{Strong-Convergence-v}), we infer the existence of an element $\mfrk \in\mathrm{L}^{\infty}(0,{T};\Drm(\A))$ with $\partial_t (\Irm+\alpha_1\A) \mfrk \in \mathrm{L}^{2}(0,{T};\Vb^{\prime})$ such that
\begin{equation}\label{lin-ub32}
		\left.    \begin{aligned}
			\mfrk_n\xrightharpoonup{w^*}&\ \mfrk &&\text{ in }\ \ \ \ \	\mathrm{L}^{\infty}(0,{T};\Drm(\A)),\\
			\partial_t \Upsilon(\mfrk_n) \xrightharpoonup{w}&\;  \partial_t \Upsilon (\mfrk)   && \text{ in }  \ \ \ \ \ \mathrm{L}^{2}(0,{T};\Vb^{\prime}),
			\\
			\mfrk_n \to & \ \mfrk && \text{ in }\ \ \ \ \	\mathrm{L}^{2}(0,{T};\Vb),
		\end{aligned}\right\}
	\end{equation}
    along a subsequence (still denoted by the same symbol). Making use of weak and strong convergence \eqref{lin-ub32}, $\vb\in\mathrm{L}^{\infty}(0,{T};\Drm(\A^{\frac{3}{2}}))$ and strong convergence $\Prm_n\Pcal\boldsymbol{\psi}\to \Pcal\boldsymbol{\psi} $ in $\mathrm{L}^{2}(0,{T};\Hb))$, one can easily pass the limit $n\to\infty$ in \eqref{finite-dimS-linerarized} and obtain
\begin{align}
 	\partial_t \Upsilon (\mfrk) &=  \Pcal\boldsymbol{\psi} - \nu\A \mfrk - \Pcal[(\vb\cdot\nabla)\Upsilon (\mfrk)] - \Pcal[(\mfrk\cdot \nabla) \Upsilon (\vb)] 
          - \Pcal\bigg[\displaystyle\sum_{j=1}^2[\Upsilon (\mfrk)]^j\nabla \vb^j \bigg] 
        -  \Pcal\bigg[\displaystyle\sum_{j=1}^2[\Upsilon (\mfrk)]^j\nabla (\mfrk)^j\bigg] 
         \nonumber \\
         & \quad
        + (\alpha_1+\alpha_2) \Pcal [\mathrm{div}(\Arm(\vb) \Arm(\mfrk)+ \Arm(\mfrk)\Arm(\vb))]  +\beta \Pcal [[\mathrm{div}(\vert \Arm(\vb)\vert^{2}\Arm(\mfrk))]]+2\beta \Pcal [{\rm div}[(\Arm(\mfrk):\Arm(\vb))\Arm(\vb)]],
 \end{align}
 in  $\mathrm{L}^{2}(0,{T};\Vb^{\prime})$. Next, since $(\Irm + \alpha_1\A)^{\frac12}\mfrk\in\mathrm{L}^{2}(0,{T};\Vb)$ and $\partial_t(\Irm+\alpha_1\A)^{\frac12}\mfrk \in \mathrm{L}^{2}(0, {T};\Vb')$, it implies from \cite[Lemma 1.2, p.176]{temam2001navier} that $(\Irm+\alpha_1\A)^{\frac12}\mfrk\in \mathrm{C}([0, {T}];\Hb)$, the real-valued function $t\mapsto\|(\Irm+\alpha_1\A)^{\frac12}\mfrk(t)\|_{2}^2$ is absolutely continuous and the following equality is satisfied:
	\begin{align}
		\frac12\frac{\drm}{\drm t}\|(\Irm+\alpha_1\A)^{\frac12}\mfrk(t)\|_{2}^2 =  \left<\partial_t (\Irm+\alpha\A)^{\frac12}\mfrk(t),(\Irm+\alpha_1\A)^{\frac12}\mfrk(t)  \right>,\  \ \ \ \text{ for a.e. } t\in[0, {T}].
	\end{align}
	%This also implies
	%\begin{align}\label{EE1-lin}
%		\frac12 \frac{\drm}{\drm t}\bigg[\|\mfrk(t)\|_{2}^2+\alpha_1\|\nabla\mfrk(t)\|_{2}^2\bigg]& =  \left<\partial_t (\Irm+\alpha_1\A)\mfrk(t),\mfrk(t)  \right>,
%	\end{align}
%	for a.e.  $t\in[0,{T}]$, which provides \eqref{eeq-L} immediately. 
Also, the initial condition in the Definition \ref{def-LSTGF} is well defined. 

    \vskip 2mm
	\noindent
	\textbf{Step IV.}
	\textit{Uniqueness:}
	Define $\mathfrak{M}:=\mfrk_1-\mfrk_2$, where $\mfrk_1$ and $\mfrk_2$ are two solutions of system \eqref{eqn-linearized} in the sense of Definition \ref{def-LSTGF}. Since, system \eqref{eqn-linearized} is linear in $\mfrk$, therefore $\mathfrak{M}$ is also a solution of system \eqref{eqn-linearized} replacing $\boldsymbol{\psi}$ with $\boldsymbol{0}$ and it satisfies
	\begin{align}\label{eeq-L-Uni}
		&\|\mathfrak{M}(t)\|_{2}^2+ \alpha_1 \|\nabla\mathfrak{M}(t)\|_{2}^2 + 2 \nu\int_0^t\|\nabla\mathfrak{M}(s)\|_{2}^2\drm s + \beta \int_0^t\||\Arm(\vb(s))|\Arm(\mathfrak{M}(s))\|^2_2 \drm s 
	+ 2 \beta \int_0^t\int_{\mathbb{T}^2}\left[\Arm(\vb(s)):\Arm(\mathfrak{M}(s))\right]^2\drm x \drm s    
		\nonumber\\&= - 2 \int_0^t [ \alpha_1 b ( \vb(t), \mathfrak{M}(t), \Delta \mathfrak{M}(t)) +  b ( \mathfrak{M}(t), \vb(t), \mathfrak{M} (t)) - \alpha_1  b (\mathfrak{M} (t), \vb(t), \Delta \mathfrak{M}(t)) ] \drm s
		\nonumber\\ & \quad +  (\alpha_1+\alpha_2) \int_0^t\int_{\mathbb{T}^2}[\Arm(\vb(s))\Arm(\mathfrak{M}(s))+ \Arm(\mathfrak{M}(s))\Arm(\vb(s))]:\Arm(\mathfrak{M}(s))\drm x \drm s ,
	\end{align}
	for all $t\in[0,{T}]$. A
	reasoning similar to that used to obtain \eqref{lin-ub7} and $\mathfrak{M}(0)=\textbf{0}$ yield
	\begin{align}
		&\frac{\drm}{\drm t}\bigg[\|\mathfrak{M}(t)\|_2^2+\alpha_1 \|\nabla\mathfrak{M}(t)\|_2^2\bigg] \leq  C \{1+ \|\vb(t)\|_{\dot{\Hb}^3}\}\bigg[\|\mathfrak{M}(t)\|_2^2+\alpha_1 \|\nabla\mathfrak{M}(t)\|_2^2\bigg],
	\end{align}
	for all $t\in[0,{T}]$. Hence, an application of Gronwall's inequality and the fact that $\vb\in \mathrm{L}^{\infty}(0,{T};\Drm(\A^{\frac{3}{2}}))$   give $\mfrk_1(t)=\mfrk_2(t)$ in $\Vb$, for all $t\in[0,{T}]$.
\vskip 2mm
	\noindent
	\textbf{Step V.}
	\textit{Moments:}    Finally we show that $\mfrk\in \Lrm^2(\Omega; \mathrm{L}^{\infty}(0,T; \Drm(\A)))$. 
    \iffalse 
    An argument similar to \eqref{lin-ub19} leads to 
    \begin{align}\label{Moment-1}
\|\nabla\mfrk(t)\|_2^2+\alpha_1 \|\A\mfrk(t)\|_2^2  \leq  C  e^{Ct + C \int_0^t \|\vb(s)\|^2_{\dot{\Hb}^3}\drm s}  \int_0^t \|\boldsymbol{\psi}(s)\|^2_2 \drm s,
 \end{align}
for all $t\in[0,T]$.
\fi
Taking supremum over $t\in[0,T]$ and then expectation of \eqref{lin-ub19}, we find for $p>2$
    \begin{align}\label{Moment-2}
& \Eb\bigg[\sup_{t\in[0,T]}\|\nabla\mfrk_n(t)\|_2^2+\alpha_1 \sup_{t\in[0,T]} \|\A\mfrk_n(t)\|_2^2 \bigg] 
\nonumber\\ & \leq  C \Eb \left[\exp\left\{CT + C \int_0^T \|\vb(s)\|^2_{\dot{\Hb}^3}\drm s\right\} \cdot \|\boldsymbol{\psi}(s)\|^2_{\Lrm^2(0,T;\dot{\Lb}^2)}\right]
\nonumber\\ & \leq  C \left\{\Eb \left[\exp\left\{\frac{pC}{p-2}(T +  \int_0^T \|\vb(s)\|^2_{\dot{\Hb}^3}\drm s)\right\} \right]\right\}^{\frac{p-2}{p}}\cdot  \|\boldsymbol{\psi}(s)\|^2_{\Lrm^p(\Omega;\Lrm^2(0,T;\dot{\Lb}^2))}
  <+\infty,
 \end{align}
 where we have used \eqref{eqn-vb-exponential-moments} and the fact that $\boldsymbol{\psi} \in \Lrm^p(\Omega; \mathrm{L}^2(0,T; \dot{\Lb}^2(\mathbb{T}^2)))$ (for some $p>2$). Therefore, by Banach-Alaoglu theorem, and uniqueness of weak and weak* limits, we have (along a subsequence)
 \begin{equation}\label{Moment-3}
	%	\left. 
        \begin{aligned}
	%		\mfrk_n\xrightharpoonup{w^*}&\ \mfrk &&\text{ in }\ \ \ \ \	\Lrm^2(\Omega;\mathrm{L}^{\infty}(0,{T};\Drm(\A))),\\
			\mfrk_n\xrightharpoonup{w}&\ \mfrk &&\text{ in }\ \ \ \ \	\Lrm^2(\Omega;\mathrm{L}^{2}(0,{T};\Drm(\A))).
		\end{aligned}%\right\}
	\end{equation}
This completes the proof of Theorem \ref{solution-L}.
\end{proof}

%\Eb\bigg[\exp\bigg\{\frac14\min\{\alpha_1,c_{\zb}\} \sup_{t\in [0,T]}\|\vb(s)\|^2_{\dot{\Hb}^3} \bigg\}\bigg]

\section{G\^ateaux differentiability of the control-to-state mapping}\label{Sec-Gateaux-differentiability}\setcounter{equation}{0}

This section analyzes the differentiability of the control-to-state mapping. Specifically, we demonstrate that the solution to the linearized equation corresponds to the G\^ateaux derivative of the control-to-state mapping. In order to prove such result, we need the following stability result:

\begin{proposition}
    Let $\vb_1,\vb_2$ be two solutions to system \eqref{STGF-Projected} in the sense of Definition \ref{def-Solution} satisfying \eqref{eqn-vb-exponential-moments} associated with external forcing $\fb_1,\fb_2\in \Lrm^2(\Omega;\mathrm{L}^2(0,T;\dot{\Hb}^{1}(\mathbb{T}^2)))$ and same initial data $\vb_{0}\in \Drm(\A^{\frac{3}{2}})$, respectively. Then, there exists a constant $K$ depending only on the given data $\nu$, $\alpha_1$, $\alpha_2$ and $\beta$ such that 
    \begin{align}\label{Stability-inequality}
    %\|\vfrk(t)\|_{2}^2 + 2\alpha_1\|\nabla\vfrk(t)\|_{2}^2+ \alpha_1^2
    \sup_{t\in[0,T]}\|\A\vfrk(t)\|_{2}^2 
    & \leq C e^{C\int_0^T [ 1+ \|\vb_1(t)\|^2_{\dot{\Hb}^3}+\|\vb_2(t)\|^2_{\dot{\Hb}^3}]\drm s} \|\fb_1-\fb_2\|^2_{\Lrm^2(0,T;\dot{\Lb}^2)}, \;\; \; \Pb\text{-a.s.}
\end{align}
    %\begin{align}%\label{Stability-inequality}
  %      \Eb\bigg[\sup_{t\in[0,T]}\|\A(\vb_1(t)-\vb_2(t))\|_2^6 \bigg] \leq K \|\fb_1 -\fb_2\|^6_{\Lrm^p(\Omega;\Lrm^2(0,T;\dot{\Lb}^2))}.
 %   \end{align}
\end{proposition}
\begin{proof}
Defining  $\vfrk:=\vb_1-\vb_2$, then the following equation holds 
\begin{equation}\label{Stability}
	\left\{
	\begin{aligned}
		\partial_t \Upsilon(\mathfrak{v}(t)) &= -\left[\mathcal{G}(\vb_1(t), \boldsymbol{0})-\mathcal{G}(\vb_2(t), \boldsymbol{0})\right] + \fb_1-\fb_2, \\
		\mathfrak{v}(0)&= \vb_{0,1}-\vb_{0,2},
	\end{aligned}
	\right.
\end{equation}
for a.e. $t\in[0,T]$ and $\Pb$-a.s. in $\Vb'$, where $\mathcal{G}(\cdot,\cdot)$ is given by \eqref{eqn-mathcal-G}. In view of $\Upsilon(\vb_1-\zb)$, $\Upsilon(\vb_2-\zb)\in \mathrm{L}^2(0,T;\Vb)$ (see \eqref{H3-regularity}) and $\partial_t\Upsilon(\vb_1-\zb)$, $\partial_t\Upsilon(\vb_1-\zb)\in \mathrm{L}^{2}(0,T;\Vb')$ (see \eqref{EE-ubt7}), we infer that  $\Upsilon(\vfrk) \in \mathrm{L}^2(0,T;\Vb)$ and $\partial_t\Upsilon(\vfrk)\in \mathrm{L}^{2}(0,T;\Vb')$. Therefore, an application of \cite[Lemma 1.2, p.176]{temam2001navier} implies $\Upsilon(\vfrk) \in \mathrm{C}([0,T];\Hb)$, the real-valued function $t\mapsto\|\Upsilon(\mathfrak{v}(t))\|_{2}^2$ is absolutely continuous and the following equality is satisfied:
\begin{align*}
    \frac12\frac{\drm}{\drm t}\|\Upsilon(\mathfrak{v}(t))\|_{2}^2 =  \left<\partial_t \Upsilon(\mathfrak{v}(t)),\Upsilon(\mathfrak{v}(t))  \right>,\  \ \ \ \text{ for a.e. } t\in[0,T].
\end{align*}
or equivalently
\begin{align}\label{Stability-1}
   & \frac12 \frac{\drm}{\drm t}\bigg[\|\vfrk(t)\|_{2}^2 + 2\alpha_1\|\nabla\vfrk(t)\|_{2}^2+ \alpha_1^2\|\A\vfrk(t)\|_{2}^2\bigg]
          = - \left<\mathcal{G}(\vb_1(t), \boldsymbol{0})-\mathcal{G}(\vb_2(t), \boldsymbol{0}),\Upsilon(\mathfrak{v}(t)) \right> + \underbrace{\left< \fb_1(t)-\fb_2(t),\Upsilon(\mathfrak{v}(t)) \right>}_{=:V_0},
\end{align}
 for a.e.  $t\in[0,T]$, $\Pb$-a.s. Our next aim is to estimate the right hand side of \eqref{Stability-1} appropriately. Using integration by parts, \eqref{b0}, \eqref{Au-orthogonal}, we infer
\begin{align}\label{Stability-2}
    & - \left< \mathcal{G}(\vb_1, \boldsymbol{0})-\mathcal{G}(\vb_2, \boldsymbol{0}),\Upsilon(\mathfrak{v}) \right>
    \nonumber\\ & = -\nu\|\nabla\vfrk\|_2^2 -\nu\alpha_1\|\A\vfrk\|_2^2 -  \frac{\beta}{2}\int_{\mathbb{T}^2} |\Arm(\vfrk(x))|^2( |\Arm(\vb_1(x))|^2+|\Arm(\vb_2(x))|^2)\drm x 
      \nonumber\\ & \quad +\underbrace{\beta \alpha_1  \left< \diver\{[\Arm(\vfrk):\Arm(\vb_1)]\Arm(\vb_1)+[\Arm(\vfrk):\Arm(\vb_2)]\Arm(\vb_1) + 
      |\Arm(\vb_2)|^2\Arm(\vfrk) \},\A\vfrk \right>}_{=:V_1}
      \nonumber\\ & \quad -  \underbrace{[b(\vfrk,\Upsilon(\vb_1), \Upsilon(\mathfrak{v})) - b(\Upsilon(\mathfrak{v}),\Upsilon(\vb_1) , \vfrk )+ b(\Upsilon(\mathfrak{v}), \vb_2, \Upsilon(\mathfrak{v}))  ]}_{=:V_2}
   - \underbrace{(\alpha_1+\alpha_2) \left<\diver[\Arm(\vb_1)\Arm(\vfrk)+\Arm(\vfrk)\Arm(\vb_2)],\Upsilon(\mathfrak{v}) \right>  }_{=:V_3}.
\end{align}
Using H\"older's, Sobolov's and Young's inequalities, we have 
\begin{align}
    |V_1|& =\bigg|\int_{\Tb^2}\mathfrak{D}^{2}\vfrk\cdot\mathfrak{D}^{1}\vb_1\cdot\mathfrak{D}^{1}\vb_1\cdot\mathfrak{D}^{2}\vfrk\drm x + \int_{\Tb^2}\mathfrak{D}^{1}\vfrk\cdot\mathfrak{D}^{2}\vb_1\cdot \mathfrak{D}^{1}\vb_1\cdot\mathfrak{D}^{2}\vfrk\drm x + \int_{\Tb^2}\mathfrak{D}^{2}\vfrk \cdot \mathfrak{D}^{1}\vb_2 \cdot \mathfrak{D}^{1}\vb_1 \cdot \mathfrak{D}^{2}\vfrk\drm x 
    \nonumber\\ & \qquad + \int_{\Tb^2}\mathfrak{D}^{1}\vfrk \cdot \mathfrak{D}^{2}\vb_2 \cdot \mathfrak{D}^{1}\vb_1 \cdot \mathfrak{D}^{2} \vfrk\drm x+ \int_{\Tb^2}\mathfrak{D}^{1}\vfrk \cdot \mathfrak{D}^{1}\vb_2 \cdot \mathfrak{D}^{2}\vb_1 \cdot \mathfrak{D}^{2}\vfrk\drm x + \int_{\Tb^2}\mathfrak{D}^{2}\vfrk \cdot \mathfrak{D}^{1}\vb_2 \cdot \mathfrak{D}^{1}\vb_2 \cdot \mathfrak{D}^{2}\vfrk\drm x 
    \nonumber\\ & \qquad + \int_{\Tb^2}\mathfrak{D}^{1}\vfrk \cdot \mathfrak{D}^{2}\vb_2 \cdot \mathfrak{D}^{1}\vb_2 \cdot \mathfrak{D}^{2}\vfrk\drm x \bigg|
    \nonumber\\ & \leq C \|\mathfrak{D}^{1}\vb_1\|^2_{\infty} \|\mathfrak{D}^{2}\vfrk\|^2_2 + C \|\mathfrak{D}^{1}\vfrk\|_{4} \|\mathfrak{D}^{2}\vb_1\|_4 \|\mathfrak{D}^{1}\vb_1\|_{\infty} \|\mathfrak{D}^{2}\vfrk\|_2 +  C \|\mathfrak{D}^{1}\vb_2\|_{\infty}  \|\mathfrak{D}^{1}\vb_1\|_{\infty} \|\mathfrak{D}^{2}\vfrk\|^2_{2} 
    \nonumber\\ & \qquad +  C \|\mathfrak{D}^{1}\vfrk\|_4 \|\mathfrak{D}^{2}\vb_2\|_4 \|\mathfrak{D}^{1}\vb_1\|_{\infty} \|\mathfrak{D}^{2} \vfrk\|_2 + C \|\mathfrak{D}^{1}\vfrk\|_4 \|\mathfrak{D}^{1}\vb_2\|_{\infty} \|\mathfrak{D}^{2}\vb_1\|_4 \|\mathfrak{D}^{2}\vfrk\|_2 +  C\|\mathfrak{D}^{1}\vb_2\|_{\infty}^2 \|\mathfrak{D}^{2}\vfrk\|_2^2 
    \nonumber\\ & \qquad + C \|\mathfrak{D}^{1}\vfrk\|_4 \|\mathfrak{D}^{2}\vb_2\|_4 \|\mathfrak{D}^{1}\vb_2\|_{\infty} \|\mathfrak{D}^{2}\vfrk\|_2 
    \nonumber\\ & \leq C \big[\|\vb_1\|^2_{\dot{\Hb}^3}+\|\vb_2\|^2_{\dot{\Hb}^3}\big] \|\A\vfrk\|^2_2,\label{Stability-3} \\
    |V_2+V_3|& =\bigg|\int_{\Tb^2}\mathfrak{D}^{0}\vfrk\cdot\mathfrak{D}^{3}\vb_1\cdot\mathfrak{D}^{2}\vfrk\drm x + \int_{\Tb^2}\mathfrak{D}^{2}\vfrk\cdot\mathfrak{D}^{1}\vb_2\cdot\mathfrak{D}^{2}\vfrk\drm x +  \int_{\Tb^2}\mathfrak{D}^{2}\vb_1\cdot\mathfrak{D}^{1}\vfrk \cdot \mathfrak{D}^{2}\vfrk\drm x  +  \int_{\Tb^2}\mathfrak{D}^{1}\vb_1\cdot\mathfrak{D}^{2} \vfrk \cdot \mathfrak{D}^{2}\vfrk\drm x 
    \nonumber\\ & \qquad  +  \int_{\Tb^2}\mathfrak{D}^{1}\vfrk\cdot\mathfrak{D}^{2}\vb_2 \cdot \mathfrak{D}^{2}\vfrk\drm x  \bigg|
    \nonumber\\ & \leq C \|\vfrk\|_{\infty}\|\mathfrak{D}^{3}\vb_1\|_2\|\mathfrak{D}^{2}\vfrk\|_2 + C \|\mathfrak{D}^{1}\vb_2\|_{\infty} \|\mathfrak{D}^{2}\vfrk\|_2^2 +  C \|\mathfrak{D}^{2}\vb_1\|_4 \|\mathfrak{D}^{1}\vfrk\|_4 \|\mathfrak{D}^{2}\vfrk\|_2   +  C \|\mathfrak{D}^{1}\vb_1\|_{\infty} \|\mathfrak{D}^{2} \vfrk\|_2^2 
    \nonumber\\ & \qquad  +   C  \|\mathfrak{D}^{1}\vfrk\|_4 \|\mathfrak{D}^{2}\vb_1\|_4 \|\mathfrak{D}^{2}\vfrk\|_2 
    \nonumber\\ & \leq C \big[ 1+ \|\vb_1\|^2_{\dot{\Hb}^3}+\|\vb_2\|^2_{\dot{\Hb}^3}\big] \|\A\vfrk\|^2_2,\label{Stability-4}\\ 
    |V_0| & \leq  C\|\fb_1-\fb_2\|_2^2 +  C \|\A\vfrk\|^2_2.\label{Stability-5}
\end{align}
In view of \eqref{Stability-1}-\eqref{Stability-5}, we write 
\begin{align}\label{Stability-6}
   & \frac12 \frac{\drm}{\drm t}\bigg[\|\vfrk(t)\|_{2}^2 + 2\alpha_1\|\nabla\vfrk(t)\|_{2}^2+ \alpha_1^2\|\A\vfrk(t)\|_{2}^2\bigg]
   \nonumber\\ & \leq C\|\fb_1(t)-\fb_2(t)\|_2^2 +  C \big[ 1+ \|\vb_1(t)\|^2_{\dot{\Hb}^3}+\|\vb_2(t)\|^2_{\dot{\Hb}^3}\big] \bigg[\|\vfrk(t)\|_{2}^2 + 2\alpha_1\|\nabla\vfrk(t)\|_{2}^2+ \alpha_1^2\|\A\vfrk(t)\|_{2}^2\bigg],
\end{align}
for a.e.  $t\in[0,T]$ and $\Pb$-a.s. Hence, applying Gronwall's inequality, we complete the proof.
\iffalse
obtain
\begin{align}
    %\|\vfrk(t)\|_{2}^2 + 2\alpha_1\|\nabla\vfrk(t)\|_{2}^2+ \alpha_1^2
    \sup_{t\in[0,T]}\|\A\vfrk(t)\|_{2}^6 
    & \leq C e^{C\int_0^T \big[ 1+ \|\vb_1(t)\|^2_{\dot{\Hb}^3}+\|\vb_2(t)\|^2_{\dot{\Hb}^3}\big]\drm s} \|\fb_1-\fb_2\|^6_{\Lrm^2(0,T;\dot{\Lb}^2)}.
\end{align}
$\Pb$-a.s., which, using H\"older's inequality and \eqref{eqn-exponential-moments} give
\begin{align*}
    %\|\vfrk(t)\|_{2}^2 + 2\alpha_1\|\nabla\vfrk(t)\|_{2}^2+ \alpha_1^2
   \Eb \bigg[\sup_{t\in[0,T]}\|\A\vfrk(t)\|_{2}^6 \bigg]
    & \leq C \Eb \bigg[ e^{C\int_0^T \big[ 1+ \|\vb_1(t)\|^2_{\dot{\Hb}^3}+\|\vb_2(t)\|^2_{\dot{\Hb}^3}\big]\drm s} \|\fb_1-\fb_2\|^6_{\Lrm^2(0,T;\dot{\Lb}^2)}\bigg]
     \leq C \|\fb_1 -\fb_2\|^6_{\Lrm^p(\Omega;\Lrm^2(0,T;\dot{\Lb}^2))}.
\end{align*}
This completes the proof.\fi
\end{proof}

Let us now prove the main result of this section in the following proposition.

\begin{proposition}\label{prop-GD}
Suppose that $\vb_{0}\in \Drm(\A^{\frac{3}{2}})$, $\boldsymbol{\psi}\in \Lrm^2(\Omega;\mathrm{L}^2(0,T;\dot{\Hb}^{1}(\mathbb{T}^2)))\cap\Lrm^p(\Omega;\mathrm{L}^2(0,T;\dot{\Lb}^{2}(\mathbb{T}^2)))$ (for some $p>4$) and $\fb \in \Lrm^2(\Omega;\mathrm{L}^2(0,T;\dot{\Hb}^{1}(\mathbb{T}^2)))$. Let $\vb$ and $\vb_{\rho}$ be two solutions of system \eqref{STGF-Projected} in the sense of Definition \ref{def-Solution} with same initial data $\vb_0$ and corresponding to external forcing $\fb$ and $\fb_{\rho}:=\fb+\rho \boldsymbol{\psi}$, $\rho\in(0,1)$, respectively, then 
\begin{align}
    \chi_{\rho}:=  \frac{\vb_\rho - \vb }{\rho} -\mfrk, \;\; \text{ satisfies }\;\; \lim_{\rho\to 0}\Eb\left[\sup_{t\in[0,T]}\bigg[\|\chi_{\rho}(t)\|_2^2+\alpha_1 \|\nabla\chi_{\rho}(t)\|_2^2\bigg]\right] = 0,
\end{align}
where $\mfrk$ is the solution to system \eqref{eqn-linearized} in the sense of Definition of \ref{def-LSTGF}.
\end{proposition}
\begin{proof}
    Since $\vb_{\rho}$ and $\vb$ are the solution of system \eqref{STGF-Projected} corresponding to $\fb+\rho \boldsymbol{\psi}$ and $\fb$, respectively, and $\mfrk$ is the solution of system \eqref{eqn-linearized}, therefore $\chi_{\rho} :=\mfrk_\rho -\mfrk :=   \frac{\vb_\rho - \vb }{\rho} -\mfrk$ satisfies
    \begin{equation}\label{TGF-chi-rho}
	\left\{
	\begin{aligned}
\partial_t\Upsilon(\chi_{\rho})+\nu\A\chi_{\rho} 
        + \rho\Pcal[(\mfrk_{\rho}\cdot\nabla)\Upsilon(\mfrk_{\rho})] + \Pcal[(\vb\cdot\nabla)\Upsilon(\chi_{\rho})] + \Pcal[(\chi_{\rho}\cdot\nabla)\Upsilon(\vb)]
        \\ + \rho\Pcal\left[\sum_{j=1}^2 [\Upsilon(\mfrk_{\rho})]^j \nabla \mfrk_{\rho}^j \right]  + \Pcal\left[\sum_{j=1}^2\{ [\Upsilon(\vb)]^j\nabla\chi_{\rho}^j + [\Upsilon(\chi_{\rho})]^j\nabla \vb^j\}\right]
		- (\alpha_1+\alpha_2)\rho \mathcal{P}\diver[\Arm(\mfrk_\rho)\Arm(\mfrk_\rho)] 
	\\	- (\alpha_1+\alpha_2) \mathcal{P}\diver[\Arm(\chi_\rho)\Arm(\vb)+\Arm(\vb)\Arm(\chi_\rho)]
		-\beta\rho^2 \mathcal{P}\diver[|\Arm(\mfrk_\rho)|^2\Arm(\mfrk_\rho)] -\beta\rho \mathcal{P}\diver[|\Arm(\mfrk_\rho)|^2\Arm(\vb)] 
		\\ 
		- 2\beta \rho \mathcal{P}\diver [(\Arm(\mfrk_\rho):\Arm(\vb))\Arm(\mfrk_\rho)]
		- 2\beta  \mathcal{P}\diver [(\Arm(\chi_\rho):\Arm(\vb))\Arm(\vb)]
		-\beta \mathcal{P}\diver[|\Arm(\vb)|^2\Arm(\chi_\rho)]
		& = \boldsymbol{0},\\
		\chi_\rho(0)&=\boldsymbol{0},
	\end{aligned}
	\right.
\end{equation}
for a.e. $t\in[0, {T}]$ and $\Pb$-a.s., in $\Vb^{\prime}$. Taking the inner product by $\chi_{\rho}$ to the equation $\eqref{TGF-chi-rho}_1$, and using the integration by parts, \eqref{b0} and \eqref{Au-orthogonal}, we have 
\begin{align}\label{eqn-GD1}
  &  \frac12 \frac{\drm}{\drm t}\bigg[\|\chi_\rho(t)\|^2_{2} + \alpha_1 \|\nabla\chi_\rho(t)\|^2_2\bigg] + \nu\|\nabla\chi_\rho(t)\|^2_2 + \frac{\beta}{2}\||\Arm(\vb(t))| \Arm(\chi_\rho(t))\|^2_2   + \beta\int_{\mathbb{T}^2}[\Arm(\vb(x,t)):\Arm(\chi_\rho (x,t))]^2\drm x
    \nonumber\\ & =  \underbrace{\rho b(\mfrk_{\rho}(t), \chi_\rho(t), \Upsilon(\mfrk_\rho(t)))}_{=:X_1} + \underbrace{\alpha_1 b(\vb(t), \chi_\rho(t), \A\chi_{\rho}(t)) }_{=:X_2}
    - \underbrace{ \alpha_1 \rho b(\chi_{\rho}(t), \mfrk_{\rho}(t), \A\mfrk_\rho(t)) }_{=:X_3}
      - \underbrace{b(\chi_\rho(t), \vb(t), \Upsilon(\chi_{\rho}(t)))}_{=:X_4} 
      \nonumber\\ & \quad - \underbrace{\frac{\alpha}{2} \int_{\mathbb{T}^2}[\Arm(\chi_\rho(x,t))\Arm(\vb(x,t)) + \Arm(\vb(x,t))\Arm(\chi_\rho(x,t))]:\Arm(\chi_\rho(x,t))\drm x }_{=:X_5}  \nonumber\\ & \quad - \underbrace{ \frac{\alpha\rho}{2} \int_{\mathbb{T}^2} \Arm(\mfrk_\rho(x,t))\Arm(\mfrk_\rho(x,t)):\Arm(\chi_\rho(x,t))\drm x }_{=:X_6}
   % -  \underbrace{\frac{\beta\rho^2}{2}\int_{\mathbb{T}^2}|\Arm(\mfrk_\rho(x,t))|^2[\Arm(\mfrk_\rho(x,t)):\Arm(\chi_\rho(x,t))]\drm x }_{=:X_7}
     - \underbrace{ \frac{\beta\rho}{2}\int_{\mathbb{T}^2}|\Arm(\mfrk_\rho(x,t))|^2 [\Arm(\vb_{\rho}(x,t)):\Arm(\chi_\rho(x,t))]\drm x }_{=:X_7}
     \nonumber\\ & \quad - \underbrace{\beta\rho  \int_{\mathbb{T}^2}[\Arm(\mfrk_\rho(x,t)):\Arm(\vb(x,t))][\Arm(\mfrk_\rho(x,t)):\Arm(\chi_\rho(x,t))]\drm x}_{=:X_8},
\end{align}
for a.e. $t\in[0,{T}]$ and $\Pb$-a.s. Let us now estimate each term of right hand side of \eqref{eqn-GD1} using H\"older's, Sobolev's and Young's inequalities as follows:
\begin{align}
    |X_1| & \leq C\rho \|\mfrk_{\rho}\|_{\infty}\|\nabla\chi_\rho\|_{2}\|\A\mfrk_{\rho}\|_2 \leq  C\rho \|\nabla\chi_\rho\|_{2}\|\A\mfrk_{\rho}\|^2_2 \leq  C\|\nabla\chi_\rho\|^2_{2}+  C\rho^2\|\A\mfrk_{\rho}\|^4_2, \label{eqn-GD2}\\
    |X_2| & = \left|\int_{\Tb^2}\mathfrak{D}^1\vb\cdot\mathfrak{D}^1\chi_\rho\cdot\mathfrak{D}^1\chi_\rho\drm x \right| \leq  C \|\nabla\vb\|_{\infty}\|\nabla\chi_\rho\|_2^2 \leq  C \|\vb\|_{\dot{\Hb}^3}\|\nabla\chi_\rho\|_2^2,\label{eqn-GD3}\\
    |X_3| & \leq C\rho \|\chi_\rho\|_4\|\nabla\mfrk_\rho\|_4\|\A\mfrk_\rho\|_2 \leq C\rho \|\nabla\chi_\rho\|_2\|\A\mfrk_\rho\|_2^2 \leq  C\|\nabla\chi_\rho\|^2_{2}+  C\rho^2\|\A\mfrk_{\rho}\|^4_2,\label{eqn-GD4}\\
    |X_4| & = \left|\int_{\Tb^2}\mathfrak{D}^1\chi_\rho\cdot\mathfrak{D}^1\vb\cdot\mathfrak{D}^1\chi_\rho\drm x + \int_{\Tb^2}\mathfrak{D}^0\chi_\rho\cdot\mathfrak{D}^2\vb\cdot\mathfrak{D}\chi_\rho\drm x \right| 
     \nonumber\\ & \leq C \|\nabla\chi_\rho\|_{2}\|\nabla\vb\|_{\infty}\|\nabla\chi_\rho\|_2 + C \|\chi_\rho\|_{4}\|\mathfrak{D}^2\vb\|_{4} \|\nabla\chi_\rho\|_2
      \leq C \|\vb\|_{\dot{\Hb}^3}\|\nabla\chi_\rho\|_2^2,\label{eqn-GD5}\\
      |X_5| &  \leq C \|\nabla\vb\|_{\infty}\|\nabla\chi_\rho\|_2^2 \leq C\|\vb\|_{\dot{\Hb}^3}\|\nabla\chi_\rho\|_2^2,\label{eqn-GD6}\\
      |X_6| &  \leq C \rho \|\nabla\mfrk_\rho\|_{4}^2\|\nabla\chi_\rho\|_2 \leq  C\rho \|\nabla\chi_\rho\|_{2}\|\A\mfrk_{\rho}\|^2_2 \leq  C\|\nabla\chi_\rho\|^2_{2}+  C\rho^2\|\A\mfrk_{\rho}\|^4_2, \label{eqn-GD7}\\
    %  |X_7| &  \leq C \rho^2 \|\nabla\mfrk_\rho\|_{6}^3\|\nabla\chi_\rho\|_2 \leq  C\rho^2  \|\nabla\chi_\rho\|_{2}\|\A\mfrk_{\rho}\|^3_2 \leq  C\|\nabla\chi_\rho\|^2_{2}+  C\rho^4\|\A\mfrk_{\rho}\|^6_2, \label{eqn-GD8}\\
      |X_7+X_8| &  \leq C \rho [\|\nabla\vb\|_{\infty}+\|\nabla\vb_{\rho}\|_{\infty}]\|\nabla\mfrk_\rho\|_{4}^2\|\nabla\chi_\rho\|_2 \leq C \rho [\|\vb\|_{\dot{\Hb}^3}+\|\vb_{\rho}\|_{\dot{\Hb}^3}]\|\nabla\mfrk_\rho\|_{4}^2\|\nabla\chi_\rho\|_2 
      \nonumber\\ & \leq C [\|\vb\|^2_{\dot{\Hb}^3}+\|\vb_{\rho}\|^2_{\dot{\Hb}^3}] \|\nabla\chi_\rho\|_2^2 + C \rho^2 \|\A\mfrk_\rho\|_{2}^4 \label{eqn-GD9}.
\end{align}
Combining \eqref{eqn-GD1}-\eqref{eqn-GD9}, we reach at
\begin{align}\label{eqn-GD10}
    & \frac{\drm}{\drm t}\bigg[\|\chi_\rho(t)\|^2_{2} + \alpha_1 \|\nabla\chi_\rho(t)\|^2_2\bigg]
      \leq C\big\{1+\|\vb(t)\|^2_{\dot{\Hb}^3}+\|\vb_{\rho}(t)\|^2_{\dot{\Hb}^3}\big\}\bigg[\|\chi_\rho(t)\|^2_{2} + \alpha_1 \|\nabla\chi_\rho(t)\|^2_2\bigg]+  C\rho^2\|\A\mfrk_{\rho}(t)\|^4_2.
\end{align}
for a.e. $t\in[0,T]$ and $\Pb$-a.s. Applying Gronwall's inequality to \eqref{eqn-GD10} and using \eqref{Stability-inequality}, we obtain
\begin{align}\label{eqn-GD11}
     \|\chi_\rho(t)\|^2_{2} + \alpha_1 \|\nabla\chi_\rho(t)\|^2_2
      & \leq C \rho^2e^{C\int_0^t[1+\|\vb(s)\|^2_{\dot{\Hb}^3}+\|\vb_{\rho}(s)\|^2_{\dot{\Hb}^3}]\drm s }\cdot\|\A\mfrk_{\rho}(t)\|^4_2
      \nonumber\\ & \leq C \rho^2e^{C\int_0^t[1+\|\vb(s)\|^2_{\dot{\Hb}^3}+\|\vb_{\rho}(s)\|^2_{\dot{\Hb}^3}]\drm s }\cdot \|\boldsymbol{\psi}\|^4_{\Lrm^2(0,T;\dot{\Lb}^2)} .
\end{align}
for all $t\in[0,T]$ and $\Pb$-a.s. Taking expectation of \eqref{eqn-GD11}, and using H\"older's inequality, \eqref{eqn-vb-exponential-moments} and the fact that $\boldsymbol{\psi}\in \Lrm^p(\Omega;\mathrm{L}^2(0,T;\dot{\Lb}^{2}(\mathbb{T}^2)))$ (for some $p>4$), we complete the proof.
\end{proof}

\subsection{Variation of the cost functional \texorpdfstring{\eqref{eqn-cost-functional}}{} }
As a consequence of Proposition \ref{prop-GD}, we get the following result on the variation for the cost functional \eqref{eqn-cost-functional}.
\begin{proposition}\label{vari-cost}
	Suppose that $\vb_{0}\in \Drm(\A^{\frac{3}{2}})$, $\boldsymbol{\psi}\in \Lrm^2(\Omega;\mathrm{L}^2(0,T;\dot{\Hb}^{1}(\mathbb{T}^2)))\cap\Lrm^p(\Omega;\mathrm{L}^2(0,T;\dot{\Lb}^{2}(\mathbb{T}^2)))$ (for some $p>4$) and $\fb \in \Lrm^2(\Omega;\mathrm{L}^2(0,T;\dot{\Hb}^{1}(\mathbb{T}^2)))$. Let $\vb , \vb_\rho$ be the solutions of system \eqref{STGF-Projected} corresponding to external forcing $\fb$ and $\fb_\rho:=\fb+\rho\boldsymbol{\psi}$, $\rho\in(0,1)$, respectively, in the sense of Definition \ref{def-Solution} and $\mfrk$ be the solution of system \eqref{eqn-linearized} in the sense of Definition \ref{def-LSTGF}. Then, we have 
%	{\color{red} 
		\begin{equation*}
		%	\begin{array}{ll}
		\mathrm{J}(\fb_\rho,\vb_\rho)=\mathrm{J}(\fb,\vb) + \rho \; \Eb\left[\int_0^{T}\{(\nabla_{\fb}\mathfrak{L}(t,\fb(t),\vb(t)),\boldsymbol{\psi}(t))+(\nabla_{\vb}\mathfrak{L}(t,\fb(t),\vb(t)),\mfrk(t))\}\drm t\right]+o(\rho).
		%\end{array}
	\end{equation*}
%}
\end{proposition}

\section{Adjoint system}\label{Sec-adjoint}\setcounter{equation}{0}

 The main goal of this section is to show that, for $\gb\in \Lrm^p(\Omega; \mathrm{L}^2(0,T; \dot{\Lb}^2(\mathbb{T}^2)))$ (for some $p>2$) and  the unique solution $\vb$ to the state equation \eqref{STGF-Projected}, as defined in Definition \ref{def-ranndom-TGF}, satisfying \eqref{eqn-vb-exponential-moments}, the following adjoint system is well-posed: 
\begin{equation}
	\left\{
	\begin{aligned}
		-\partial_t (\pb - \alpha_1 \Delta \pb) & =   \gb -\nabla\pi^{\ast} + \nu\Delta \pb + (\Irm-\alpha_1\Delta)[(\vb\cdot\nabla)\pb] - \sum_{j=1}^2\pb^j \nabla(\vb-\alpha_1\Delta\vb)^j 
          - (\Irm-\alpha_1\Delta)[(\pb\cdot\nabla)\vb] 
          \\   & \quad   + (\pb\cdot\nabla)(\vb-\alpha_1\Delta\vb)    + (\alpha_1+\alpha_2) \mathrm{div}(\Arm(\vb) \Arm(\pb)+ \Arm(\pb)\Arm(\vb))   +\beta [\mathrm{div}(\vert \Arm(\vb)\vert^{2}\Arm(\pb))]
         \\
		 & \quad +2\beta  {\rm div}[(\Arm(\pb):\Arm(\vb))\Arm(\vb)],&& \hspace{-30mm}\text{in }  \mathbb{T}^2 \times (0,T],\\
		\text{div}\; \pb&=0, \quad && \hspace{-30mm} \text{in } \mathbb{T}^2 \times [0,{T}],\\
		%		\z &= \boldsymbol{0} &&\text{on } \partial\mathbb{T}^2\times [0,\infty),\\
		\pb(x,{T})&=\boldsymbol{0}, \quad && \hspace{-30mm} \text{in } \mathbb{T}^2,
	\end{aligned}
	\right.\label{eqn-adjoint-pi}
\end{equation}
with periodic boundary conditions, where $\pb =(\pb^{1},\pb^2)$ and $\pi^{\ast}$ are unknown vector and scalar fields, receptively, and $\Arm(\pb)= \nabla \pb + (\nabla \pb)^{T}$. Taking the projection $\mathcal{P}$ to the system \eqref{eqn-adjoint-pi}, we obtain
\begin{equation}
	\left\{
	\begin{aligned}
		-\partial_t \Upsilon(\pb) &= \mathcal{P}\gb - \nu\A \pb + \Pcal [(\Irm-\alpha_1\Delta)[(\vb\cdot\nabla)\pb]] - \Pcal\left[\sum_{j=1}^2\pb^j \nabla[\Upsilon(\vb)]^j \right]
          - \Pcal\left[(\Irm-\alpha_1\Delta)[(\pb\cdot\nabla)\vb] \right]
             + \Pcal[(\pb\cdot\nabla)\Upsilon(\vb)] 
             \\   & \quad + (\alpha_1+\alpha_2) \mathcal{P} [\mathrm{div}(\Arm(\vb) \Arm(\pb)+ \Arm(\pb)\Arm(\vb))] 
           +\beta \mathcal{P}[\mathrm{div}(\vert \Arm(\vb)\vert^{2}\Arm(\pb))]
             +2\beta \mathcal{P}[ {\rm div}[(\Arm(\pb):\Arm(\vb))\Arm(\vb)]],
		\\ 
		%\mathrm{div}\;\z&=0 && \mbox{in}\
		%	\mathbb{T}^2\times (0,T), \vspace{2mm}\\
		\pb({T})&=\boldsymbol{0},
	\end{aligned}
	\right.\label{eqn-adjoint}
\end{equation}
in $\Vb^{\prime}$. We obtain the well-posedness of system \eqref{eqn-adjoint} in the following sense:
\begin{definition}\label{def-AdSTGF}
	A stochastic process  $\{\pb(t)\}_{t\in[0,T]}$ with trajectories in  $\mathrm{C}([0,T]; \Vb)\cap \mathrm{L}^{\infty}(0,{T};\Drm(\A))$,  $\partial_t\pb \in\mathrm{L}^{2}(0,{T};\Vb)$ and $\pb(T)=\boldsymbol{0}$, 
	is called a \emph{solution} to system \eqref{eqn-adjoint}, if for $\gb\in \Lrm^p(\Omega; \mathrm{L}^2(0,T; \dot{\Lb}^2(\mathbb{T}^2)))$ (for some $p>2$), it satisfies for any ${\phi}\in \Vb,$ 
		\begin{align}\label{W-AdTGF}
			-\left<\partial_t\Upsilon(\pb(t)), {\phi}\right>
			  &=   \big\langle   \gb(t) - \nu\A \pb(t) + (\Irm-\alpha_1\Delta)[(\vb\cdot\nabla)\pb] - \sum_{j=1}^2\pb^j \nabla[\Upsilon(\vb)]^j 
          - (\Irm-\alpha_1\Delta)[(\pb\cdot\nabla)\vb] 
           + (\pb\cdot\nabla)\Upsilon(\vb)  
           \nonumber \\   & \quad  + (\alpha_1+\alpha_2)  \mathrm{div}(\Arm(\vb(t)) \Arm(\pb(t))+ \Arm(\pb(t))\Arm(\vb(t)))  
           +\beta \mathrm{div}(\vert \Arm(\vb(t))\vert^{2}\Arm(\pb(t))) 
            \nonumber \\ & \quad  + 2\beta  {\rm div}((\Arm(\pb(t)):\Arm(\vb(t)))\Arm(\vb(t))) , {\phi} \big\rangle,
		\end{align}
		for a.e. $t\in[0,{T}]$ and $\Pb$-a.s.
\end{definition}

Next theorem is the main result of this section which provides the existence and uniqueness of global solutions to system \eqref{eqn-adjoint} in the sense of Definition \ref{def-AdSTGF}.

\begin{theorem}\label{solution-Ad}
	There exists a unique solution $\pb$ to system \eqref{eqn-adjoint} in the sense of Definition \ref{def-AdSTGF}. 
\end{theorem}
\begin{proof}[Proof of Theorem \ref{solution-Ad}]
	Notice  that $\pb$ is the solution of \eqref{eqn-adjoint-pi} if and only if $\qb(t)=\pb({T}-t)$ is the solution of the following initial value problem with $\bar \vb(t)= \vb({T}-t), \bar \gb(t)= \gb({T}-t)$ and $ \bar \pi^{\ast}(t) = \pi^{\ast}({T}-t)$
	\begin{equation}
		\left\{
		\begin{aligned}
			\partial_t (\qb - \alpha_1 \Delta \qb) & =  \bar\gb -\nabla\bar\pi^{\ast} + \nu\Delta \qb + (\Irm-\alpha_1\Delta)[(\bar\vb\cdot\nabla)\qb] - \sum_{j=1}^2\qb^j \nabla(\bar\vb-\alpha_1\Delta\bar\vb)^j 
          - (\Irm-\alpha_1\Delta)[(\qb\cdot\nabla)\bar\vb] 
          \\   & \quad   + (\qb\cdot\nabla)(\bar\vb-\alpha_1\Delta\bar\vb)  + (\alpha_1+\alpha_2) \mathrm{div}(\Arm(\bar\vb) \Arm(\qb)+ \Arm(\qb)\Arm(\bar\vb))  \\
			  & \quad +\beta [\mathrm{div}(\vert \Arm(\bar\vb)\vert^{2}\Arm(\qb))]+2\beta  {\rm div}[(\Arm(\qb):\Arm(\bar\vb))\Arm(\bar\vb)],&&\hspace{-20mm} \text{in }  \mathbb{T}^2 \times (0,{T}],\\
			\text{div}\; \qb&=0, \quad && \hspace{-20mm} \text{in } \mathbb{T}^2 \times [0,{T}],\\
			\qb(x,0)&=\boldsymbol{0}, \quad && \hspace{-20mm} \text{in } \mathbb{T}^2,
		\end{aligned}
		\right.\label{eqn-adjoint-pi-q}
	\end{equation}
	with periodic boundary conditions. According to the  Definition \ref{def-AdSTGF}, $\qb$ is the solution of system \eqref{eqn-adjoint-pi-q} if 
	$\qb \in \mathrm{C}([0,{T}]; \Vb)\cap \mathrm{L}^{\infty}(0,{T};\Drm(\A))$ with $\partial_t\qb \in \mathrm{L}^{2}(0,{T};\Vb^{\prime})$, $\qb(0)=\boldsymbol{0}$ and, in addition, the following equality holds, for all  $\boldsymbol{\phi} \in \Vb$,
	\begin{align}\label{W-AdTGF-q}
		\left<\partial_t\Upsilon(\qb), \boldsymbol{\phi}\right>
		&=   \big\langle  \bar \gb - \nu\A \qb + (\Irm-\alpha_1\Delta)[(\bar\vb\cdot\nabla)\qb] - \sum_{j=1}^2\qb^j \nabla[\Upsilon(\bar\vb)]^j 
          - (\Irm-\alpha_1\Delta)[(\qb\cdot\nabla)\bar\vb] 
           + (\qb\cdot\nabla)\Upsilon(\bar\vb) 
           \nonumber \\   & \quad  + (\alpha_1+\alpha_2)  \mathrm{div}(\Arm(\bar\vb) \Arm(\qb)+ \Arm(\qb)\Arm(\bar\vb))  
           +\beta \mathrm{div}(\vert \Arm(\bar\vb)\vert^{2}\Arm(\qb))  + 2\beta  {\rm div}((\Arm(\qb):\Arm(\bar\vb))\Arm(\bar\vb)) , \boldsymbol{\phi} \big\rangle,
	\end{align}
for a.e. $t\in[0,{T}]$ and $\Pb$-a.s.
 
Let us consider the following approximate equation for system \eqref{W-AdTGF-q} on the finite-dimensional space $\Hb_n$:
\begin{equation}\label{finite-dimS-adjoint-q}
	\left\{
	\begin{aligned}
    \partial_t\Upsilon(\qb_n(t))
    &=  \Prm_n\mathcal{P}\bar\gb(t) -\nu \Prm_n\A\qb_n(t) - \Prm_n\Pcal [(\Irm-\alpha_1\Delta)[(\bar\vb(t)\cdot\nabla)\qb_n(t)]] - \Prm_n\Pcal\bigg[\sum_{j=1}^2(\qb_n)^j \nabla[\Upsilon(\bar\vb(t))]^j \bigg]
          \\   & \quad - \Prm_n\Pcal [(\Irm-\alpha_1\Delta)[(\qb_n(t)\cdot\nabla)\bar\vb(t)]] 
            + \Prm_n\Pcal[(\qb_n(t)\cdot\nabla)\Upsilon(\bar\vb(t))]\\ & \quad + (\alpha_1+\alpha_2) \Prm_n \mathcal{P} [\mathrm{div}(\Arm(\bar\vb(t)) \Arm(\qb_n(t))+ \Arm(\qb_n(t))\Arm(\bar\vb(t)))] \\
		& \quad + \beta \Prm_n \mathcal{P}[\mathrm{div}(\vert \Arm(\bar\vb(t))\vert^{2}\Arm(\qb_n(t)))]  + 2\beta \Prm_n\mathcal{P}[ {\rm div}[(\Arm(\qb_n(t)):\Arm(\bar\vb(t)))\Arm(\bar\vb(t))]],
		\\
		\qb_n(0)&=\boldsymbol{0},
	\end{aligned}
	\right.
\end{equation}
Since system \eqref{eqn-adjoint-pi-q} closely resembles system \eqref{eqn-linearized-pi}, applying the same reasoning as in the proof of Theorem \ref{solution-L} establishes the existence and uniqueness of its solution. Consequently, this also guarantees the existence and uniqueness of a solution for system \eqref{eqn-adjoint-pi} with the stated regularity. Ultimately, one also obtain 
\begin{equation}\label{Moment-4}
	%	\left.
        \begin{aligned}
		%	\pb_n\xrightharpoonup{w^*}&\ \pb &&\text{ in }\ \ \ \ \	\Lrm^2(\Omega;\mathrm{L}^{\infty}(0,{T};\Drm(\A))),\\
			\pb_n\xrightharpoonup{w}&\ \pb &&\text{ in }\ \ \ \ \	\Lrm^2(\Omega;\mathrm{L}^{2}(0,{T};\Drm(\A))).
		\end{aligned}%\right\}
	\end{equation}
    This completes the proof.
\end{proof}

\section{Existence of optimal control and optimality condition}\label{Sec-Existence-OC}\setcounter{equation}{0}

This section establishes a duality relationship between the solutions of the adjoint equation and the linearized equation, demonstrating the existence of a solution to the control problem. Using this duality, we further show that the control problem's solution satisfies the first-order optimality criterion.

\subsection{Duality property}
In the following proposition, we establish a duality relation between the solutions of the adjoint equation and the linearized equation.
\begin{proposition}\label{duality-prop}
	Let $\vb$ satisfies \eqref{eqn-vb-exponential-moments} and $\gb,\boldsymbol{\psi} \in \Lrm^p(\Omega; \mathrm{L}^2(0,T; \dot{\Lb}^2(\mathbb{T}^2)))$ (for some $p>2$). Then we have
	\begin{align*}
		\Eb\left[\int_0^{{T}}(\boldsymbol{\psi}(t),\pb(t))\drm t\right] = \Eb\left[ \int_0^{{T}} (\gb(t),\mfrk(t))\drm t\right], 
	\end{align*}
	where $\pb$ is the solution of \eqref{eqn-adjoint-pi} and $\mfrk$ is the solution of \eqref{eqn-linearized-pi}.
\end{proposition}

\begin{proof} 
	Let us consider the following approximate equations corresponding to systems \eqref{eqn-linearized} and \eqref{eqn-adjoint} on the finite-dimensional space $\Hb_n$, respectively:
	\begin{equation}\label{finite-dimS-linerarized-IP}
		\left\{
		\begin{aligned}
			(\partial_t\Upsilon(\mfrk_n), \wb_n) & =  (\boldsymbol{\psi} -\nu \A\mfrk_n  -(\vb\cdot\nabla)\Upsilon(\mfrk_n) - (\mfrk_n\cdot \nabla) \Upsilon(\vb)
          - \displaystyle\sum_{j=1}^2[\Upsilon(\mfrk_n)]^j\nabla \vb^j 
         - \displaystyle\sum_{j=1}^2[\Upsilon(\vb)]^j\nabla [\mfrk_n]^j 
          \\
        & \quad  + (\alpha_1+\alpha_2) [\mathrm{div}(\Arm(\vb) \Arm(\mfrk_n)+ \Arm(\mfrk_n)\Arm(\vb))]  
        +\beta \mathrm{div}(\vert \Arm(\vb)\vert^{2}\Arm(\mfrk_n))
        \\
		 & \quad +2\beta {\rm div}((\Arm(\mfrk_n):\Arm(\vb))\Arm(\vb)), \wb_n),
			\\
			\mfrk_n(0)&=\boldsymbol{0},
		\end{aligned}
		\right.
	\end{equation}
for a.e. $t\in[0,{T}]$ and for all $\wb_n\in\Hb_n$, $\Pb$-a.s., and 
	\begin{equation}\label{finite-dimS-adjoint-IP}
		\left\{
		\begin{aligned}
		-(\partial_t\Upsilon(\pb_n), \wb_n) & =  (\gb -\nu \A\pb_n - (\Irm-\alpha_1\Delta)[(\vb\cdot\nabla)\pb_n] 
         - \sum_{j=1}^2[\pb_n]^j \nabla[\Upsilon(\vb)]^j 
            - (\Irm-\alpha_1\Delta)[(\pb_n\cdot\nabla)\vb] 
          \\ & \quad  + (\pb_n\cdot\nabla)\Upsilon(\vb)
             + (\alpha_1+\alpha_2) \mathrm{div}(\Arm(\vb) \Arm(\pb_n)+ \Arm(\pb_n)\Arm(\vb))  \\
			& \quad + \beta \mathrm{div}(\vert \Arm(\vb)\vert^{2}\Arm(\pb_n))  + 2\beta  {\rm div}((\Arm(\pb_n):\Arm(\vb))\Arm(\vb)), \wb_n),
			\\
			\pb_n({T})&=\boldsymbol{0},
		\end{aligned}
		\right.
	\end{equation}
	for a.e. $t\in[0,{T}]$ and for all $\wb_n\in\Hb_n$, $\Pb$-a.s.

	Setting $\wb_n=\pb_n$ in $\eqref{finite-dimS-linerarized-IP}_1$, we get
	\begin{align}\label{Duality1}
		(\partial_t\Upsilon(\mfrk_n), \pb_n) & =  (\boldsymbol{\psi} -\nu \A\mfrk_n  - (\vb\cdot\nabla)\Upsilon(\mfrk_n) - (\mfrk_n\cdot \nabla) \Upsilon(\vb) 
           - \displaystyle\sum_{j=1}^2[\Upsilon(\mfrk_n)]^j\nabla \vb^j  
         - \displaystyle\sum_{j=1}^2[\Upsilon(\vb) ]^j\nabla [\mfrk_n]^j\bigg] 
       \nonumber  \\
        \vspace{2mm} & \quad  + (\alpha_1+\alpha_2)\mathrm{div}(\Arm(\vb) \Arm(\mfrk_n)+ \Arm(\mfrk_n)\Arm(\vb))
         +\beta \mathrm{div}(\vert \Arm(\vb)\vert^{2}\Arm(\mfrk_n)) 
       \nonumber \\ & \quad + 2\beta {\rm div}((\Arm(\mfrk_n):\Arm(\vb))\Arm(\vb)), \pb_n),
	\end{align}
for a.e. $t\in[0,{T}]$ and $\Pb$-a.s. Also, setting $\wb_n=\mfrk_n$ in $\eqref{finite-dimS-adjoint-IP}_1$, we get
	\begin{align}\label{Duality2}
		-(\partial_t\Upsilon(\pb_n), \mfrk_n) 
          & =  (\gb -\nu\A\pb_n - (\Irm-\alpha_1\Delta)[(\vb\cdot\nabla)\pb_n]
        - \sum_{j=1}^2[\pb_n]^j \nabla[\Upsilon(\vb)]^j 
           - (\Irm-\alpha_1\Delta)[(\pb_n\cdot\nabla)\vb] 
        \nonumber   \\   & \quad + (\pb_n\cdot\nabla)\Upsilon(\vb)
          + (\alpha_1+\alpha_2) \mathrm{div}(\Arm(\vb) \Arm(\pb_n)+ \Arm(\pb_n)\Arm(\vb))  
         + \beta \mathrm{div}(\vert \Arm(\vb)\vert^{2}\Arm(\pb_n))  
         \nonumber\\	& \quad + 2\beta {\rm div}((\Arm(\pb_n):\Arm(\vb))\Arm(\vb)), \mfrk_n),
	\end{align}
	for a.e. $t\in[0,{T}]$ and $\Pb$-a.s. The differentiation rules give
	\begin{align}\label{Duality3}
		-(\partial_t\Upsilon(\pb_n(t)), \mfrk_n(t))  = (\partial_t\Upsilon(\mfrk_n(t)), \pb_n(t)) -\partial_t\big[(\pb_n(t),\mfrk_n(t))+\alpha_1(\nabla\pb_n(t), \nabla\mfrk_n(t)) \big].
	\end{align}
	Integrating \eqref{Duality3} with respect to $t$ over
	$[0,{T}]$, and using that $\mfrk_n(0)=\pb_n({T})=0$, we obtain
	\begin{align}\label{Duality4}
	-\int_0^{{T}}(\partial_t\Upsilon(\pb_n(t)), \mfrk_n(t))\drm t= \int_0^{{T}}(\partial_t\Upsilon(\mfrk_n(t)), \pb_n(t))\drm t,
	\end{align}
$\Pb$-a.s.	Substituting \eqref{Duality2} and \eqref{Duality1} in \eqref{Duality4}, integrating by parts and  using the fact that $(A,B)=(A^T,B^T)$, for any $A,B \in \mathcal{M}_{2\times 2}(\mathbb{R})$, we obtain
	\begin{align*}
		\Eb \left[\int_0^{T}(\boldsymbol{\psi}(t),\pb_n(t))\drm t\right] =  \Eb \left[\int_0^{T}(\gb(t),\mfrk_n(t))\drm t\right].
	\end{align*}
	Therefore, in view of \eqref{Moment-3} and \eqref{Moment-4}, taking the limit as $n\to \infty$, we complete the proof.
\end{proof}

Considering  $\gb=\nabla_{\vb}\mathfrak{L}(\cdot,\fb,\vb) \in \Lrm^p(\Omega; \mathrm{L}^{2}(0,{T};\dot{\Lb}^{2}(\mathbb{T}^2)))$ (for some $p>2$) in Proposition \ref{duality-prop}, we obtain
\begin{corollary} 
Under the assumptions of Proposition \ref{duality-prop}, the following duality relation  holds
	\begin{align*}
		\Eb \left[\int_0^{T}(\boldsymbol{\psi}(t),\pb(t))\drm t\right] = \Eb \left[ \int_0^{T}(\nabla_{\vb}\mathfrak{L}(t,\fb,\vb) ,\mfrk(t))\drm t\right] . 
	\end{align*}
\end{corollary}

\subsection{Existence of an optimal control for \texorpdfstring{\eqref{eqn-control-problem}}{}}
Suppose that $\{\fb_n,\vb_n\}_{n\in\N}$ is a minimizing sequence, remark that $\{\fb_n\}_{n\in\N}$ is uniformly bounded in the compact set $ \mathcal{F}_{ad} \subset \Lrm^p(\Omega;\mathrm{L}^{2}(0,{T};\dot{\Hb}^{1}(\mathbb{T}^2)))$, for $p>4$. According to our construction of unique solution $\vb_n$ to system \eqref{STGF-Projected} (where $\fb$ is replaced  by $\fb_n$), there exists a unique solution $\ub_n$ to system \eqref{CTGF-Projected} (where $\fb$ is replaced  by $\fb_n$) such that  (see Theorems \ref{Wellposedness-state} and \ref{solution}) 
\begin{center}
    $\{\ub_n\}_{n\in\N}$ is uniformly bounded $\mathrm{C}([0,{T}];\Drm(\A))\cap\mathrm{L}^{\infty}(0,{T};\Drm(\A^{\frac{3}{2}}))\cap \mathrm{H}^1(0,{T};\Vb)$, $\Pb$-a.s.
\end{center}
and 
\begin{center}
    $\{\ub_n\}_{n\in\N}$ is uniformly bounded in $\Lrm^q(\Omega;\Lrm^{\infty}(0,T;\Drm(\A^{\frac{3}{2}})))$, \; $q\geq2$.
\end{center}

Since $\mathcal{F}_{ad}$ is compact set, there exist $\wi\fb\in\mathcal{F}_{ad}$ such that 
\begin{align}\label{strong-cv-fn}
    \fb_n &\to  \wi \fb \quad \text{  in }  \Lrm^p(\Omega;\mathrm{L}^{2}(0,{T};\dot{\Hb}^{1}(\mathbb{T}^2))), \text{ for }p>4, 
\end{align} 
which implies
\begin{align}\label{strong-cv-fn-pathwise}
    \fb_n &\to  \wi \fb \quad \text{  in }  \mathrm{L}^{2}(0,{T};\dot{\Hb}^{1}(\mathbb{T}^2)), \Pb\text{-a.s.}, 
\end{align} 
By Banach-Alaoglu theorem, there exists 
\begin{center}
    $\wi\ub \in \Lrm^q(\Omega;\Lrm^{\infty}(0,T;\Drm(\A^{\frac{3}{2}})))$, $q\geq2$, and $\wi\ub\in \mathrm{C}([0,{T}];\Drm(\A))\cap\mathrm{L}^{\infty}(0,{T};\Drm(\A^{\frac{3}{2}}))\cap \mathrm{H}^1(0,{T};\Vb)$, $\Pb$-a.s.
\end{center}
such that up to a subsequence (denoting by the same)
\begin{equation}\label{weak-cv-yn}
\left\{
\begin{aligned}
\ub_n &\rightharpoonup \wi \ub \quad &&\text{  in } \Lrm^2(\Omega; \mathrm{L}^{2}(0,{T};\Drm(\A^{\frac{3}{2}}))), \\
		\ub_n &\overset{\ast}{\rightharpoonup} \wi \ub \quad && \text{  in }  \mathrm{L}^{\infty}(0,{T};\Drm(\A^{\frac{3}{2}})),\;\; \Pb\text{-a.s.},\\
		\ub_n &\rightharpoonup \wi \ub \quad &&\text{  in }  \mathrm{L}^{2}(0,{T};\Drm(\A^{\frac{3}{2}})), \;\; \Pb\text{-a.s.},\\
		\partial_t\ub_n &\rightharpoonup \partial_t \wi \ub \quad &&\text{  in }  \mathrm{L}^2(0,{T};\Vb),\;\; \Pb\text{-a.s.}.
\end{aligned}
\right.
\end{equation}
By \eqref{weak-cv-yn} and \cite[Lemma 3.1, p. 404]{temam2012infinite}, we  observe  that $\wi \ub\in \mathrm{C}([0,{T}],\Drm(\A))$ and therefore $\ub_n(0)$ converges to $\wi \ub(0)$ in $\Drm(\A)$, which gives  $\wi \ub(0)=\vb_0$.
Now, a similar arguments as in the proof of Theorem \ref{solution} provide that $( \wi\fb, \wi\vb:=\wi\ub+\zb)$ solves \eqref{STGF-Projected}.

Recall that  $\mathrm{J}:\mathcal{F}_{ad} \to \mathbb{R}^+$ given by \eqref{eqn-cost-functional} is convex and  continuous. From Proposition \ref{thm_z_alpha}, \eqref{weak-cv-yn} and \eqref{strong-cv-fn}, we have
\begin{align*}
	\fb_n \rightharpoonup \wi \fb \quad \text{  in } \mathrm{L}^{2}(\Omega; \mathrm{L}^{2}(0,{T};\dot{\Lb}^{2}(\mathbb{T}^2))) \text{ and }
	\vb_n \rightharpoonup \wi \vb \quad \text{  in } \mathrm{L}^{2}(\Omega; \mathrm{L}^{2}(0,{T};\Hb)).
\end{align*}
The weak lower semicontinuity of $\mathrm{J}$ ensures
$$ \mathrm{J}( \wi\fb, \wi\vb) \leq \liminf_n \mathrm{J}( \fb_n, \vb_n), $$
which gives that $( \wi\fb, \wi\vb)$ is an optimal pair.

\subsection{A necessary optimality condition for \texorpdfstring{\eqref{eqn-control-problem}}{}}

Suppose that $( \wi\fb, \wi\vb)$ is an optimal control pair. Consider $\boldsymbol{\psi} \in \mathcal{F}_{ad}$ and  define $\fb_\rho:=\wi\fb+\rho(\boldsymbol{\psi}-\wi\fb)$.  Thanks to Propositions \ref{prop-GD} and \ref{vari-cost}, we have   
\begin{align*}
	\dfrac{\mathrm{J}(\fb_\rho,\vb_\rho)-\mathrm{J}(\wi \fb, \wi \vb)}{\rho}= \Eb\left[\int_0^{{T}}\{ (\nabla_{\fb}\mathfrak{L}(t,\wi \fb(t), \wi\vb(t)),\boldsymbol{\psi}(t)-\wi{\fb}(t))+(\nabla_{\vb}\mathfrak{L}(t,\wi \fb(t), \wi\vb(t)),\mfrk(t))\}\drm t \right] +\dfrac{o(\rho)}{\rho}.
\end{align*}
Then, the  G\^ateaux derivative of the cost functional $\mathrm{J}$  is given by
\begin{align*}
	\lim_{\rho\to 0}\dfrac{\mathrm{J}(\fb_\rho,\vb_\rho)-\mathrm{J}(\wi \fb, \wi \vb)}{\rho} = \Eb\left[ \int_0^{{T}}\{ (\nabla_{\fb}\mathfrak{L}(t,\wi \fb(t), \wi\vb(t)),\boldsymbol{\psi}(t)-\wi{\fb}(t))+(\nabla_{\vb}\mathfrak{L}(t,\wi \fb(t), \wi\vb(t)),\mfrk(t))\}\drm t \right].% \geq 0.
\end{align*}
Therefore, we have
\begin{align}\label{888}
\Eb\left[\int_0^{{T}}\{ (\nabla_{\fb}\mathfrak{L}(t,\wi \fb(t), \wi\vb(t)),\boldsymbol{\psi}(t)-\wi{\fb}(t))+(\nabla_{\vb}\mathfrak{L}(t,\wi \fb(t), \wi\vb(t)),\mfrk(t))\}\drm t \right] \geq 0,
\end{align}
where $\mfrk$ is the unique solution to the linearized problem \eqref{eqn-linearized-pi} with $\boldsymbol{\psi}$ replaced by $\boldsymbol{\psi}-\wi\fb$. 
\vspace{2mm}\\
Let $\tilde \pb$ be  the unique solution of \eqref{eqn-adjoint-pi}. The application of  Proposition \ref{duality-prop} yields
\begin{align}\label{8888}
	\Eb\left[ \int_0^{T}(\boldsymbol{\psi}(t)-\wi{\fb}(t),\tilde \pb(t))dt \right] =  \Eb\left[\int_0^{T}(\nabla_{\vb}\mathfrak{L}(t,\wi{\fb}(t),\wi{\vb}(t)) ,\mfrk(t))\drm t \right] .
\end{align}
Finally, we obtain the following optimality condition,  for any   $\boldsymbol{\psi} \in \mathcal{F}_{ad}$
\begin{align}\label{equation-duality}
	& \Eb\left[ \int_0^{T}(\boldsymbol{\psi}(t)-\wi{\fb}(t),\tilde \pb(t)+\nabla_{\fb}\mathfrak{L}(t,\wi{\fb}(t),\wi{\vb}(t)))\drm t \right] \nonumber\vspace{2mm}\nonumber\\& = \Eb\left[ \int_0^{T}(\nabla_{\vb}\mathfrak{L}(t,\wi{\fb}(t),\wi{\vb}(t)) ,\mfrk(t))\drm t \right]  +  \Eb\left[ \int_0^T (\nabla_{\fb}\mathfrak{L}(t,\wi{\fb}(t),\wi{\vb}(t)),\boldsymbol{\psi}(t)-\wi{\fb}(t))\drm t \right] \geq 0,
\end{align}
where we have used \eqref{888} and \eqref{8888}. The combination of the preceding sections leads to the proof of Theorem \ref{main-thm1}.

\medskip\noindent
{\bf Acknowledgments:}   
This work is funded by national funds through the FCT – Fundação para a Ciência e a Tecnologia, I.P., under the scope of the projects UID/00297/2025 (\url{https://doi.org/10.54499/UID/00297/2025}) and UID/PRR/00297/2025 (\url{https://doi.org/10.54499/UID/PRR/00297/2025}) (Center for Mathematics and Applications)”.

\medskip\noindent
\textbf{Data availability:} No data was used for the research described in the article.

\medskip\noindent
\textbf{Declarations}: During the preparation of this work, the authors have not used AI tools.

\medskip\noindent
\textbf{Conflict of interest:} The authors declare no conflict of interest.

\medskip\noindent
\textbf{Author Contributions:} All authors contributed equally.

%All authors wrote the main manuscript text and reviewed the manuscript.


\begin{thebibliography}{99}




%\bibitem{Adler_1990}  R.J. Adler, {\it An introduction to continuity, extrema, and related topics for general Gaussian processes}, Institute of Mathematical Statistics Lecture Notes---Monograph Series, 12, Inst. Math. Statist., Hayward, CA, 1990.%; MR1088478


%\bibitem{Almeida+Chemetov+Cipriano_2022} A. Almeida, N.V. Chemetov and F. Cipriano, Uniqueness for optimal control problems of two-dimensional second grade fluids, \emph{Electron. J. Differential Equations}, 2022, Paper no. 22, 12 pp.




\bibitem{Amrouche+Cioranescu_1997}  C. Amrouche and D. Cioranescu, On a class of fluids of grade 3, \emph{Int. J. Non-Linear Mech.}, \textbf{32} (1997), 73--88.





\bibitem{JBPCK} J. 	Babutzka, and P. C. Kunstmann, 
$L^q$-Helmholtz decomposition on periodic domains and applications to Navier-Stokes equations,	\emph{J. Math. Fluid Mech.}, {\bf 20}(3) (2018), 1093–-1121. 


%\bibitem{Benyi+Oh+Zhao_2025}  \'A. B\'enyi, T. Oh and T. Zhao, Fractional Leibniz rule on the torus, \emph{Proc. Amer. Math. Soc.}, {\bf 153}(1) (2025), 207--221.



\bibitem{Breckner_1999_thesis} H.I. Breckner, Approximation and Optimal Control of the Stochastic Navier–Stokes Equation, \emph{Halle (Saale)}, 1999 (Ph.D. Thesis).


\bibitem{Busuioc_2002} A.V. Busuioc, The regularity of bidimensional solutions of the third grade fluids equations, \emph{Math. Comput. Modelling}, {\bf 35}(7--8) (2002) , 733--742. %MR1901284


\bibitem{Busuioc+Iftimie_2004}
A.V. Busuioc and  D. Iftimie,
\newblock {Global existence and uniqueness of solutions for the equations of third grade fluids,} \newblock \emph{Int. J.Non-Linear Mech.}, \textbf{39} (2004), 1–12.



\bibitem{Busuioc+Iftimie_2007} A.V. Busuioc and D. Iftimie, A non-Newtonian fluid with Navier boundary conditions, \emph{J. Dyn. Differ. Equ.}, \textbf{18} (2007), 357--379.



\bibitem{Chemetov+Cipriano_2025}  N.V. Chemetov and F. Cipriano, Optimal Control of Newtonian Fluids in a Stochastic Environment, \emph{SIAM J. Math. Anal.}, {\bf 57}(1) (2025), 364--403. % MR4848694





\bibitem{Chemetov+Cipriano_2017}  N.V. Chemetov and F. Cipriano, Optimal control for two-dimensional stochastic second grade fluids,  \emph{Stoch. Processes Appl.} {\bf 128} (2018), 2710--2749.





\bibitem{Chemin+Xu_1997} J.-Y. Chemin and C.J. Xu, Inclusions de Sobolev en calcul de Weyl-H\"ormander et champs de vecteurs sous-elliptiques, \emph{Ann. Sci. \'Ecole Norm. Sup. (4)}, {\bf 30}(6) (1997), 719--751.%; MR1476294






\bibitem{Cipriano+Didier+Guerra_2021} F. Cipriano, P. Didier and S. Guerra, Well-posedness of stochastic third grade fluid equation, \emph{J. Differ. Equ.}, \textbf{285} (2021), 496--535.




\bibitem{Cutland+Grzesiak_2007} N.J. Cutland and K. Grzesiak, Optimal control for two-dimensional stochastic Navier-Stokes equations, \emph{Appl. Math. Optim.}, \textbf{55} (2007), 61--91.

\bibitem{DaPrato+Debussche_2000} G. Da Prato and A. Debussche, Dynamic programming for the stochastic Navier-Stokes equations, \emph{Math. Model. Numer. Anal.}, \textbf{34}(2) (2000), 459--475.



 \bibitem{DaZ} \newblock G. Da Prato and J. Zabczyk,
	\newblock \emph{Stochastic Equations in Infinite Dimensions 2nd ed.},
	\newblock Cambridge University Press, 2014.


\bibitem{FR80} R.L. Fosdick and K.R. Rajagopal, {Thermodynamics and stability of fluids of third grade,} \emph{Proc. Roy. Soc. London Ser. A},  \textbf{339} (1980), 351--377.





\bibitem{Grafakos_2014} L. Grafakos, \emph{Classical Fourier Analysis}, third edition, Graduate Texts in Mathematics, Springer, New York, 2014.






%\bibitem{Hamza+Paicu_2007} M. Hamza and M. Paicu, \newblock {Global existence and uniqueness result of a class of third-grade fluids equations,} \newblock	\textit{Nonlinearity}, \textbf{20}(5)	(2007), 1095--1114.




\bibitem{HUAal20} T. Hayat, I. Ullah, A. Alsaedi and A. Bashir, Impact of temperature dependent heat source and non-linear radiative flow of third grade fluid with chemical aspects, \emph{Therm Sci.}, \textbf{24} (2 Part B) (2020), 1173--1182.





%\bibitem{Lisei_2000} H. Lisei, A minimum principle for the stochastic Navier-Stokes equation, \emph{Studi Univ. Babes-Bolyai Math.}, \textbf{45}(2) (2000), 37--65.

 %\bibitem{Lisei_2002} H. Lisei, Existence of optimal and Epsilon-optimal controls for the stochastic Navier–Stokes equation, \emph{Nonlinear Anal. Ser. A}, \textbf{51}(1) (2002), 95--118.



\bibitem{Liu+Rockner_Book_2015} W. Liu and M. R\"ockner, \emph{Stochastic Partial Differential Equations: An Introduction}, Springer, Berlin, 2015.
 
 
 \bibitem{Menaldi+Sritharan_2003} J.L. Menaldi and S.S. Sritharan, Impulse control of stochastic Navier-Stokes equations, \emph{Nonlinear Anal.}, \textbf{52}(2) (2003), 357--381.





%\bibitem{Li_2019}  D. Li, On Kato-Ponce and fractional Leibniz, \emph{Rev. Mat. Iberoam.}, \textbf{35}(1) (2019), 23--100. 





%\bibitem{Lu+Zhu_2024}    L. L\"u{} and R. Zhu, Stationary solutions to stochastic 3D Euler equations in H\"older space, \emph{Stochastic Process. Appl.}, {\bf 177} (2024), Paper No. 104465, 27 pp.%; MR4790453



\bibitem{Papoulis_1984} A. Papoulis, \emph{Probability, Random Variables, and Stochastic Processes}, Second edition, McGraw-Hill Book Co., New York, 1984.







\bibitem{PP19}  M. Parida and  S. Padhy, {Electro-osmotic flow of a third-grade fluid past a channel having stretching walls},  \emph{ Nonlinear Eng.}, \textbf{8}(1) (2019), 56--64.










   



	\bibitem{Pazy_1983}  A. Pazy,  \emph{Semigroups of Linear Operators and Applications to Partial Differential Equations}, Applied Mathematical Sciences, 44, Springer-Verlag, New York, 1983.


\bibitem{RHK18}
G.J. Reddy, A. Hiremath and M. Kumar, {Computational modeling of unsteady third-grade fluid flow	over a vertical cylinder: A study of heat transfer visualization}, \emph{Results Phys.} \textbf{8}, 671--682, 2018.




\bibitem{RE55} R.S. Rivlin and J.L. Ericksen, {Stress-deformation relations for isotropic materials}, \emph{Arch. Rational Mech. Anal.} \textbf{4} (1995),  323--425.



\bibitem{JCR}	J. C. Robinson, \emph{Infinite-Dimensional Dynamical Systems: An Introduction to Dissipative Parabolic PDEs and the Theory of Global Attractors}, Cambridge University Press, 2001. 


\bibitem{RRS}		J.C. Robinson,  J.L. Rodrigo and  W. Sadowski, \emph{The Three-Dimensional Navier–Stokes equations, Classical Theory}, Cambridge Studies in Advanced Mathematics, Cambridge University Press, Cambridge, UK, 2016.





\bibitem{Sequeira+Videman_1995} A. Sequeira and J. Videman, Global existence of classical solutions for the equations of third grade fluids, \emph{J. Math. Phys. Sci.}, \textbf{29} (1995), 47--69.


\bibitem{Sritharan_2000} S.S. Sritharan, Deterministic and stochastic control of Navier–Stokes equation with linear, monotone, and hyperviscosities, \emph{Appl. Math. Optim.}, \textbf{41}(2) (2000), 255--308.





\bibitem{Tahraoui+Cipriano_2023} Y. Tahraoui and F. Cipriano, Optimal control of two dimensional third grade fluids, \emph{J. Math. Anal. Appl.}, {\bf 523} (2023), no.~2, Paper No. 127032, 33 pp.%; MR4542559







\bibitem{Tahraoui+Cipriano_2025_ESAIM} Y. Tahraoui and F. Cipriano, Optimal control of third grade fluids with multiplicative noise, \emph{ESAIM Control Optim. Calc. Var.}, {\bf 31} (2025), Paper No. 16.










\bibitem{Tahraoui+Cipriano_2024} Y. Tahraoui and F. Cipriano, Local strong solutions to the stochastic third grade fluid equations with Navier boundary conditions, \emph{Stoch. PDE: Anal. Comp.}, \textbf{12} (2024), 1699--1744.



    \bibitem{temam2001navier} R. Temam,  \emph{Navier-Stokes Equations, Theory and Numerical Analysis}, North-Holland, Amsterdam, 1984.


	\bibitem{temam2012infinite}	R. Temam,  \emph{Infinite-Dimensional Dynamical Systems in Mechanics and Physics}, Springer, New York, 1997.



\end{thebibliography}
\end{document}